\documentclass[reqno,a4paper]{amsart}

\usepackage{amsmath, amsthm, amscd, amsfonts, amssymb, graphicx, color,mathrsfs,mathtools,enumerate}
\usepackage{lineno}
\usepackage{dsfont}
\modulolinenumbers[5]
\usepackage{bm}

\usepackage{booktabs,multirow,array}
\usepackage{comment}
\usepackage{graphicx}%
\usepackage{multirow}%
\usepackage[title]{appendix}%
\usepackage{xcolor}%
\usepackage{textcomp}%
\usepackage{manyfoot}%
\usepackage{booktabs}%
\usepackage{algorithm}%
\usepackage{algorithmicx}%
\usepackage{algpseudocode}%
\usepackage{listings}%
\usepackage{enumerate, enumitem}
\usepackage{hyperref}
\usepackage{float}

\newtheorem{thm}{Theorem}[section]
\newtheorem{lemma}[thm]{Lemma}
\newtheorem{prop}[thm]{Proposition}
\newtheorem{coro}[thm]{Corollary}

\newtheorem{rmk}[thm]{Remark}
\newtheorem{defn}[thm]{Definition}

\newcommand{\N}{\mathbb N}
\newcommand{\R}{\mathbb R}
\newcommand{\X}{\mathcal{X}}
\newcommand{\K}{\mathcal{K}}

\newcommand{\Id}{{\operatorname{I}}}
\newcommand{\norm}[1]{{\left\|#1\right\|}}
\newcommand{\scal}[2]{{\left\langle #1,#2\right\rangle}}
\newcommand{\abs}[1]{{\left|#1\right|}}

\numberwithin{equation}{section}

\allowdisplaybreaks

\title{Stochastic damped wave and Euler-Bernoulli equations with singular locally H\"older continuous drift}
\author{Davide Addona}
\address{D.Addona: Dipartimento di Scienze Matematiche, Fisiche e Informatiche, Universit\`a degli Studi di Parma, Parco Area delle Scienze, 53/A, 43124 Parma, Italy PR}
\email{\textcolor[rgb]{0.00,0.00,0.84}{davide.addona@unipr.it}}

\author{Davide Augusto Bignamini}
\address{D.A. Bignamini: Dipartimento di Scienza e Alta Tecnologia (DISAT), Universit\`a degli Studi dell'In\-su\-bria, Via Valleggio 11, 22100 Como, Italy}
\email{\textcolor[rgb]{0.00,0.00,0.84}{da.bignamini@uninsubria.it}}
\thanks{Corresponding author: Davide Augusto Bignamini}

\keywords{pathwise uniqueness; weak uniqueness; regularization by noise; H\"older continuous drift; It\^o-Tanaka trick; stochastic damped wave equation \\$ $\\ 
D.Addona: Dipartimento di Scienze Matematiche, Fisiche e Informatiche, Universit\`a degli Studi di Parma, Parco Area delle Scienze, 53/A, 43124 Parma, Italy, davide.addona@unipr.it\\
D.A. Bignamini: Dipartimento di Scienza e Alta Tecnologia (DISAT), Universit\`a degli Studi dell'In\-su\-bria, Via Valleggio 11, 22100 Como, Italy, da.bignamini@uninsubria.it}


\begin{document}
\begin{abstract}
Let $U$ and $H$ be two separable Hilbert spaces and $T>0$. We consider a stochastic differential 
equation which evolves in the Hilbert space $H$ of the form
\begin{align}
\label{SDEa}
dX(t)=AX(t)dt+B(t,X(t))dt+GdW(t), \quad t\in[0,T], \quad X(0)=x \in H,    
\end{align}
where $A:D(A)\subseteq H\to H$ is the infinitesimal generator of a strongly continuous semigroup 
$(e^{tA})_{t\geq0}$, $W=(W(t))_{t\geq0}$ is a $U$-cylindrical Wiener process defined on a normal 
filtered probability space $(\Omega,\mathcal{F},\{\mathcal{F}_t\}_{t\in [0,T]},\mathbb{P})$, 
$G:U\to H$ is a linear bounded operator and $B:[0,T]\times H\to H_\beta$ is a locally 
$\theta$-H\"older continuous function with respect to the second variable, uniformly with respect 
to the first one, for some suitable $\theta\in(0,1)$. Here, $H_\beta$ is a Hilbert space which contains $H$ with continuous embedding, so that the drift term is singular: it is not well-defined from the ambient space $H$ into itself. The coefficient $\beta\geq0$ measures the order of such a singularity and the case $\beta=0$ corresponds to a drift term with values in $H$. We prove 
that, under suitable assumptions on the coefficients, weak and pathwise uniqueness hold true for 
equation \eqref{SDEa}. In particular, the conditions assumed on the coefficients cover the 
stochastic damped wave equation in dimension $1$ and the stochastic damped Euler--Bernoulli beam 
equation up to dimension $3$, even in the hyperbolic case.
\end{abstract}
\maketitle

\tableofcontents

\section{Introduction}

The main goal of this paper is to study weak and pathwise uniqueness for a class of semilinear stochastic damped equations with singular drift, which includes the stochastic damped wave and Euler-Bernoulli equations given by
{\small
\begin{align}
\label{concrete_damped_equation}
\left\{
\begin{array}{l}
\displaystyle (\partial^2_{tt}y)(t,\xi)=(\Delta_\xi^k y)(t,\xi)-\rho((-\Delta_\xi^k)^\alpha\partial_ty)(t,\xi)+((-\Delta_\xi^k)^{\beta}f)(t,\xi,y(t,\xi),(\partial_ty)(t,\xi))+(-\Delta_\xi^k)^{-\delta} \dot{W}(t,\xi),  \\[2mm]
\displaystyle y(0,\xi)=g_0(\xi), \ (\partial_ty)(0,\xi)=g_1(\xi), \qquad \qquad \qquad  \xi\in[0,1]^d,  \\[2mm]
\displaystyle 
(\Delta_\xi^{j}y)(t,\xi)=0,
\qquad  \qquad \qquad \qquad \qquad   t\in[0,T], \  \xi\in\partial[0,1]^d, \ j\in\{0,\ldots,k\},
\end{array}
\right.
\end{align}
}
where $T>0$, $\dot{W}$ is a space-time white noise, $k=1,2$, $\Delta_\xi^k$, $k\in\N$, denotes the realization of the $k$-th power of the Laplace operator in $L^2([0,1]^d)$ with homogeneous Dirichlet boundary conditions and, for every $\gamma\in\R$, 
\[
((-\Delta_\xi^k)^{\gamma}f)(\xi)=\sum_{j=1}^{ \infty}\lambda_j^{\gamma}\scal{f}{e_j}_{L^2([0,1]^d)}e_j(\xi),
\]
for every $f\in D((-\Delta_\xi^k)^{\gamma})$, where $\{e_n\}_{n\in\N}$ is an orthonormal basis of $L^2([0,1]^d)$ consisting of eigenvectors of $\Delta_\xi^k$ and $\{\lambda_n\}_{n\in\N}$ is the sequence of corresponding eigenvalues. Further, $\alpha\in(0,1)$, $\rho>0$, $\beta,\delta\geq0$ are parameters of the problem, $g_0\in H^1_0([0,1]^d)$ if $k=1$ and $g_0, \Delta g_0\in H^1_0([0,1]^d)$ if $k=2$ (the Sobolev space of functions with zero trace on the boundary), $g_1\in L^2([0,1]^d)$ and $f$ is locally $\theta$-H\"older continuous for some $\theta\in(0,1)$.

Problems of this form have been considered: in \cite{AddMasPri23}, with an abstract method relying on systems of Forward-Backward Stochastic Differential Equations in infinite dimension, and applied to prove pathwise uniqueness for a class of semilinear stochastic Euler-Bernoulli damped equations; and in \cite{AddBig,AddBig2}, where pathwise uniqueness has been shown also for semilinear stochastic damped wave equations by means of a spectral approach. We also mention \cite{AddBig2026-1}, where weak uniqueness for quasi-linear stochastic damped equations is obtained. We stress that, unlike the case studied here, in the quoted papers only non-singular drifts have been considered, namely $\beta\leq 0$, and, even in this range of $\beta$, our results improve upon those in \cite{AddBig,AddBig2,AddMasPri23}. For a better quality of the presentation, we split the introduction into the following subsections: \emph{Main results}, \emph{State of the art}, \emph{Novelties and improvements} and \emph{Organization of the paper}.

\subsection{Main results} 

Equation \eqref{concrete_damped_equation} can be rewritten as an abstract evolution equation in a separable Hilbert space $H$
\begin{align}
\label{intro_SDE}
dX(t)=AX(t)dt+\mathcal VC(t,X(t))dt+\mathcal GdW(t), \quad t\in[0,T], \qquad X(0)=x_0\in H, 
\end{align}
where $A:D(A)\subset H\to H$ is the infinitesimal generator of a strongly continuous semigroup $\{e^{tA}\}_{t\geq0}$ of linear bounded operators on $H$, $\mathcal V\in\mathscr L(U,H_\beta)$, $\mathcal G\in\mathscr L(U,H)$, $C\in C([0,T]\times H;U)$ and $W=\{W(t)\}_{t\geq0}$ is a $U$-valued cylindrical Wiener process. Here, $U$ and $H_\beta$ are separable Hilbert spaces with $H\subset H_\beta$ with continuous embedding. In particular, we focus on operators $A$ arising from the theory of elastic systems with structural damping (see for instance \cite{chen-russel, CPT87,LT98,trig03} and Section \ref{sec:formulazione} for more details).

Under suitable assumptions, we are able to prove weak and pathwise uniqueness for the mild solution to \eqref{intro_SDE} given by, for every $t\in[0,T]$,
\begin{align}
\label{intro_mild_solution}
X(t)=e^{tA}x_0+\int_0^te^{(t-s)A}\mathcal VC(s,X(s))ds+\int_0^te^{(t-s)A}\mathcal GdW(s), \qquad \mathbb P\textup{-a.s.}
\end{align}
The main results of this paper are Theorems \ref{weak-UB} and \ref{thm:path_uniq}, where we establish sufficient conditions ensuring weak and pathwise uniqueness for \eqref{intro_mild_solution}. As an application of Theorem \ref{weak-UB}, Corollary \ref{cor:weak} provides sufficient conditions for weak uniqueness of the concrete equation \eqref{concrete_damped_equation}. Similarly, Corollary \ref{cor:path} applies Theorem \ref{thm:path_uniq} to derive pathwise uniqueness for \eqref{concrete_damped_equation}. We start by listing the main cases where weak uniqueness holds for \eqref{concrete_damped_equation}.\\ $ $\\
{\bf Stochastic damped wave equation}. Let $d=k=1$. Weak uniqueness holds for \eqref{concrete_damped_equation} in the following cases:
\begin{equation*}
\begin{cases}
&\alpha\in\left(\frac14,\frac12\right], \\[1mm]
&\theta\in (0,1), \\[1mm]
&\beta\in \left[0, \alpha-\frac14\right), \\[1mm]
&\delta\in \left(\frac14-\frac\alpha2,\frac\alpha2-\beta\right)\cap [0,\infty), 
\end{cases}\qquad \qquad 
\begin{cases}
&\alpha\in \left(\frac12,1\right), \\[1mm]
&\theta\in (0,1), \\[1mm]
&\beta\in \left[0,\frac{1-\alpha}2\right), \\[1mm]
&\delta=0.
\end{cases}
\end{equation*}
$ $\\
{\bf Stochastic Euler-Bernoulli beam equation}.
Let $k=2$. Weak uniqueness holds for \eqref{concrete_damped_equation} in the following cases:
\begin{enumerate}
\item $\bm{d=1}: \ \begin{cases}
&\alpha\in\left(\frac18,\frac12\right], \\[1mm]
&\theta\in (0,1), \\[1mm]
&\beta\in \left[0,\min\left\{\frac\alpha 2; \alpha-\frac18\right\}\right), \\[1mm]
&\delta\in \left(\frac18-\frac\alpha2,\frac\alpha2-\beta\right)\cap [0,\infty), 
\end{cases}\qquad \qquad
\begin{cases}
&\alpha\in \left(\frac12,1\right), \\[1mm]
&\theta\in (0,1), \\[1mm]
&\beta\in \left[0,\frac{1-\alpha}2\right), \\[1mm]
&\delta=0.
\end{cases}$ \\[3mm]
\item $\bm{d=2}: \ \begin{cases}
&\alpha\in\left(\frac14,\frac12\right], \\[1mm]
&\theta\in (0,1), \\[1mm]
&\beta\in \left[0, \alpha-\frac14\right), \\[1mm]
&\delta\in \left(\frac14-\frac\alpha 2,\frac\alpha 2-\beta\right)\cap [0,\infty), 
\end{cases}\qquad \qquad 
\begin{cases}
&\alpha\in \left(\frac12,1\right), \\[1mm]
&\theta\in (0,1), \\[1mm]
&\beta\in \left[0,\frac{1-\alpha}2\right), \\[1mm]
&\delta=0.
\end{cases}$ \\[3mm]
\item $\bm{d=3}: \ \begin{cases}
&\alpha\in\left(\frac38,\frac12\right], \\[1mm]
&\theta\in (0,1), \\[1mm]
&\beta\in \left[0,\alpha-\frac38\right),  \\[1mm]
&\delta\in \left(\frac38-\frac\alpha 2,\frac\alpha 2-\beta\right)\cap [0,\infty), 
\end{cases}\qquad \qquad 
\begin{cases}
&\alpha\in \left[\frac12,1\right), \\[1mm]
&\theta\in (0,1), \\[1mm]
&\beta\in \left[0,\min\left\{\frac{1-\alpha}2;\frac18\right\}\right), \\[1mm]
&\delta\in \left(\frac38-\frac\alpha 2,\frac{1-\alpha}2-\beta\right)\cap [0,\infty).
\end{cases}$
\end{enumerate}
$ $\\
In the concrete example \eqref{concrete_damped_equation} with $k=1,2$, Corollary \ref{cor:path} provides the following sufficient conditions for pathwise uniqueness of mild solutions to hold.\\ $ $\\
 {\bf Stochastic damped wave equation}. Let $k=d=1$. Pathwise uniqueness holds for 
\eqref{concrete_damped_equation} in the following cases:
\begin{align*}
&\begin{cases}
&\alpha\in\left(\frac38,\frac12\right], \\[1mm]
&\theta\in \left(\frac{3-2\alpha}{6\alpha},1\right), \\[1mm]
&\beta\in \left[0, \frac{\alpha(1+3\theta)}2-\frac34\right), \\[1mm]
&\delta\in \left(\frac14-\frac\alpha 2,\frac\alpha 2-\beta\right),
\end{cases}\quad
\begin{cases}
&\alpha\in \left(\frac12,\frac34\right), \\[1mm]
&\theta\in \left(\frac23,1\right), \\[1mm]
&\beta\in \left[0,\min\left\{\frac{1-\alpha}2;\ \alpha-\frac34
+\frac{(1-\alpha)(3\theta-1)}2\right\}\right), \\[1mm]
&\delta=0,
\end{cases}\quad
&\begin{cases}
&\alpha\in \left[\frac34,1\right), \\[1mm]
&\theta\in \left(\frac{2\alpha-1}\alpha,1\right), \\[1mm]
&\beta\in \left[0,\frac{1-\alpha}2\right), \\[1mm]
&\delta=0.
\end{cases}
\end{align*}
$ $\\
{\bf Stochastic Euler-Bernoulli beam equations}. Set $k=2$. Pathwise uniqueness holds for 
\eqref{concrete_damped_equation} in the following cases:
\begin{enumerate}
\item $\bm{d=1}$:
\begin{align*}
&\begin{cases}
&\alpha\in\left(\frac5{16},\frac12\right], \\[1mm]
&\theta\in \left(\max\left\{\frac23;\frac{5-4\alpha}{12\alpha}\right\}, 1\right), \\[1mm]
&\beta\in \left[0,\min\left\{\frac\alpha 2;\ \frac{\alpha(1+3\theta)}2-\frac58\right\}\right), \\[1mm]
&\delta=0,
\end{cases}\\
&\begin{cases}
&\alpha\in \left(\frac12,\frac{3}{4}\right), \\[1mm]
&\theta\in \left(\frac23,1\right), \\[1mm]
&\beta\in \left[0,\min\left\{\frac{1-\alpha}2;\ \alpha-\frac58
+\frac{(1-\alpha)(3\theta-1)}2\right\}\right), \\[1mm]
&\delta=0,
\end{cases}\quad
\begin{cases}
&\alpha\in \left[\frac34,1\right), \\[1mm]
&\theta\in \left(\frac{2\alpha-1}\alpha,1\right), \\[1mm]
&\beta\in \left[0,\frac{1-\alpha}2\right), \\[1mm]
&\delta=0.
\end{cases}
\end{align*}
\item $\bm{d=2}$:
\begin{align*}
&\begin{cases}
&\alpha\in\left(\frac38,\frac12\right], \\[1mm]
&\theta\in \left(\frac{3-2\alpha}{6\alpha},1\right), \\[1mm]
&\beta\in \left[0, \frac{\alpha(1+3\theta)}2-\frac34\right), \\[1mm]
&\delta\in \left(\frac14-\frac\alpha2,\frac\alpha2-\beta\right),
\end{cases}\\ 
&\begin{cases}
&\alpha\in \left(\frac12,\frac34\right), \\[1mm]
&\theta\in \left(\frac23,1\right), \\[1mm]
&\beta\in \left[0,\min\left\{\frac{1-\alpha}2;\ \alpha-\frac34
+\frac{(1-\alpha)(3\theta-1)}2\right\}\right), \\[1mm]
&\delta=0,
\end{cases}\quad
\begin{cases}
&\alpha\in \left[\frac34,1\right), \\[1mm]
&\theta\in \left(\frac{2\alpha-1}\alpha,1\right), \\[1mm]
&\beta\in \left[0,\frac{1-\alpha}2\right), \\[1mm]
&\delta=0.
\end{cases}
\end{align*}
\item $\bm{d=3}$:
\begin{align*}
&\begin{cases}
&\alpha\in\left(\frac{7}{16},\frac12\right], \\[1mm]
&\theta\in \left(\frac{7-4\alpha}{12\alpha}, 1\right), \\[1mm]
&\beta\in \left[0,\frac{\alpha(1+3\theta)}2-\frac78\right), \\[1mm]
&\delta\in \left(\frac38-\frac\alpha 2,\frac\alpha 2-\beta\right),
\end{cases} \\ 
&\begin{cases}
&\alpha\in \left[\frac12,\frac34\right), \\[1mm]
&\theta\in \left(\frac{11-12\alpha}{12(1-\alpha)},1\right), \\[1mm]
&\beta\in \left[0,\min\left\{\frac{1-\alpha}2;\ \alpha-\frac78
+\frac{(1-\alpha)(3\theta-1)}2\right\}\right), \\[1mm]
&\delta\in \left(\frac 38-\frac\alpha 2,\min\left\{\frac{1-\alpha}2-\beta;\ \frac32-2\alpha;\  
\frac{1-\alpha(2-\theta)}{2(1-\theta)}\right\}\right),
\end{cases}\quad
\begin{cases}
&\alpha\in \left(\frac34,1\right), \\[1mm]
&\theta\in \left(\frac{2\alpha-1}\alpha,1\right), \\[1mm]
&\beta\in \left[0,\min\left\{\frac{1-\alpha}2;\ \frac{1-\alpha(1-\theta)}2-\frac38\right\}\right), \\[1mm]
&\delta=0.
\end{cases}
\end{align*}
\end{enumerate}

To the best of our knowledge, both the weak and pathwise uniqueness results stated above are the first in the literature for singular equations of the form \eqref{concrete_damped_equation}, that is, when $\beta>0$. To better highlight the novelty of the present paper with respect to the existing literature, we briefly review the current state of the art on this topic.

\subsection{State of the art}
As stressed above, this paper deals with well-posedness of semilinear stochastic damped equations with a singular drift which is H\"older continuous but, in general, not Lipschitz continuous. In the deterministic counterpart, i.e., equations of the form of \eqref{concrete_damped_equation} without noise $\dot W$, this implies that, in general, the equation is ill-posed, as shown for instance in \cite{AddBig}. Hence, the presence of the noise in \eqref{concrete_damped_equation} brings a regularizing effect on the studied equation. This phenomenon is not surprising and goes back to \cite{Zv74}, where the author proved that a stochastic differential equation in dimension $1$, driven by a Brownian motion and with only bounded and measurable drift, admits a unique strong solution, provided that the coefficient of the diffusion satisfies some weak regularity conditions. Since then, this theory has been strongly developed and we recall the main results by Veretennikov \cite{V1980}, who extended the result of Zvonkin in any finite dimension, and Krylov and Rockner \cite{KR2005}, who proved existence and uniqueness for stochastic differential equations with possibly unbounded drift satisfying some integrability condition both in $t$ and $x$.

The first extension to infinite dimension is due to Gyongy and Pardoux, who showed in \cite{GP1993_1,GP1993} that a similar regularizing effect appears in quasi-linear heat equations with locally unbounded drift perturbed by a space-time white noise.
Later, in \cite{FGP2010} the authors proved that a transport equation, with bounded and H\"older continuous drift, admits strong existence and uniqueness when perturbed by a multiplicative Stratonovich noise. 

In the same year, \cite{Dap-Fla2010} provided the first pathwise regularization-by-noise result for stochastic differential equations of the form 
\begin{align}\label{heat-intro}
dX(t)+\Lambda X(t)dt=B(X(t))dt+\Lambda^{-\delta}dW(t), \ t\in(0,T], \qquad X(0)=x_0\in \mathcal{X},    
\end{align}
where $\mathcal{X}$ is a real separable Hilbert space, $\Lambda:{\rm Dom}(\Lambda)\subseteq \mathcal{X}\rightarrow\mathcal{X}$ is a strictly positive self-adjoint operator, $B:\mathcal{X}\rightarrow\mathcal{X}$ is a bounded and $\theta$-H\"older continuous function, $\{W(t)\}_{t\geq0}$ is an $\mathcal{X}$-cylindrical Wiener process and $\delta\geq 0$.
We also mention \cite{Anz-But-Ger-Sha2025}, where, by exploiting the stochastic Sewing Lemma instead of the Zvonkin transformation, pathwise uniqueness is proved for \eqref{heat-intro}, weakening the assumptions on the coefficients of \cite{Dap-Fla2010}. Still for equation \eqref{heat-intro}, a notion of pathwise uniqueness almost everywhere with respect to a Gaussian measure was studied in the case where $B$ is merely measurable and bounded (\cite{DPFPR1}) and locally bounded (\cite{DPFPR2}). The case of gradient systems in Hilbert spaces perturbed by a bounded and measurable drift term was also studied; see \cite{DPFRV}. Regarding weak uniqueness in a framework similar to that of \cite{Dap-Fla2010}, we mention \cite{Cho-Gol1995} and \cite{Kun2013}. We finally refer to \cite{ABGP2006,Z2000} for weak uniqueness results in other contexts.

As far as singular drift terms are concerned, only a few results are available on weak uniqueness and, to the best of our knowledge, only one on pathwise uniqueness. All these works deal with equations of the type \eqref{heat-intro} with $B:H\rightarrow {\rm Dom}(\Lambda^{-\beta})$ locally $\theta$-H\"older continuous and $\beta\geq 0$. In particular, at the end of \cite{D2003}, the author mentions, as a possible application of the main result of the paper, weak uniqueness for \eqref{heat-intro} with $B=\Lambda^{\frac12}F$, where $F:H\to H$ is bounded, H\"older continuous, and has sufficiently small H\"older norm. However, the seminal papers on weak uniqueness for singular SPDEs are \cite{Pri2021,Pri2021c}, where the author proves weak uniqueness for equations of the form \eqref{heat-intro} with $B:H\rightarrow {\rm Dom}(\Lambda^{-1/2})$ locally $\theta$-H\"older continuous and $\delta=0$. In the more recent paper \cite{BOS}, weak uniqueness is studied for \eqref{heat-intro} with $B:H\rightarrow {\rm Dom}(\Lambda^{-\beta})$ continuous and $\beta+\delta<1/2$.

In the latter paper, the authors exploit Zvonkin's method, i.e., replace the bad (drift) term by means of the solution to an infinite-dimensional elliptic Kolmogorov equation and exploit the regularity of the solution. However, dealing with regularity properties of solutions to Kolmogorov equations in infinite dimension is quite hard, and usually requires strong assumptions on the coefficients and involves substantial technical difficulties. 

To avoid these problems, in \cite{ABOS26} the authors study both weak and pathwise uniqueness for \eqref{heat-intro} with $B:H\rightarrow {\rm Dom}(\Lambda^{-\beta})$ locally $\theta$-H\"older continuous and $\beta\geq 0$, taking advantage of the method exploited in \cite{AddBig,AddBig2}, where suitable finite-dimensional approximations of both the Kolmogorov equation and of the solution to \eqref{heat-intro} are introduced. Thanks to this approach, it is also possible to rescale the approximating Kolmogorov equations according to the behaviour of the eigenvectors of the operator $A$. This scaling allows us to remove a condition on the growth of the H\"older norm of the drift term, which appeared in \cite{Dap-Fla2010} and strongly limited the class of admissible drift terms. Hence, in \cite{ABOS26}, sufficient conditions on the drift term, ensuring pathwise uniqueness for the solutions to the stochastic PDEs, only involve the H\"older exponent of the drift term. It is worth mentioning the recent paper \cite{Orr-Sca-Ste2026}, where weak stability by noise for approximations of doubly nonlinear evolution equations is studied.

However, due to its structure, the abstract equation \eqref{heat-intro} cannot cover equations of the type \eqref{concrete_damped_equation}. Hence, a different approach is required in order to study a class of abstract equations of the form \eqref{abstract_formulation}.

\subsection{Main novelties and improvements}

Inspired by \cite{ABOS26}, in this paper the main tool to obtain weak and pathwise uniqueness for weak mild solutions to \eqref{abstract_formulation} is an It\^o-Tanaka trick, where the singular and irregular drift term is replaced by means of solutions to a family of rescaled finite dimensional Kolmogorov equations. However, from a spectral and functional-analytic point of view, the singular equation \eqref{intro_SDE} is more complex than \eqref{heat-intro}. Indeed, we consider a class of operators $A$ that are not self-adjoint and do not admit an orthonormal basis of $H$ consisting of eigenvectors of $A$. These facts, combined with the non-injectivity of the operator $\mathcal{G}$, lead to further difficulties in the study of the associated finite-dimensional Kolmogorov equation, since we are not in the elliptic setting (as in \cite{ABOS26}) but only in the hypoelliptic one. As a consequence, the study of the singular case in this framework requires suitable techniques and sharp computations in order to obtain the precise values of the parameters involved in the main estimates. Furthermore, even in the non-singular case, namely $\beta=0$, we considerably improve the results in \cite{AddBig, AddBig2} by removing the condition on the H\"older norm of $C$ (see \cite[Formula (13)]{AddBig} and \cite[Formula (42)]{AddBig2}), which strongly restricted the class of admissible drift terms. Another remarkable improvement with respect to the quoted paper consists in the possibility of allowing $\alpha<\frac12$, i.e., the hyperbolic setting, without any additional condition; indeed, the results in \cite{AddBig} do not cover this case, while in \cite{AddBig2} we were able to consider the hyperbolic case only by assuming further regularity on the range of the drift term, i.e., $\beta<0$ ($\beta$ correspond to $-\sigma$ in the notations of \cite[Section 4.2]{AddBig2}), which is in complete contrast with the singularity of the drift, tuned by the value of $\beta>0$. As far as the method used in \cite{AddMasPri23} to study the Euler-Bernoulli equation in dimension 1 is concerned, it is essential to have the structure condition $\beta=-\delta$; consequently, this method cannot be applied in the singular case.\\
Finally, we deal with time-dependent singular drifts. This implies that, instead of elliptic Kolmogorov equations (as in \cite{ABOS26}), in the approximating procedure we are forced to consider parabolic Kolmogorov equations, with the problem of the time regularity of the solution.

Due to these improvements, we are able to treat both weak and pathwise uniqueness for semilinear stochastic damped equations with singular drift which, to the best of our knowledge, have not been considered so far.


Finally, we stress that, in this paper, our setting is slightly less general than those considered in \cite{AddBig, AddBig2,AddBig2026-1}, where the authors consider abstract operators $A$ which satisfy suitable assumptions on the decomposition of the space $H$, according to their eigenvectors, and suitable behaviour at infinity of the norm of the associated eigenvalues. Clearly, operators $A$ arising from the theory of elastic systems with structural damping satisfy such assumptions. However, two reasons motivate our choice. First, we have preferred to confine ourselves to this special class of operators since they are the best-known and most important in applications. Secondly, this choice significantly reduces and simplifies both the computations and the notation, which remain lengthy and technically involved even in this (simplified) setting. Even though we believe that the results obtained in this paper extend to the operators considered in the quoted papers, with a suitable choice of the parameters, we hope that our decision (to trade a measure of generality for a measure of clarity) will help the reader.

\subsection{Organization of the paper}
The paper is organized as follows. In Section \ref{sec:notation} we will list the notation used in the paper. In Section \ref{sec:formulazione} we will show the abstract reformulation of \eqref{concrete_damped_equation} and  we will prove the functional properties of the coefficients that will be fundamental in the next sections. In Section \ref{sec:kolmo} we will study the finite-dimensional Kolmogorov equation that we will use for the Zvonkin transformation. In Section \ref{sec:uniq} we will prove the main results of the paper and show how to apply such results to the case of stochastic semi-linear damped wave and Euler-Bernoulli beam equations. Finally, in Appendix \ref{app:mart} we provide the equivalence between martingale problem and weak mild solutions for nonautonomous operators which has been exploited to prove weak uniqueness in the first part of Section \ref{sec:uniq}.

\section{Notation}
\label{sec:notation}
Let $\K_1$ and $\K_2$ be two real separable Banach spaces equipped with the norms $\norm{\cdot}_{\K_1}$ and $\norm{\cdot}_{\K_2}$, respectively. 
 We denote by $\mathcal{B}(\K_1)$ the Borel $\sigma$-algebra associated to the norm topology in $\K_1$. 
We denote by $\Id_{\K_1}$ the identity operator on $\K_1$. Let $n\in\N$.
We denote by $\mathcal{L}^{(n)}(\K_1;\K_2)$ the set of multilinear continuous maps from $\K^n_1$ to $\K_2$. If $n=1$, then we write $\mathcal{L}(\K_1;\K_2)$ and, if $\K_1=\K_2$, then we write $\mathcal{L}^{(n)}(\K_1)$.

We denote by $B_b(\K_1;\K_2)$ the set of bounded and Borel measurable functions from $\K_1$ into $\K_2$. If $\K_2=\R$, then we simply write $B_b(\K_1)$. We denote by $C_b(\K_1;\K_2)$ the space of bounded and continuous functions from $\K_1$ into $\K_2$. If $\K_2=\R$, then we simply write $C_b(\K_1)$. We endow $C_b(\K_1;\K_2)$ with the norm
\[
\norm{f}_{\infty}=\sup_{x\in \K_1}\|f(x)\|_{\K_2}.
\]

We denote by $C_b^\theta(\K_1;\K_2)$ the subspace of all $\theta$-H\"older-continuous functions in $C_b(\K_1;\K_2)$. The space $C_b^\theta(\K_1;\K_2)$ is a Banach space if endowed with the norm
\begin{align*}
\norm{f}_{C_b^\theta(\K_1;\K_2)}:=\norm{f}_\infty+[f]_{C_b^\theta(\K_1;\K_2)},
\end{align*}
where $[\cdot]_{C_b^\theta(\mathcal{K}_1;\mathcal{K}_2)}$ denotes the standard seminorm on $C_b^\theta(\mathcal{K}_1;\mathcal{K}_2)$.
If $\K_2=\R$, then we simply write $C_b^\theta(\K_1)$.

Let $k\in\N$ and let $f:\K_1\rightarrow \K_2$ be a $k$-times Fr\'echet differentiable function. We denote by $D^k f(x)$ its Fr\'echet derivative of order $k$ at $x\in\K_1$. 
For $k\in\N$, we denote by $C_b^{k}(\K_1;\K_2)$ the space of bounded, continuous and $k$ times Fr\'echet differentiable functions $f:\K_1\rightarrow\K_2$ such that $D^if\in C_b(\K_1;\mathcal{L}^{(i)}(\K_1;\K_2))$, for $i=1,\ldots,k$. We endow $C_b^{k}(\K_1;\K_2)$ with the norm
\[
\norm{f}_{C_b^{k}(\K_1;\K_2)}:=\norm{f}_{\infty}+\sum^k_{i=1}\sup_{x\in\K_1}\|D^if(x)\|_{\mathcal{L}^{(i)}(\K_1;\K_2)}.
\]

Let $(\Omega,\mathcal{F},\mathbb{P})$ be a probability space and let $Y:(\Omega,\mathcal{F},\mathbb{P}) \rightarrow (\K_1,\mathcal{B}(\K_1))$ be a random variable. We denote by 
\[
\mathbb{E}[Y]:=\int_\Omega Y(\omega)\mathbb{P}(d\omega)=\int_{\K_1} x[\mathbb{P}\circ Y^{-1}](dx)
\]
the expectation of $Y$ with respect to $\mathbb{P}$, where the integral is meant in the Bochner sense.

Let $\X$ be a separable Hilbert space and let $\{g_k\}_{k\in\N}$ be an orthonormal basis of $\X$. We call $\X$-cylindrical Wiener process a stochastic process $\{W(t)\}_{t\geq 0}$ defined on a complete filtered probability space $(\Omega,\mathcal{F},\{\mathcal{F}_t\}_{t\geq 0},\mathbb{P})$, such that 
\[
W(t):=\sum_{k=1}^{\infty} \beta_k(t)g_k \qquad  t\geq0,
\]
where $\{\beta_n(t)\}_{t\geq 0}$, $n\in\N$, are real independent Brownian motions on $(\Omega,\mathcal{F},\{\mathcal F_t\}_{t\geq0},\mathbb{P})$.

Given two sequences $\{a_n\}_{n\in\N}$ and $\{b_n\}_{n\in\N}$ of positive real numbers, by $a_n\sim b_n$ we mean that there exist positive constants $c_1,c_2$ such that eventually $c_1a_n\leq b_n\leq c_2 a_n$.

\section{Abstract reformulation of the equation}
\label{sec:formulazione}

In this section, we introduce the abstract framework in which the SPDE \eqref{concrete_damped_equation} is studied. Moreover, we establish several estimates on the coefficients of the equation that will be essential for proving weak and pathwise uniqueness. The following assumptions define the framework in which we will work.

\label{sub:assumptions_operators}
\begin{enumerate}[start=0,label={{(H\arabic*})}]
\item \label{H0} $T>0$. $U$ is a real separable Hilbert space with scalar product $\scal{\cdot}{\cdot}_U$. We set $H=U\times U$. 
\item\label{H0.5}$\{W(t)\}_{t\geq 0}$ is a $U$-cylindrical Wiener process. 

\item \label{H1} $\Lambda: D(\Lambda)\subseteq U\rightarrow U$ is a linear self-adjoint 
operator and there exist an orthogonal basis $\{e_k\}_{k\in\N}$ of $U$ consisting of 
eigenvectors of $\Lambda$ and a non-decreasing sequence $\{\mu_k\}_{k\in\N}$ of eigenvalues, such that $\mu_1>0$, $\mu_k\to\infty$ as 
$k\to\infty$ and
\[
\Lambda e_k=\mu_k e_k, \qquad k\in\N.
\]
The normalization of the vectors $e_k$ will be fixed in \eqref{Def:normalizzazione}.

\item \label{H2} $C:[0,T]\times H\rightarrow U$ is a continuous map and there exists $\theta\in (0,1)$ such that for every $R>0$ we have
\[
\sup_{t\in [0,T]}\sup_{x,y\in B_R(0)\ :\ x\neq y}\dfrac{\norm{C(t,x)-C(t,y)}_U}{\norm{x-y}_H^\theta}< \infty,
\]
where $B_R(0)$ is the closed ball of $H$ with radius $R$ and center $0$.

\item \label{H3} $\alpha\in (0,1)$, $\rho>0$, $\beta\in [0,1/2)$ and $\delta\geq 0$ are such that $\rho\neq 2\mu_n^{{1/2}-\alpha}$, for every $n\in\N$, $\beta<\alpha$ and $\sum_{n=1}^{ \infty}\mu_n^{-\alpha-2\delta}< \infty$.
\end{enumerate}
By assumption \ref{H1}, the operator $\Lambda^\zeta:D(\Lambda^\zeta)\subseteq U\to U$ can be defined for every $\zeta\in\R$ by setting
\begin{align*}
D(\Lambda^\zeta):=\left\{u\in U:\sum_{k\in\N} \frac{\mu_k^{2\zeta}|\langle u,e_k\rangle_U|^2}{||e_k\|_U^2} <\infty\right\}, \qquad \Lambda^\zeta u=\sum_{k\in\N}\frac{\mu_k^\zeta\langle u,e_k\rangle_U}{\|e_k\|_U^2}e_k, \ u\in D(\Lambda^\zeta).  
\end{align*}
The operators $A:D(A)\subseteq H\rightarrow H$, $\mathcal{V}:{\rm Dom}(\Lambda^{\beta})\subseteq U\rightarrow H$ and $\mathcal G: U\rightarrow H$ are defined as
\begin{align}
&Ah = \left(\begin{matrix}
0 & \Lambda^{\frac12} \\
-\Lambda^{\frac12} & -\rho \Lambda^\alpha
\end{matrix}\right)h,\qquad h\in  D(A):= \left\{h=\begin{pmatrix}
h_1 \\ h_2    
\end{pmatrix}\in H:h_1\in D(\Lambda^\frac12), \ h_2\in D(\Lambda^{\alpha\vee\frac12})\right\},\label{Def:A}\\ 
&\mathcal{V}u :=  \left(\begin{matrix}0 \\ \Lambda^{\beta} u \end{matrix}\right), \phantom{aaaaaaaaaai} u\in D(\Lambda^\beta),\label{Def:V}\\
&\mathcal Gu:= \left(\begin{matrix}0 \\ \Lambda^{-\delta} u \end{matrix}\right),  \phantom{aaaaaaaaai} u\in U.\label{Def:V_G}
\end{align}
Under assumptions \ref{H0}-\ref{H3}, we will consider the following stochastic evolution equation in $H$
\begin{align}
\label{abstract_formulation}
 dX(t)=AX(t)dt+\mathcal VC(t,X(t))dt+\mathcal GdW(t), \ \  t\in[0,T], \qquad X(0)=x\in H. 
\end{align}
It is possible to rewrite equation \eqref{concrete_damped_equation} as \eqref{abstract_formulation} by choosing $U=L^2([0,1]^d)$ and $-\Lambda$ as the realization (or the square of the realization) of the Laplacian in $U$ with Dirichlet boundary conditions. \\
We note that the operator $A$ generates a strongly continuous semigroup $\{e^{tA}\}_{t\geq 0}$ on $H$, which is analytic if $\alpha\in\left[\frac12,1\right)$ and immediately differentiable (but not analytic) if $\alpha\in\left(0,\frac12\right)$ (see \cite{chen-russel}). In the present paper, we are going to study the well-posedness of \eqref{abstract_formulation} in the sense of mild solutions.

\begin{defn}\label{def:mild-sol}
Assume \ref{H0}--\ref{H3} hold. We say that $(\Omega,\mathcal F,(\mathcal F_t)_{t\geq0},\mathbb P, W,X)$ is a weak mild solution to \eqref{abstract_formulation} if $(\Omega,\mathcal F,(\mathcal F_t)_{t\geq0},\mathbb P)$ is a complete filtered probability space, $X=\{X(t)\}_{t\in[0,T]}$ is a continuous $(\mathcal F_t)_{t\in[0,T]}$-adapted $H$-valued process and $W=(W(t))_{t\geq0}$ is a $U$-valued cylindrical Wiener process such that, for every $t\in[0,T]$,
\begin{align}
\label{mild_solution_generale}
X(t)=e^{tA}x+\int_0^te^{(t-s)A}\mathcal VC(s,X(s))ds+\int_0^t e^{(t-s)A}\mathcal GdW(s), \qquad \mathbb P\textrm{-a.s.}    
\end{align}
\end{defn}
\begin{rmk}
We underline that assumption \ref{H3} guarantees that the mild formulation \eqref{mild_solution_generale} is well defined. Indeed by \cite[Proposition 5.3]{AddBig} the condition $\sum_{n=1}^{ \infty}\mu_n^{-\alpha-2\delta}<\infty$ implies that the stochastic convolution process $\{W_A(t)\}_{t\geq0}$ defined by
    \begin{align}
W_A(t)=\int_0^t e^{(t-s)A}\mathcal GdW(s), \qquad \mathbb P\textrm{-a.s.},    
\end{align}
has continuous trajectories in $H$ $\mathbb P\textrm{-a.s.}$ Moreover, the condition $\beta<\alpha$ guarantees that the process 
\[
\left\{\int_0^te^{(t-s)A}\mathcal VC(s,X(s))ds\right\}_{t\geq 0}
\]
takes values in $H$, see Proposition \ref{prop:estens_etAB} in the next subsection.
\end{rmk}

We are going to prove weak and pathwise uniqueness in the following sense.

\begin{defn}\label{uniqueness}
Assume \ref{H0}--\ref{H3} hold. 
\begin{itemize}
\item[(Weak)] We say that weak uniqueness holds for \eqref{abstract_formulation} if for every $x\in H$ and every pair $(X_1,W_1)$ and $(X_2,W_2)$ of weak mild solutions to \eqref{abstract_formulation} defined on probability spaces $(\Omega_1,\mathcal F_1,\{\mathcal F_{1,t}\}_{t\geq0},\mathbb P_1)$ and $(\Omega_2,\mathcal F_2,\{\mathcal F_{2,t}\}_{t\geq0},\mathbb P_2)$, respectively, with $X_1(0)=X_2(0)=x$, it holds that $X_1$ and $X_2$ have the same law on $C([0,T]; H)$, namely that for every measurable bounded $\psi:C([0,T]; H)\to \R$
\[
\mathbb{E}_1\left[\psi(X_1)\right] =\mathbb{E}_2\left[\psi(X_2)\right].
\]
\item[(Pathwise)] We say that pathwise $($or strong$)$ uniqueness holds for \eqref{abstract_formulation}
if for every $x\in H$ and for every pair $(X_1,W)$ and $(X_2,W)$ of weak mild solutions to \eqref{abstract_formulation}
with the same $U$-cylindrical process $W$ defined on the same probability space $(\Omega,\mathcal F,\{\mathcal F_t\}_{t\geq0},\mathbb P)$ with $X_1(0)=X_2(0)=x$, it holds that
\[
\mathbb{P}\left(X_1(t)=X_2(t), \;\;\forall\,t\in [0,T] \right)=1.
\]
\end{itemize}
\end{defn}

In the next subsection, we will prove the functional properties of $A$, $\mathcal{V}$, $\mathcal{G}$, and $C$ that will be fundamental in the next sections.

\subsection{Functional properties of the operators}
\label{subs:funct_prop}
Assume \ref{H0}, \ref{H1}, \ref{H2} and \ref{H3} hold. In this subsection, we present some preparatory results that will be used in the following sections.

Since $\rho\neq 2\mu_n^{1/2-\alpha}$, we consider the following eigenvalues of $A$
\begin{align}
\label{autovalori}
\lambda_n^{\pm}=\frac{-\rho\mu_n^{\alpha}\pm\sqrt{\rho^2\mu_n^{2\alpha}-4\mu_n}}{2}, \qquad \lambda_n^++\lambda_n^-=-\rho\mu_n^\alpha, \qquad \lambda_n^+\lambda_n^-=\mu_n, \qquad n\in\N,    
\end{align}
with corresponding eigenvectors
\begin{align}
\label{autovettoriA}
\Phi_n^+=\left(\begin{matrix}\mu_n^{\frac12} e_n \\ \lambda_n^+e_n\end{matrix}\right), \qquad \Phi_n^-=\chi_n\left(\begin{matrix}\mu_n^{\frac12} e_n \\
\lambda_n^-e_n\end{matrix}\right), \qquad n\in\N,   
\end{align}
{where we fix the following normalization
\begin{equation}\label{Def:normalizzazione}
\|e_n\|_U:=\big(\mu_n+|\lambda_n^+|^2\big)^{-\frac12},\qquad 
\chi_n:=\left(\frac{\mu_n+|\lambda_n^+|^2}{\mu_n+|\lambda_n^-|^2}\right)^{\frac12},
\qquad n\in\N,
\end{equation}
so that $\|\Phi_n^+\|_H=\|\Phi_n^-\|_H=1$ for every $n\in\N$. Moreover we note that, if $\alpha\in(0,\frac12)$ then 
$\rho^2\mu_n^{2\alpha}-4\mu_n<0$ for $n$ large, so that $\lambda_n^\pm$ and $\Phi_n^\pm$ are 
complex. Accordingly, throughout the paper $U$ and $H$ are understood to be complexified.\\
Moreover, we set
\begin{equation}\label{Def:base-ON}
\widetilde e_k:=\frac{e_k}{\|e_k\|_U},\qquad k\in\N,
\end{equation}
so that $\{\widetilde e_k\}_{k\in\N}$ is an orthonormal basis of $U$. In particular, the Hilbert--Schmidt norm of operators defined on $U$ will always be computed with respect to $\{\widetilde e_k\}_{k\in\N}$.
}
It is possible to write $H=H^+\oplus H^-$, where
\begin{align}
\label{def_H+_H-}
H^+=\overline{\{\Phi_n^+:n\in\N\}}, \qquad H^-=\overline{\{\Phi_n^-:n\in\N\}}.    
\end{align}
Hence, any $h\in H$ can be written in a unique way as $h=h^++h^-$ with $h^+\in H^+$ and $h^-\in H^-$. 
{In particular $\{\Phi_n^+\}_{n\in\N}$ and $\{\Phi_n^-\}_{n\in\N}$ are orthonormal systems in $H$, but
\[
\scal{\Phi_n^+}{\Phi_n^-}_H=\chi_n\big(\mu_n+\lambda_n^+\overline{\lambda_n^-}\big)\|e_n\|_U^2
\neq 0,\qquad n\in\N .
\]}

Further, for every $t\geq 0$ the operators $A$ and $e^{tA}$, act on $D(A)$ and $H$, respectively, as
\begin{align}
Ah= & \sum_{n=1}^\infty\left(\lambda_n^+\langle h^+,\Phi_n^+\rangle_H\Phi_n^++\lambda_n^-\langle h^-,\Phi_n^-\rangle_H\Phi_n^-\right), && h\in D(A),\label{FourierA} \\
e^{tA}h
= & \sum_{n=1}^\infty\left(e^{\lambda_n^+t}\langle h^+,\Phi_n^+\rangle_H\Phi_n^++e^{\lambda_n^-t}\langle h^-,\Phi_n^-\rangle_H\Phi_n^-\right), && h\in H.\label{FourierETA}
\end{align}

{We note that $A$ is dissipative. Indeed, for every $h=(h_1,h_2)\in D(A)$, we have 
\[
\scal{Ah}{h}_H=\scal{\Lambda^{\frac12}h_2}{h_1}_U-\scal{\Lambda^{\frac12}h_1}{h_2}_U
-\rho\scal{\Lambda^{\alpha}h_2}{h_2}_U,
\]
and, since $\Lambda^{\frac12}$ is self-adjoint, the first two terms are one the conjugate of the other, so that
\[
{\rm Re}\scal{Ah}{h}_H=-\rho\|\Lambda^{\frac\alpha2}h_2\|_U^2\leq 0 .
\]
Hence $\{e^{tA}\}_{t\geq0}$ is a contraction semigroup on $H$.}

In the following, we will need an explicit expression for $h^+,h^-$. We notice that, for every $h\in H$, there exist unique sequences $\{a_n^+\}_{n\in\N}, \{a_n^-\}_{n\in\N}$ such that
\begin{align*}
h=\sum_{n\in\N}(a_n^+\Phi_n^++a_n^-\Phi_n^-).  \qquad h\in H.  
\end{align*}
It follows that
\begin{align*}
h^+=\sum_{n\in\N}a_n^+\Phi_n^+, \qquad h^-=\sum_{n\in\N}a_n^-\Phi_n^-, \qquad h\in H.    
\end{align*}
Comparing the above expressions, we infer that $a_n^\pm=\langle h^\pm,\Phi_n^\pm\rangle_H$ for every $n\in\N$ and
\begin{align}
\label{expr_h+_h-}
h^+=\sum_{n\in\N}\langle h^+,\Phi_n^+\rangle_H\Phi_n^+, \qquad 
h^-=\sum_{n\in\N}\langle h^-,\Phi_n^-\rangle_H\Phi_n^-, \qquad h\in H.    
\end{align}
We also define the operators $(-A)^\eta$, $\eta\in\R$, as
\begin{align*}
D((-A)^\eta):= & \Big\{h\in H:\sum_{n\in\N}\big(|\lambda_n^+|^{2\eta}\langle h^+,\Phi_n^+\rangle_H^2+|\lambda_n^-|^{2\eta}\langle h^-,\Phi_n^-\rangle_H^2\big)<\infty\Big\}, \\   
(-A)^\eta h:= & \sum_{n=1}^\infty \big((-\lambda_n^+)^\eta\langle h^+,\Phi_n^+\rangle_H\Phi_n^++(-\lambda_n^-)^\eta\langle h^-,\Phi_n^-\rangle_H\Phi_n^-\big), \qquad h\in D((-A)^\eta).
\end{align*}
Clearly, $D((-A)^{\eta})=H$ if $\eta\leq0$. For every $u\in U$, the element $\left(\begin{matrix}
0 \\ u
\end{matrix}\right)$ can be written as
\begin{align*}
\left(
\begin{matrix}
0 \\ u
\end{matrix}
\right)= \sum_{n=1}^\infty\left(b_n^+ u_n\Phi_n^++b_n^- u_n\Phi_n^-\right), \qquad u_n:=\langle u,\widetilde e_n\rangle_U, \qquad n\in\N,
\end{align*}
for suitable coefficients $b_n^+,b_n^-$, $n\in\N$. It follows that, for every $u\in D(\Lambda^\beta)$, we have
\begin{align}
\label{lambda_beta_n}
\left(\begin{matrix}
0 \\ \Lambda^\beta u    
\end{matrix}\right)
= \sum_{n=1}^\infty\mu_n^\beta u_n\left(b_n^+\Phi_n^++b_n^-\Phi_n^-\right), \qquad u_n=\langle u,\widetilde e_n\rangle_U, \qquad n\in\N.
\end{align}
Finally, we recall that, from the above construction, it follows that, if $\alpha\in\left(0,\frac12\right]$, then \begin{align}
& {\rm Re}(\lambda_n^{+}), \ {\rm Re}(\lambda_n^{-})\sim -\mu_n^\alpha, \quad |\lambda_n^{+}|,|\lambda_n^{-}|\sim \mu_n^{\frac12}, \quad  \,  \|e_n\|_{U}\sim \mu_n^{-\frac12}, \quad |\lambda_n^+-\lambda_n^-|\sim \mu_n^{\frac12}, \notag \\
& |b_n^{\pm}|,\chi_n\sim 1,
\label{damped_stime_coefficienti_avl}
\end{align}
eventually with respect to $n\in\N$ (see for instance formulae \cite[(2.3.14)-(2.3.18)]{trig}). The case $\alpha\in\left[\frac12,1\right)$ is analogously treated. We stress that, in this case, the asymptotic behaviour in \eqref{damped_stime_coefficienti_avl} is replaced by
\begin{align}
& {\rm Re}(\lambda_n^+)\sim -\mu_n^{1-\alpha}, \quad {\rm Re}(\lambda_n^-)\sim -\mu_n^\alpha, \quad |\lambda_n^+|\sim \mu_n^{1-\alpha}, \quad |\lambda_n^-|\sim\mu_n^\alpha, \quad  \|e_n\|_{U}\sim \mu_n^{-\frac12}, \notag \\
& |\lambda^-_n-\lambda_n^+|\sim \mu_n^{\alpha}, \quad |b_n^-|\sim 1, \quad \chi_n,|b_n^+|\sim \mu_n^{\frac12-\alpha}
\label{damped_stime_coefficienti_avl_2}
\end{align}
definitively with respect to $n\in\N$. 

Fix $n\in\N$. For every $\zeta\in\R$, we denote by $Q_n$ the projection of $U$ onto $U_n={\rm span}\{e_1,\ldots,e_n\}\subset D(\Lambda^\zeta)$ and we set $\Lambda_n^\zeta=\Lambda^\zeta Q_n$, $\mathcal V_n=\mathcal VQ_n$ and $\mathcal G_n=\mathcal GQ_n$.
Further, we denote by $P_n$ the projection of $H$ onto $H_n:={\rm span}\{\Phi_j^+,\Phi_j^-:j=1,\ldots,n\}$ and by $A_n$ and $e^{tA_n}$ the projection of $A$ and $e^{tA}$ onto $H_n$, respectively. It follows that, for every $n\in\N$, the operator $A_n=P_nA=AP_n\in\mathscr L(H_n)$, it generates the uniformly continuous semigroup $\{e^{tA_n}\}_{t\geq 0}$ and the following representation formulae hold true:
\begin{align}
P_n h = & \sum_{j=1}^n\left(\langle h^+,\Phi_j^+\rangle_H\Phi_j^++\langle h^-,\Phi_j^-\rangle_H\Phi_j^-\right), && h\in H, \label{rep_P_nh} \\
A_nh= & \sum_{j=1}^n\left(\lambda_j^+\langle h^+,\Phi_j^+\rangle_H\Phi_j^++\lambda_j^-\langle h^-,\Phi_j^-\rangle_H\Phi_j^-\right), && h\in H, \label{rep_A_n}\\
e^{tA_n}h=e^{tA}P_nh=P_ne^{tA}h
= & \sum_{j=1}^n\left(e^{\lambda_j^+t}\langle h^+,\Phi_j^+\rangle_H\Phi_j^++e^{\lambda_j^-t}\langle h^-,\Phi_j^-\rangle_H\Phi_j^-\right), && h\in H, \ t\geq0. \label{rep_etA_n}
\end{align}
We note that $P_n$ in \eqref{rep_P_nh} is the spectral projection onto $H_n$ but it is not an orthogonal projection.
First, we need to estimate the norm of the operator $e^{tA_n}\mathcal V_n:U\to H_n$ for $t\in (0,\infty)$ and $n\in\N$, uniformly with respect to $n\in\N$.
\begin{lemma}
\label{lemma:singolarita_0_drift}
Assume \ref{H0}, \ref{H1} and \ref{H3} hold. Then, there exist positive constants $C_1,C_2$, independent of $t\in(0,\infty)$ and $n\in\N$, such that
\begin{align}
\label{stima_etA_nB_n_1-v}
\left\|e^{tA_n}\mathcal V_n\right\|_{\mathscr L(U,H_n)}\leq C_1t^{-\gamma}e^{-C_2\mu_1^{\alpha\wedge (1-\alpha)}t}, \qquad t\in(0,\infty), \qquad n\in\N.
\end{align}
where
\begin{equation}\label{gamma-unificato}
\gamma:=\begin{cases}
\dfrac{\beta}{\alpha}, & \alpha\in\left(0,\frac12\right],\\[2mm]
\max\left\{\dfrac{\beta}{\alpha};\ \dfrac{\beta+\frac12-\alpha}{1-\alpha}\right\}, 
& \alpha\in\left[\frac12,1\right),
\end{cases}
\end{equation}
We notice that, for $\alpha\in[\frac12,1)$, $\gamma=(\beta+\frac12-\alpha)(1-\alpha)^{-1}$ if and only if $\beta\geq\frac14$ and $\alpha\leq 2\beta$. 
\end{lemma}
\begin{proof}
For every $u\in U_n$, from \eqref{lambda_beta_n} and \eqref{rep_etA_n} we get
\begin{align}
\label{e^tAB_representation}
e^{tA_n}\mathcal V_nu
= & \sum_{k=1}^n\mu_k^\beta u_k\left(e^{\lambda_k^+t}b_k^+\Phi_k^++e^{\lambda_k^- t}b_k^-\Phi_k^-\right), \qquad t\in(0,T],
\end{align}
which, from \eqref{damped_stime_coefficienti_avl}, gives, for $\alpha\in\left(0,\frac12\right]$, $n\in\N$ and $t\in(0,\infty)$,
\begin{align}
\label{stima_espl_0_1}
\left\|e^{tA_n}\mathcal V_n u\right\|_{H_n}^2
\leq & C\sum_{k=1}^n\mu_k^{2\beta}(|b_k^+|^2e^{-2\rho \mu_k^\alpha t}+|b_k^-|^2e^{-2\rho \mu_k^\alpha t})u_k^2
\leq C_1 t^{-\frac{2\beta}{\alpha}}e^{-\rho\mu_1^\alpha t}\sum_{k=1}^n u_k^2,
\end{align}
for some positive constant $C_1$ independent of $t>0$ and $n\in\N$. \\
If $\alpha\in\left(\frac12,1\right)$, then from \eqref{damped_stime_coefficienti_avl_2} and \eqref{e^tAB_representation} we infer that
\begin{align*}
\left\|e^{tA_n}\mathcal V_nu\right\|_{H_n}^2
\leq & C\sum_{k=1}^n\mu_k^{2\beta}(|b_k^+|^2e^{-2\rho\mu_k^{1-\alpha}t}+|b_k^-|^2 e^{-2\rho\mu_k^\alpha t})u_k^2 \\
\leq & C\sum_{k=1}^n (\mu_k^{2\beta+1-2\alpha}e^{-2\rho \mu_k^{1-\alpha}t}+\mu_k^{2\beta}e^{-2\rho \mu_k^\alpha t})u_k^2 \\
\leq & C_1 (t^{-\left(\frac{2\beta+1-2\alpha}{1-\alpha}\vee 0\right)}+t^{-\frac{2\beta}{\alpha}})e^{-\rho\mu_1^{1-\alpha}t}\sum_{k=1}^nu_k^2, \qquad t\in(0,\infty),
\end{align*}
for some positive constant $C_1$ independent of $t\in(0,\infty)$ and $n\in\N$. To conclude, we notice that, if $\beta\leq \frac14$, then $\frac{2\beta+1-2\alpha}{1-\alpha}\geq \frac{2\beta}{\alpha}$ if and only if $\alpha\in\left(2\beta,\frac12\right)$, which implies that, on $(0,1]$, we have $t^{-\frac{2\beta+1-2\alpha}{1-\alpha}}\leq t^{-\frac{2\beta}{\alpha}}$ if  $\alpha\in\left(\frac12,1\right)$. \\
Assume now that $\beta>\frac14$. Then, $\frac{2\beta+1-2\alpha}{1-\alpha}\geq \frac{2\beta}{\alpha}$ holds true whenever $\alpha\in\left(\frac12,2\beta\right)$. 
\end{proof}

\begin{prop}
\label{prop:estens_etAB}
Assume \ref{H0}, \ref{H1} and \ref{H3} hold. For every $t>0$, the operator $e^{tA}\mathcal V$, a priori defined on $D(\Lambda^\beta)$, extends to a linear bounded operator from $U$ into $H$ which we still denote by $e^{tA}\mathcal V$. For every $u\in U$, the function $t\mapsto e^{tA}\mathcal V u$ is continuous from $(0,\infty)$ with values in $H$ and there exist positive constants $C_1,C_2$, independent of $t\in(0,\infty)$, which fulfill
\begin{align}
\label{stima_tilde S}
\|e^{tA}\mathcal V\| _{\mathscr L(U,H)}\leq \frac{C_1e^{-C_2\mu_1^{\alpha\wedge (1-\alpha)}t}}{t^{\gamma}}, \qquad t\in(0,\infty),
\end{align}
where
\begin{align*}
\gamma=
\begin{cases}
\displaystyle\frac\beta\alpha , & \alpha\in\left(0,\frac12\right], 
\\[2mm] \displaystyle \max\left\{\frac\beta\alpha;\frac{\beta+\frac12-\alpha}{1-\alpha}\right\}, & 
\alpha\in\left[\frac12,1\right).  
\end{cases}
\end{align*}
\end{prop}
\begin{proof}
We prove the statement when $\gamma=\frac \beta\alpha$, since the other case can be obtained arguing similarly. Notice that, from \eqref{stima_etA_nB_n_1-v} with $\gamma=\frac\beta\alpha$ and the fact that, for every $u\in U$, every $n,m\in\N$ with $m\leq n$ and every $t>0$, $u_n\in D(\Lambda^\beta)$ and $e^{tA}\mathcal Vu_m=e^{tA_n}\mathcal V_n u_m$, we infer that
\begin{align}
\label{cauchy_sequence_etA_nB_n}
\left\|e^{tA}\mathcal V u_m-e^{tA}\mathcal V u_n\right\|_H\leq {C_1}t^{-\frac\beta\alpha}e^{-C_2\mu_1^{\alpha\wedge (1-\alpha)}t}\|u_m-u_n\|_U  
\end{align}
for every $m,n\in\N$, where $C_1,C_2$ are the positive constants, independent of $t\in(0,\infty)$ and $n,m\in\N$, which appear in \eqref{stima_etA_nB_n_1-v}. It follows that, for every $t>0$ and $u\in U$, the sequence $\{e^{tA}\mathcal V u_n\}_{n\in\N}$ is a Cauchy sequence in $H$ and we denote by $e^{tA}\mathcal Vu$ its limit. \\ 
For every $t>0$, the operator $e^{tA}\mathcal V:U\to H$, defined in this way, is linear since it is pointwise limit of linear operators. Further, again from \eqref{stima_etA_nB_n_1-v} we deduce that $\|e^{tA}\mathcal V\|_{\mathscr L(U,H)}\leq C_1t^{-\frac\beta\alpha}e^{-C_2\mu_1^{\alpha\wedge (1-\alpha)}t}$. \\
To conclude the proof, we notice that, fixed $u\in U$, the convergence of $\left(e^{tA}\mathcal Vu_n\right)_{n\in\N}$ to $e^{tA}\mathcal Vu$ is uniform on compact subsets of $(0,\infty)$. Since the uniform limit of continuous functions is continuous, we obtain the claim. 
\end{proof}

We conclude this subsection by showing that, given a function $F\in C_b^\theta(H;H)$, its projections $H\ni h\mapsto (F(h))^+$ and $H\ni h\mapsto (F(h))^-$ on $H^+$ and $H^-$, respectively, defined in \eqref{expr_h+_h-}, preserve the same regularity.

\begin{lemma}\label{lem:Norma-H}
For every $F\in C_b^\theta(H;H)$ and $\theta\in(0,1)$, the functions
$H\ni h\mapsto (F(h))^+$ and $H\ni h\mapsto (F(h))^-$ 
belong to $C_b^\theta(H;H)$. Further, there exists a positive constant $c$ such that $[F(\cdot)^+]_{C_b^\theta(H;H)}$, $[F(\cdot)^-]_{C_b^\theta(H;H)}\leq c[F]_{C_b^\theta(H;H)}$.
\end{lemma}
\begin{proof}
We recall that, for every $h\in H$ and $n\in\N$,
\begin{align}
\label{dec_h}
h=\sum_{n\in\N}\left(\langle h^+,\Phi_n^+\rangle_H\Phi_n^++\langle h^-,\Phi_n^-\rangle_H\Phi_n^-\right), \qquad \Phi_n^+=\begin{pmatrix}
\mu_n^\frac12 e_n \\ \lambda_n^+ e_n    
\end{pmatrix}, \ \Phi_n^-=\chi_n\begin{pmatrix}
\mu_n^\frac12 e_n \\ \lambda_n^- e_n    
\end{pmatrix}.
\end{align}
If we fix $n\in\N$ and take $f_n=\chi_n\begin{pmatrix}
-\overline{\lambda_n^-} e_n \\ \mu_n^\frac12 e_n
\end{pmatrix}$, then $\|f_n\|_H=1$, $\langle f_n,\Phi_n^-\rangle_H=0$ and, from \eqref{dec_h}, we infer 
\begin{align*}
\langle h,f_n\rangle_H=\langle h^+,\Phi_n^+\rangle_H\langle \Phi_n^+,f_n\rangle_H=\langle h^+, \Phi_n^+\rangle_H\chi_n\mu_n^\frac12(\lambda_n^+-\lambda_n^-)\|e_n\|_U^2,    
\end{align*}
which gives 
\begin{align*}
\langle h^+, \Phi_n^+\rangle_H
= \frac{1}{\chi_n\mu_n^\frac12(\lambda_n^+-\lambda_n^-)\|e_n\|_U^2}\langle h,f_n\rangle_H, \qquad n\in\N.
\end{align*}
Since $\{f_n:n\in\N\}$ is an orthonormal system in $H$, 
from \eqref{expr_h+_h-} it follows that
\begin{align}
\label{stima_h+}
\|h^+\|_H^2=\sum_{n\in\N}|\langle h^+,\Phi_n^+\rangle_H|^2
\leq \sup_{n\in\N}\frac{1}{\chi_n^2\mu_n |\lambda_n^+-\lambda_n^-|^2\|e_n\|_U^4}\|h\|_H^2.
\end{align}
It follows that, for every $h,k\in H$,
\begin{align*}
\|(F(h))^+-(F(k))^+\|_H
= & \|(F(h)-F(k))^+\|_H 
\leq \sup_{n\in\N}\frac{1}{\chi_n\mu_n^\frac12|\lambda_n^+-\lambda_n^-|\|e_n\|_U^2}\|F(h)-F(k)\|_H \\
\leq & \sup_{n\in\N}\frac{1}{\chi_n\mu_n^\frac12|\lambda_n^+-\lambda_n^-|\|e_n\|_U^2}[F]_{C^\theta_b(H;H)}\|h-k\|_H^\theta.
\end{align*}
To conclude, it is enough to stress that, if $\alpha\in\left(0,\frac12\right]$, then from \eqref{damped_stime_coefficienti_avl} it follows that
\begin{align*}
\chi_n\mu_n^\frac12|\lambda_n^+-\lambda_n^-|\|e_n\|^2_U\sim \mu_n^{\frac12+\frac12-1}=1, \qquad n\to \infty,    
\end{align*}
while, if $\alpha\in\left[\frac12,1\right)$, then from \eqref{damped_stime_coefficienti_avl_2} we deduce that
\begin{align*}
\chi_n\mu_n^\frac12|\lambda_n^+-\lambda_n^-|\|e_n\|^2_U\sim \mu_n^{\frac12-\alpha+\frac12+\alpha-1}=1, \qquad n\to \infty.     
\end{align*}
Analogous computations give, for $h^-$,
\begin{align}
\label{stima_h-}
\|h^-\|_H^2=\sum_{n\in\N}|\langle h^-,\Phi_n^-\rangle_H|^2
\leq \sup_{n\in\N}\frac{1}{\chi_n^2\mu_n |\lambda_n^+-\lambda_n^-|^2\|e_n\|_U^4}\|h\|_H^2
\end{align}
and arguing as above we conclude.
\end{proof}
\begin{rmk}\label{norm-Pn}
By \eqref{rep_P_nh}, \eqref{stima_h+} and \eqref{stima_h-}, we obtain
\[
c:=\sup_{n\in\N}\|P_n\|_{\mathscr L(H)}<\infty.
\]
Indeed, for every $n\in\N$, we have $\|P_nh\|_H\leq\|h^+\|_H+\|h^-\|_H\leq c_0\|h\|_H$ for a constant $c_0$ 
independent of $n$. For the same reason $P_nh\to h$ in $H$ as $n\to\infty$, for every $h\in H$.
\end{rmk}

\begin{rmk}\label{rmk:exp-stab}
The semigroup $\{e^{tA}\}_{t\geq 0}$ is exponentially stable. Indeed, since 
$\lambda_n^+\lambda_n^-=\mu_n>0$ and $\lambda_n^++\lambda_n^-=-\rho\mu_n^\alpha<0$, from 
\eqref{autovalori} we infer that ${\rm Re}\big(\lambda_n^{(i)}\big)<0$ for every $n\in\N$ and 
$i\in\{-,+\}$, while \eqref{damped_stime_coefficienti_avl} and 
\eqref{damped_stime_coefficienti_avl_2} give 
$-{\rm Re}\big(\lambda_n^{(i)}\big)\sim\mu_n^{\alpha\wedge(1-\alpha)}\to\infty$ as $n\to\infty$. 
Therefore
\[
\eta:=\inf_{n\in\N,\ i\in\{-,+\}}\Big(-{\rm Re}\big(\lambda_n^{(i)}\big)\Big)>0 .
\]
Since $\{\Phi^+_n\}_{n\in\N}$ and $\{\Phi^-_n\}_{n\in\N}$ are orthonormal systems in $H$, from 
\eqref{FourierETA}, \eqref{expr_h+_h-}, \eqref{stima_h+} and \eqref{stima_h-} we deduce that, for 
every $h\in H$ and $t\geq 0$,
\begin{align*}
\norm{e^{tA}h}^2_H
& \leq 2\sum_{n\in\N}\big|e^{\lambda_n^+t}\big|^2\big|\scal{h^+}{\Phi_n^+}_H\big|^2
+2\sum_{n\in\N}\big|e^{\lambda_n^-t}\big|^2\big|\scal{h^-}{\Phi_n^-}_H\big|^2 \\
& \leq 2e^{-2\eta t}\left(\norm{h^+}_H^2+\norm{h^-}_H^2\right)
\leq 2c_0^2e^{-2\eta t}\norm{h}^2_H,
\end{align*}
where $c_0$ is the constant of Remark \ref{norm-Pn}. Hence
\begin{equation}\label{contrazione}
\norm{e^{tA}}_{\mathscr L(H)}\leq \sqrt2\,c_0\,e^{-\eta t},\qquad t\geq0,
\end{equation}
and, by \eqref{rep_etA_n} and Remark \ref{norm-Pn}, 
$\norm{e^{tA_n}}_{\mathscr L(H_n)}\leq \sqrt 2\,cc_0\,e^{-\eta t}$ for every $n\in\N$ and $t\geq0$.
\end{rmk}


\section{The Kolmogorov equation}\label{sec:kolmo}

In this section, we study a family of Kolmogorov equations which are essential to perform our It\^o-Tanaka trick. Throughout this section, we will assume the following stronger version of condition \ref{H2}.
\begin{enumerate}[start=3,label={{(H\arabic*b})}]
\item \label{H2b} $C:[0,\infty)\times H\rightarrow U$ is bounded and continuous, and there exists 
$\theta\in(0,1)$ such that
\[
\sup_{t\geq 0}\norm{C(t,\cdot)}_{C^\theta_b(H;U)}< \infty.
\]
\end{enumerate}

In the proofs of weak uniqueness and pathwise uniqueness, we will remove \ref{H2b} using a standard localization procedure.

\subsection{Regularity along directions}
In this subsection, we introduce a notion of regularity along suitable directions that is the key point of the approach of this paper; for further details on directional derivatives, we refer to \cite{BigFerForZan23} and the references therein.\\ 
Let $\mathcal{X}$ and $E$ be two separable Hilbert spaces with inner product $\scal{\cdot}{\cdot}_{\X}$ and $\scal{\cdot}{\cdot}_{E}$, respectively, and let $\mathcal{I}\in\mathscr{L}(E,\mathcal{X})$. By the usual identifications $\mathcal{X}=\mathcal{X}^*$ and $E=E^*$, we denote by $\mathcal{I}^*:\mathcal{X}\rightarrow E$ the unique linear bounded operator satisfying
\[
\scal{\mathcal{I}v}{h}_\mathcal{X}=\scal{v}{\mathcal{I}^*h}_E, \qquad h\in\mathcal X, \ v\in E.
\]
\begin{defn}\label{Def:V-dif}
We say that a function $f\in C_b(\mathcal{X})$ is $\mathcal{I}$-differentiable at $x\in \mathcal{X}$ if there exists $\ell_x\in E$ such that for every $v\in E$ it holds 
\begin{equation}\label{C-D_1}
\lim_{s\to 0}\abs{\frac{f(x+s\mathcal{I}v)-f(x)}{s}-\langle \ell_x,v\rangle_E}=0.
\end{equation}
We set $\nabla_{\mathcal{I}}f(x):=\ell_x$. We say that a function is $\mathcal{I}$-differentiable if it is $\mathcal{I}$-differentiable at every $x\in \mathcal{X}$. 

Let $\vartheta\in [0,1)$. We denote by $C^{1+\vartheta}_{b,\mathcal{I}}(\mathcal{X})$ the subspace of $C_b(\mathcal{X})$ of $\mathcal{I}$-differentiable functions $f$ such that $\nabla_{\mathcal{I}}f\in C^{\vartheta}_b(\mathcal{X};E)$ if $\vartheta\in (0,1)$ or $\nabla_{\mathcal{I}}f\in C_b(\mathcal{X};E)$ if $\vartheta=0$.
\end{defn}
Let $\vartheta\in [0,1)$. We note that $C^{1+\vartheta}_{b,\mathcal{I}}(\mathcal{X})$ is a Banach space if endowed with the norm
\[
\norm{f}_{C^{1+\vartheta}_{b,\mathcal{I}}(\mathcal{X})}:=\norm{f}_{\infty}+\norm{\nabla_\mathcal{I}f}_{C^\vartheta_b(\mathcal X;E)}.
\]
In the next proposition we emphasize that the relation between $\mathcal{I}$ and G\^ateaux differentiability.
\begin{prop}
\label{Prop:V-diff}
If $\varphi\in C_b(\mathcal{X})$ is G\^ateaux differentiable, then $\varphi$ is $\mathcal{I}$-differentiable and $\nabla_{\mathcal{I}}\varphi(x)=\mathcal{I}^*\nabla\varphi(x)$ for every $x\in \mathcal{X}$.
Specifically, $C^1_b(\mathcal{X})\subseteq C^1_{b,\mathcal{I}}(\mathcal{X})$.
\end{prop}
\begin{proof}
Let $\varphi:\mathcal{X}\to\R$ be a G\^ateaux differentiable function and let $x\in \mathcal{X}$ and $v\in E$. Since $\mathcal{I}v\in \mathcal{X}$, by the definition of G\^ateaux differentiability we infer
\begin{equation*}
\lim_{s\to 0}\abs{\frac{\varphi(x+s\mathcal{I}v)-\varphi(x)}{s}-\langle \mathcal{I}^*\nabla\varphi(x),v\rangle_E}=\lim_{s\to 0}\abs{\frac{\varphi(x+s\mathcal{I}v)-\varphi(x)}{s}-\langle \nabla\varphi(x),\mathcal{I}v\rangle_\mathcal{X}}=0.
\end{equation*}
This shows that $\varphi$ is $\mathcal{I}$-differentiable with $\nabla_{\mathcal{I}}\varphi=\mathcal{I}^*\nabla\varphi$. 
\end{proof}

\subsection{Finite-dimensional Ornstein-Uhlenbeck semigroups}
Assume that hypotheses \ref{H0}, \ref{H1} and \ref{H3} hold. Let $\{U_n\}_{n\in\N}$ and $\{H_n\}_{n\in\N}$ be the families of subspaces of $U$ and $H$, respectively, defined in Subsection \ref{subs:funct_prop}. For every $n\in\N$ we recall that $Q_n:U\to U_n$ is the orthogonal projection onto $U_n$, $P_n:H\rightarrow H_n$ is the spectral projection onto $H_n$ and, for every $t\geq 0$ and $n\in\N$, 
\[
A_n:=P_nA=AP_n,\quad e^{tA_n}:=P_ne^{tA}=e^{tA}P_n,\qquad \mathcal{V}_n:=P_n\mathcal{V}=\mathcal VQ_n,\qquad \mathcal G_n:=P_n\mathcal G=\mathcal GQ_n,
\]
where $A_n$ and $e^{tA_n}$ have been defined in \eqref{rep_A_n} and \eqref{rep_etA_n}, respectively.

For every $n\in\N$, we introduce the finite-dimensional Ornstein-Uhlenbeck semigroup $\{R_n(t)\}_{t\geq0}$ on $C_b(H_n)$, defined as
\begin{align}
\label{Def:semi-OU}
(R_n(t)\varphi)(x) & =\int_{H_n}\varphi(e^{tA_n}x+y)\mathcal N(0,Q_{t,n})(dy), \qquad t>0, \ x\in H_n, \ \varphi\in C_b(H_n), \notag \\
(R_n(0)\varphi)(x) & =\varphi(x),\qquad x\in H_n,
\end{align}
where, for every $n\in\N$ and $t>0$,
\[
Q_{t,n}:=\int_0^te^{sA_n}\mathcal{G}_n\mathcal{G}_n^*e^{sA_n^*}ds, \qquad 
\Gamma_{t,n}=Q_{t,n}^{-1/2}e^{tA_n}.
\]
We notice that the matrix $Q_{t,n}$ is positive definite and invertible and $\Gamma_{t,n}$ is well-defined for every $t>0$ and $n\in\N$ (see \cite[Section 5.1.2]{AddBig}).

We recall the following regularity estimates.
\begin{prop}{\cite[Theorem 6.2.2 and Proposition 6.2.9]{Dap-Zab2002}}\label{Prop:regOU}
Assume that Hypotheses \ref{H0}, \ref{H0.5} and \ref{H1} hold. Then, for every $n\in\N$, $\varphi\in B_b(H_n)$, $t>0$ and $x,h,k\in H_n$ we have
\begin{align}
&|\scal{\nabla(R_n(t)\varphi)(x)}{h}_{H_n}|\leq \|\Gamma_{t,n}h\|_{H_n}\|\varphi\|_{\infty},\label{Prop:regOU1}\\
&|D^2(R_n(t)\varphi)(x)(h,k)|\leq \sqrt{2}\|\Gamma_{t,n}h\|_{H_n}\|\Gamma_{t,n}k\|_{H_n}\|\varphi\|_{\infty}\label{Prop:regOU2}.
\end{align}
Further, for every $n\in\N$, $\varphi\in C_b^1(H_n)$, $t>0$ and $x,h,k\in H_n$ we have
\begin{align}
&|\scal{\nabla(R_n(t)\varphi)(x)}{h}_{H_n}|\leq \|e^{tA_n}h\|_{H_n}\|\varphi\|_{C^1_b({H_n})},\label{Prop:regOU3}\\
&|D^2(R_n(t)\varphi)(x)(h,k)|\leq \|e^{tA_n}k\|_{H_n}\|\Gamma_{t,n}h\|_{H_n}\|\varphi\|_{C^1_b({H_n})}.\label{Prop:regOU4}
\end{align}
\end{prop}
By standard interpolation arguments, see for instance \cite[Proposition 2.3.3]{Dap-Zab2002}, we deduce the following result.
\begin{prop}\label{Prop:Schauder}
Assume that Hypotheses \ref{H0}, \ref{H0.5} and \ref{H1} hold and let $\vartheta\in [0,1)$. For every $n\in\N$,  $\varphi\in C_b^\vartheta(H_n)$, $t>0$ and $x,h,k\in H_n$ we have
\begin{align}
&|\scal{\nabla(R_n(t)\varphi)(x)}{h}_{H_n}|\leq \|e^{tA_n}h\|_{H_n}^\vartheta \|\Gamma_{t,n}h\|^{1-\vartheta}_{H_n}\|\varphi\|_{C^\vartheta_b(H_n)},\label{Prop:Schauder1}\\
&|D^2(R_n(t)\varphi)(x)(h,k)|\leq 2^{(1-\vartheta)/2}\|e^{tA_n}k\|_{H_n}^\vartheta\|\Gamma_{t,n}k\|_{H_n}^{1-\vartheta}\|\Gamma_{t,n}h\|_{H_n}\|\varphi\|_{C^\vartheta_b(H_n)},\label{Prop:Schauder2}
\end{align}
where we set $C^0_b(H_n)=C_b(H_n)$.
\end{prop}

For every $n\in\N$ and $\vartheta\in [0,1)$, let $C^{1+\vartheta}_{b,\mathcal{V}_n}(H_n)$ be the space given by Definition \ref{Def:V-dif} with $\mathcal{I}=\mathcal{V}_n$, $E=U$ and $\mathcal X=H_n$. Recalling that $\mathcal{V}_n\in \mathcal{L}(U,H_n)$ for every $n\in\N$, we infer that, by Propositions \ref{Prop:V-diff} and \ref{Prop:regOU}, for every $\varphi\in B_b(H_n)$ and $t> 0$,
\begin{align}
&R_n(t)\varphi\in C^{1}_{b,\mathcal{V}_n}(H_n),\notag\\
&\scal{\nabla_{\mathcal{V}_n}R_n(t)\varphi}{u}_U=\scal{\mathcal{V}_n^*\nabla R_n(t)\varphi}{u}_U=\scal{\nabla R_n(t)\varphi}{\mathcal{V}_nu}_{H_n},\qquad u\in U.\label{V-diff-OU}
\end{align}

\begin{prop}\label{Prop:Vschauder}
Assume that Hypotheses \ref{H0}, \ref{H0.5} and \ref{H1} hold and let $\vartheta\in [0,1)$. For every $n\in\N$, $\varphi\in C^\vartheta_b(H_n)$ and $t>0$ we have
\begin{align}
&\norm{\mathcal{V}_n^*\nabla(R_n(t)\varphi)}_{C_b^\vartheta(H_n;U)}\leq\left(\|e^{tA_n}\mathcal{V}_n\|_{\mathcal{L}(U,H_n)}^\vartheta\|\Gamma_{t,n}\mathcal{V}_n\|_{\mathcal{L}(U,H_n)}^{1-\vartheta}+\|\Gamma_{t,n}\mathcal{V}_n\|_{\mathcal{L}(U,H_n)}\right)\|\varphi\|_{C^\vartheta_b(H_n)},\label{Prop:Vschauder1}\\
&\norm{R_n(t)\varphi}_{C^{1+\vartheta}_{b,\mathcal{V}_n}(H_n)}\leq\left(1+\|e^{tA_n}\mathcal{V}_n\|_{\mathcal{L}(U,H_n)}^\vartheta\|\Gamma_{t,n}\mathcal{V}_n\|_{\mathcal{L}(U,H_n)}^{1-\vartheta}+\|\Gamma_{t,n}\mathcal{V}_n\|_{\mathcal{L}(U,H_n)}\right)\|\varphi\|_{C^\vartheta_b(H_n)}.\label{Prop:Vschauder2}
\end{align}
\end{prop}
\begin{proof}
Let $\vartheta\in (0,1)$, $n\in\N$ and let $\varphi\in C_b^\vartheta(H_n)$. By \eqref{Prop:regOU1}, \eqref{Prop:regOU2} and \eqref{Prop:regOU4} with $h=\mathcal{V}_nu$ and $u\in U$, for every $t> 0$ and every $x\in H_n$, we have
\begin{align*}
&|\scal{\nabla(R_n(t)\varphi)(x)}{\mathcal{V}_nu}_{H_n}|\leq \|\Gamma_{t,n}\mathcal{V}_nu\|_{H_n}\|\varphi\|_{\infty},\qquad & \varphi\in C_b(H_n),\\
&|\scal{\nabla(R_n(t)\varphi)(x)}{\mathcal{V}_nu}_{H_n}|\leq \|e^{tA_n}\mathcal{V}_nu\|_{H_n}\|\varphi\|_{C^1_b({H_n})},\qquad & \varphi\in C^1_b(H_n),\\
&|D^2(R_n(t)\varphi)(x)(\mathcal{V}_nu,k)|\leq \|e^{tA_n}k\|_{H_n}\|\Gamma_{t,n}\mathcal{V}_nu\|_{H_n}\|\varphi\|_{C^1_b({H_n})},\qquad &k\in H_n,\ \varphi\in C^1_b(H_n),
\end{align*}
hence, by \eqref{V-diff-OU}, we deduce
\begin{align*}
&\norm{\mathcal{V}_n^*\nabla(R_n(t)\varphi)(x)}_{U}\leq \|\Gamma_{t,n}\mathcal{V}_n\|_{\mathcal{L}(U,H_n)}\|\varphi\|_{\infty},\qquad & \varphi\in C_b(H_n),\\
&\norm{\mathcal{V}_n^*\nabla(R_n(t)\varphi)(x)}_{U}\leq \|e^{tA_n}\mathcal{V}_n\|_{\mathcal{L}(U,H_n)}\|\varphi\|_{C^1_b(H_n)},\qquad & \varphi\in C^1_b(H_n),\\
&\norm{D\left[\mathcal{V}_n^*\nabla R_n(t)\varphi\right](x)}_{\mathcal{L}(H_n,U)}\leq \|e^{tA_n}\|_{\mathcal{L}(H_n)}\|\Gamma_{t,n}\mathcal{V}_n\|_{\mathcal{L}(U,H_n)}\|\varphi\|_{C^1_b({H_n})},\qquad & \varphi\in C^1_b(H_n).
\end{align*}
By interpolation, see for instance \cite[Proposition 2.3.3]{Dap-Zab2002}, we get the thesis. 
\end{proof}

\subsection{The integral Kolmogorov equation}
In this section, we show the well-posedness of an integral equation which represents the mild solution of a suitable finite-dimensional Kolmogorov equation.  Before studying this equation, we prove a lemma that will be fundamental for the estimates provided in this section. We stress that the crucial point is that such estimates are independent of $n\in\N$.

\begin{lemma}
\label{lemma:stima_ind_n}
Assume \ref{H0}, \ref{H1} and \ref{H3} hold. Hence, there exist constants $M_0\geq 1$ and $b>0$, independent of $t\in(0,\infty)$ and $n\in\N$, such that, for every $n\in\N$ and $t\in (0,\infty)$, we have 
\begin{align}
&\|e^{tA_n}\mathcal{V}_n\|_{\mathcal{L}(U,H_n)}\leq \frac{M_0e^{-b\zeta t}}{t^{\gamma_1}},\label{SuperStima1}\\
&\|\Gamma_{t,n}\mathcal{V}_n\|_{\mathcal{L}(U,H_n)}\leq \frac{M_0}{t^{\gamma_2}},\label{SuperStima2}\\
&\|\Gamma_{t,n}\mathcal{G}_n\|_{\mathcal{L}(U,H_n)}\leq \frac{M_0}{t^{1/2}},\label{SuperStima3}\\
&\|\Gamma_{t,n}\|_{\mathcal{L}(H_n)}\leq \frac{M_0}{t^{\gamma_3}},\label{SuperStima4}
\end{align}
where, if $\alpha\in\left(0,\frac12\right]$, then $\zeta=\mu_1^\alpha$ and 
\begin{align}
&\gamma_1= \frac{\beta}{\alpha},
\label{G1}\\
&\gamma_2=\frac12+\frac{\beta+\delta}{\alpha},\label{G2}\\
&\gamma_3=\frac12+\max\left\{\frac\delta\alpha, 1\right\}, \label{G4}
\end{align}
while, if $\alpha\in\left[\frac12,1\right)$, then $\zeta=\mu_1^{1-\alpha}$ and
\begin{align}
&\gamma_1= 
\max\left\{\frac{\beta}{\alpha},\frac{\beta+\frac12-\alpha}{1-\alpha}\right\}, 
\label{G1v}\\
&\gamma_2=\frac12+\frac{\beta+\delta}{1-\alpha},\label{G2v}\\
&\gamma_3=\frac12+\max\left\{\frac{\delta+\alpha-\frac12}{1-\alpha}, 1\right\}, \label{G4v}
\end{align}
\end{lemma}
\begin{proof}

Estimate \eqref{SuperStima1}, with constant $\gamma_1$ given by \eqref{G1} or \eqref{G1v}, follows from Lemma \ref{lemma:singolarita_0_drift}.

Estimate \eqref{SuperStima4}, with constant $\gamma_3$ given by \eqref{G4}, follows from 
\cite[Theorem 4.6]{AddBig2}, while the case \eqref{G4v} follows from 
\cite[Theorem 5.4]{AddBig}. 

To prove \eqref{SuperStima2}, it is enough to modify the proofs of \cite[Theorem 5.6]{AddBig} and of \cite[Theorem 4.6]{AddBig2}. Let us prove in detail estimate \eqref{SuperStima2} with $\gamma_2$ given by \eqref{G2}, since the case with $\gamma_2$ given by \eqref{G2v} can be obtained with analogous arguments.

In \cite[Theorem 5.6]{AddBig} it has been proved that, with $\delta=\gamma$, $\beta=0$ and fixed $T>0$,
\begin{align}
\label{stima_precedente_2}
\|\Gamma_{t,n}\mathcal V_na\|_{H_n}\leq \frac{\overline c\|a\|_U}{t^{\frac 12+\frac \gamma\alpha}}, \qquad t\in(0,T], \ a\in U,  
\end{align}
where $\overline c$ is a positive constant which depends on $T$, $\gamma$ and $\alpha$ but neither on $n\in \N$ nor $a\in U$. However, such a constant can also be chosen independent also of $T$, as we are going to show below. Estimate \eqref{stima_precedente_2} has been obtained from the estimates of the terms
\begin{align*}
K_1\psi_t(\tau)
= \phi_t(\tau)\sum_{k=1}^{\infty}\frac{\mu_k^{\gamma}a_k}{(\lambda_k^--\lambda_k^+)\|e_k\|_U}[e^{\lambda_k^+\tau}\lambda_k^--e^{\lambda_k^-\tau}\lambda_k^+]e_k, \qquad \tau\in(0,t),
\end{align*}
and
\begin{align*}
K_2\psi_t'(\tau)
= & -\sum_{k=1}^{\infty}\frac{\mu_k^{\gamma}a_k}{(\lambda_k^--\lambda_k^+)\|e_k\|_U}[\phi_t'(\tau)e^{\lambda_k^+\tau}(1-e^{(\lambda_k^--\lambda_k^+)\tau})e_k+\phi_t(\tau)(\lambda_k^+ e^{\lambda_k^+\tau}-\lambda_k^-e^{\lambda_k^-\tau})e_k]
\end{align*}
for every $\tau\in(0,t)$, where $\phi_t:[0,t]\to \mathbb R$ is defined as $\phi_t(\tau )=C_m \tau ^m(t-\tau )$ for every $\tau\in[0,t]$, $C_m$ is a normalizing constant which gives $\|\phi_t\|_{L^1([0,t])}=1$ and $m\in\N$ satisfies $-2(\delta+\beta)/\alpha+2m>-1$. By replacing $\gamma$ with $\delta$ and $a$ with $\Lambda_n^\beta a$, we infer that the above expressions for $K_1$ and $K_2$ read as
\begin{align*}
K_1\psi_t(\tau)
= & \phi_t(\tau)\sum_{k=1}^{n}\frac{\mu_k^{\delta+\beta}a_k}{(\lambda_k^--\lambda_k^+)\|e_k\|_U}[e^{\lambda_k^+\tau}\lambda_k^--e^{\lambda_k^-\tau}\lambda_k^+]e_k, \qquad \tau\in(0,t), \\
K_2\psi_t'(\tau) 
= & -\sum_{k=1}^{n}\frac{\mu_k^{\delta+\beta}a_k}{(\lambda_k^--\lambda_k^+)\|e_k\|_U}[\phi_t'(\tau)e^{\lambda_k^+\tau}(1-e^{(\lambda_k^--\lambda_k^+)\tau})e_k+\phi_t(\tau)(\lambda_k^+ e^{\lambda_k^+\tau}-\lambda_k^-e^{\lambda_k^-\tau})e_k]
\end{align*}
for every $\tau\in (0,t)$. From \eqref{damped_stime_coefficienti_avl} we thus obtain
\begin{align*}
\|K_1\psi_t(\tau)\|_U^2
\leq &  C|\phi_t(\tau)|^2\sum_{k=1}^{n}\mu_k^{2(\delta+\beta)}e^{-2\rho\mu_k^\alpha\tau}a_k^2
\leq C \tau^{-\frac{2(\delta+\beta)}{\alpha}}|\phi_t(\tau)|^2\sum_{k=1}^{n}a_k^2, \qquad \tau\in(0,t)
\end{align*}
and
\begin{align*}
\|K_2\psi_t'(\tau)\|_U^2
\leq & C\sum_{k=1}^{n}\mu_k^{2(\delta+\beta)}(|\tau\phi_t'(\tau)|^2e^{-2\rho\mu_k^\alpha\tau}+|\phi_t(\tau)|^2e^{-2\rho\mu_k^\alpha\tau})a_k^2 \leq C t^{-2m-2}\tau^{2m-\frac{2(\delta+\beta)}{\alpha}}\sum_{k=1}^na_k^2
\end{align*}
for every $\tau\in(0,t)$, where $C$ is a positive constant which depends on $\alpha$, $\beta$ and $\delta$. Arguing as in the proof of \cite[Theorem 5.6]{AddBig}, we conclude that
\begin{align*}
\|\Gamma_{t,n}\mathcal V_n a\|_{H_n}\leq \frac{C\|a\|_U}{t^{\frac12+\frac{\delta+\beta}{\alpha}}}, \qquad t>0, \ a\in U, \ n\in\N.    
\end{align*}
If $\alpha\in\left[\frac12,1\right)$, then analogous computations and \eqref{damped_stime_coefficienti_avl_2} give \eqref{G2v}.
Finally, estimate \eqref{SuperStima3} follows from \eqref{SuperStima2} since, in this case, $\beta=-\delta$. Finally, since $\{e^{tA}\}_{t\geq0}$ is exponentially stable by \eqref{contrazione}, standard 
arguments show that the constants in \eqref{SuperStima2}, \eqref{SuperStima3} and 
\eqref{SuperStima4} can be chosen independently of $T>0$.
\end{proof}

Assume that hypotheses \ref{H0}, \ref{H1}, \ref{H2b}, \ref{H3} hold. Let $n\in\N$, $k\in\{1,\ldots,n\}$, $i\in\{-,+\}$, $\theta\in (0,1)$, $T>0$ and let $g:[0,T]\times H_n\rightarrow \R$ be a continuous function such that 
\[
\sup_{t\in [0,T]}\norm{g(t,\cdot)}_{C^\theta_b(H_n)}<\infty.
\]
We consider the integral equation for every $t\in [0,T]$ and $x\in H_n$
\begin{equation}\label{Def:Unki}
u_{n,k}^{(i)}(t,x)=\int_t^{T}e^{{\mathfrak c}(s-t)\lambda_k^{(i)}}R_n(s-t)\left[\scal{\mathcal{V}_n^*\nabla_x u^{(i)}_{n,k}(s,\cdot)}{C_n(s,\cdot)}_{U}+g(s,\cdot)\right](x)ds,
\end{equation}
where $C_n(s,y)=C(s,P_ny)$ for every $y\in H$ and every $s\in[0,T]$, $\overline {\mathfrak c}$ is a positive constant sufficiently large which will be fixed later and $\lambda_k^{(i)}$ is defined in \eqref{autovalori}. We will study the well-posedness of \eqref{Def:Unki} in 
the space $C^{0,1+\theta}_{b,\mathcal{V}_n}([0,T]\times H_n)$ of the functions 
$\varphi:[0,T]\times H_n\rightarrow\R$ bounded and continuous such that $\mathcal V_n^*\nabla_x\varphi:[0,T]\times H_n\rightarrow U$ is  bounded and continuous  and
\[
\sup_{t\in [0,T]}\norm{\varphi(t,\cdot)}_{C^{1+\theta}_{b,\mathcal{V}_n}(H_n)}<\infty,
\]
endowed with the norm $\sup_{t\in[0,T]}\|\varphi(t,\cdot)\|_{C^{1+\theta}_{b,\mathcal V_n}(H_n)}$, 
which makes it a Banach space (see Definition \ref{Def:V-dif}).
In the following, for simplicity we set
\begin{equation}
r_{k,i}:=-{\rm Re}\left(\lambda_{k}^{(i)}\right),\qquad k\in\N ,\ i\in\{-,+\},\label{eq:aut-r}
\end{equation}
From \eqref{damped_stime_coefficienti_avl} and \eqref{damped_stime_coefficienti_avl_2}, for every  $k\in\N$ and $i\in\{-,+\}$ we deduce
\begin{align}
& r_{k,i}\sim \mu_k^\alpha, \quad \alpha\in\left(0,\frac12\right], \qquad r_{k,+}\sim \mu_k^{1-\alpha}, \ r_{k,-}\sim \mu_k^\alpha, \quad \alpha\in\left[\frac12, 1\right),
\label{r-asynt}
\end{align}
and, by \eqref{r-asynt}, there exists $\varepsilon\in (0,1]$ such that
\begin{align}\label{r-eps}
& r_{k,i}>\varepsilon, \qquad k\in\N, \ i\in\{-,+\}.
\end{align}

\begin{thm}\label{Thm: soluzione-kolmogorov}
Assume that hypotheses \ref{H0}, \ref{H1}, \ref{H2b}, \ref{H3} hold and 
\begin{equation}\label{eq:cond-1}
\beta+\delta<\frac12\min\{\alpha,1-\alpha\}.
\end{equation}
Then, there exists ${\mathfrak{c}}_0\geq 1$ such that for every $\mathfrak{c}>{\mathfrak{c}}_0$, $n\in\N$, $k\in\{1,\ldots,n\}$, $i\in\{-,+\}$ and continuous function $g:[0,T]\times H_n\rightarrow \R$ verifying 
\[
\sup_{t\in [0,T]}\norm{g(t,\cdot)}_{C^\theta_b(H_n)}<\infty,
\]
there exists a unique solution $u_{n,k}^{(i)}\in C^{0,1+\theta}_{b,\mathcal{V}_n}([0,T]\times H_n)$ to \eqref{Def:Unki} and
\begin{align}
&\sup_{t\in[0,T]}\norm{u_{n,k}^{(i)}(t,\cdot)}_{C^{1+\theta}_{b,\mathcal{V}_n}(H_n)}\leq  \mathfrak{C}_0\sup_{t\in[0,T]}\norm{g(t,\cdot)}_{C^{\theta}_b(H_n)},\label{StimaC0}\\
&\mathfrak{C}_0=\frac{4(2+\varepsilon^{-\gamma_2})M_0}{(\varepsilon\mathfrak c)^{1-\gamma_2}(1-\gamma_2)}\left(1+\sup_{t\geq 0}\|C(t,\cdot)\|_{C^\theta_b(H;U)}\right),
\end{align}
where $M_0\geq 1$ and $0<\gamma_2<1$ are the constants given by Lemma \ref{lemma:stima_ind_n} and $\varepsilon\in (0,1]$ is given by \eqref{r-eps}.
\end{thm}
\begin{rmk}
\label{rmk:conseguenze_varie}
We stress that, under assumption \ref{H2b}, we can consider an arbitrarily fixed $T>0$. Further, condition \eqref{eq:cond-1} ensures that $\gamma_2<1$. \\
Finally, we notice that, from \eqref{eq:cond-1}, it follows that $\beta<\frac14$. Since if $\alpha\in\left[\frac12,1\right)$ it is necessary that $\beta\geq \frac14$ to have $\gamma_1=\frac{\beta+\frac12-\alpha}{1-\alpha}$ $($see the statement of Lemma \ref{lemma:singolarita_0_drift}$)$, it follows that $\gamma_1=\frac\beta\alpha$ for every $\alpha$ and $\beta$ satisfying \eqref{eq:cond-1}.
\end{rmk}
\begin{proof}
First of all we note that $\max\{\gamma_1\theta+\gamma_2(1-\theta),\gamma_2\}=\gamma_2$, where $\gamma_1$ and $\gamma_2$ are given by \eqref{G1} (or \eqref{G1v}) and \eqref{G2} (or \eqref{G2v}), respectively. By \eqref{Prop:Vschauder2}, \eqref{SuperStima1} and \eqref{SuperStima2}, for every $T>0$, $n\in\N$ and $t\in (0,T]$ we have 
\begin{align}\label{stima-t-beta}
\norm{R_n(t)\varphi}_{C^{1+\theta}_{b,\mathcal{V}_n}(H_n)}\leq 2M_0\left(1+\frac{1}{t^{\gamma_2}}\right)\norm{\varphi}_{C^\theta_b(H_n)},\qquad \varphi\in C^\theta_b(H_n).
\end{align}
Fix $n\in\N$, $k\in\{1,\ldots,n\}$, $i\in\{-,+\}$ and $g$ as in the statement, and let $V$ be the Volterra operator defined, for every $u\in C^{0,1+\theta}_{b,\mathcal{V}_n}([0,T]\times H_n)$, by
\[
[V(u)](t,x):=\int_t^{T}e^{\mathfrak c(s-t)\lambda_k^{(i)}}R_n(s-t)\left[\scal{\mathcal{V}_n^*\nabla_x u(s,\cdot)}{C_n(s,\cdot)}_{U}+g(s,\cdot)\right](x)ds,\quad t\in [0,T],\ x\in H_n.
\]
where $\mathfrak c>1$ will be chosen later. By \eqref{Prop:Vschauder1}, \eqref{r-eps} and \eqref{stima-t-beta}, for every $u\in C^{0,1+\theta}_{b,\mathcal V_n}([0,T]\times H_n)$ we have
\begin{align*}
& \norm{[V(u)](t,\cdot)}_{C^{1+\theta}_{b,\mathcal{V}_n}(H_n)} \\
\leq & \int_t^{T}e^{-\mathfrak c(s-t)r_{k,i}}\norm{R_n(s-t)\left[\scal{\mathcal{V}_n^*\nabla_x u(s,\cdot)}{C_n(s,\cdot)}_{U}+g(s,\cdot)\right]}_{C^{1+\theta}_{b,\mathcal{V}_n}(H_n)}ds\\
\leq & 
2M_0\int_t^{T}e^{-\mathfrak c(s-t)r_{k,i}}\left(1+\frac{1}{(s-t)^{\gamma_2}}\right)\left[\norm{\mathcal{V}_n^*\nabla_x u(s,\cdot)}_{C^{\theta}_b(H_n;U)}  \norm{C_n(s,\cdot)}_{C^{\theta}_b(H_n;U)}+\norm{g(s,\cdot)}_{C^{\theta}_b(H_n)}\right]ds\\
\leq & 
2M_0\int_0^{(T-t)r_{k,i}}e^{-\mathfrak c\xi}(r_{k,i}^{-1}+r_{k,i}^{\gamma_2-1}\xi^{-\gamma_2})d\xi\cdot \sup_{s\in[0,T]}\left[\norm{ u(s,\cdot)}_{C_{b,\mathcal V_n}^{1+\theta}(H_n)} \norm{C_n(s,\cdot)}_{C^{\theta}_b(H_n;U)}+\norm{g(s,\cdot)}_{C^{\theta}_b(H_n)}\right] \\
\leq & 2M_0\varepsilon^{\gamma_2-1}\left(\int_0^{\infty}e^{-\mathfrak c \xi}(\varepsilon^{-\gamma_2}+\xi^{-\gamma_2})d\xi\right)\sup_{s\in[0,T]}\left[\norm{ u(s,\cdot)}_{C_{b,\mathcal V_n}^{1+\theta}(H_n)}  \norm{C_n(s,\cdot)}_{C^{\theta}_b(H_n;U)}+\norm{g(s,\cdot)}_{C^{\theta}_b(H_n)}\right].
\end{align*}
By  \eqref{eq:cond-1} $\gamma_2<1$ and by \ref{H2b} we have
\begin{equation}\label{stima-normH}
\sup_{s\in[0,T]}\norm{C_n(s,\cdot)}_{C^{\theta}_b(H_n;U)}=\sup_{s\in[0,T]}\norm{C(s,P_n(\cdot))}_{C^{\theta}_b(H;U)}\leq\sup_{s\geq 0}\norm{C(s,\cdot)}_{C^{\theta}_b(H;U)}<\infty.
\end{equation}
Therefore, for every $u\in C^{0,1+\theta}_{b,\mathcal V_n}([0,T]\times H_n)$  we have
\begin{align}
&\sup_{t\in[0,T]}\norm{[V(u)](t,\cdot)}_{C^{1+\theta}_{b,\mathcal{V}_n}(H_n)}\leq  \frac{\mathfrak{C}_0}{2} \left[\sup_{t\in[0,T]}\norm{u(t,\cdot)}_{C^{1+\theta}_{b,\mathcal{V}_n}(H_n)}+\sup_{t\in[0,T]}\norm{g(t,\cdot)}_{C^{\theta}_b(H_n)}\right],\quad T>0,\label{StimaV}\\
&\mathfrak{C}_0=4M_0(\varepsilon \mathfrak c)^{\gamma_2-1}\left(1+\sup_{t\geq 0}\|C(t,\cdot)\|_{C^\theta_b(H;U)}\right)\int_0^{\infty}e^{-\xi}(\varepsilon^{-\gamma_2}+\xi^{-\gamma_2}) d\xi\to 0, \qquad \mathfrak c\to \infty.\label{LimT}
\end{align}
Hence, taking also into account hypothesis \ref{H2b}, we have $V(u)\in C^{0,1+\theta}_{b,\mathcal{V}_n}([0,T]\times H_n)$.
Now we prove that $V$ is a contraction on $C^{0,1+\theta}_{b,\mathcal{V}_n}([0,T]\times H_n)$ for $\mathfrak c$ big enough. For every $u,v\in C^{0,1+\theta}_{b,\mathcal V_n}([0,T]\times H_n)$, by \eqref{Prop:Vschauder1}, \eqref{stima-t-beta} and calculations analogous to those used to prove \eqref{StimaV}, we have
\begin{align}\label{StimaVcont}
\sup_{t\in[0,T]}\norm{[V(u)](t,\cdot)-[V(v)](t,\cdot)}_{C^{1+\theta}_{b,\mathcal{V}_n}(H_n)}&\leq \frac{\mathfrak{C}_0}{2}  \left[\sup_{t\in[0,T]}\norm{u(t,\cdot)-v(t,\cdot)}_{C^{1+\theta}_{b,\mathcal{V}_n}(H_n)}\right]
\end{align}
and so, by \eqref{LimT}, there exists ${\mathfrak c}_0\geq 1$ such that $\mathfrak{C}_0\leq1$ for every $\mathfrak c> {\mathfrak c}_0$, that is,
\begin{align*}
\sup_{t\in[0,T]}\norm{[V(u)](t,\cdot)-[V(v)](t,\cdot)}_{C^{1+\theta}_{b,\mathcal{V}_n}(H_n)}&\leq  \frac12\sup_{t\in[0,T]}\norm{u(t,\cdot)-v(t,\cdot)}_{C^{1+\theta}_{b,\mathcal{V}_n}(H_n)},  
\end{align*}
i.e., $V$ is a contraction on $C^{0,1+\theta}_{b,\mathcal{V}_n}([0,T]\times H_n)$ and we denote by $u_{n,k}^{(i)}$ the unique fixed point of $V$. Finally, \eqref{StimaC0} follows from \eqref{StimaV} and the choice of ${\mathfrak c}_0$.
\end{proof}

We now prove uniform estimates that are crucial in the proofs of both weak and pathwise uniqueness.

\begin{thm}\label{Thm:stime-unif}
Assume that conditions \ref{H0}, \ref{H1}, \ref{H2b}, \ref{H3} and \eqref{eq:cond-1} hold. For every $n\in\N$, $k\in\{1,\ldots,n\}$, $i\in\{-,+\}$ and every continuous function $g:[0,T]\times H_n\rightarrow \R$ satisfying 
\[
\sup_{t\in [0,T]}\norm{g(t,\cdot)}_{C^\theta_b(H_n)}<\infty,
\]
let $\mathfrak c>{\mathfrak c}_0$ and $u_{n,k}^{(i)}$ be the unique solution to \eqref{Def:Unki} given by Theorem \ref{Thm: soluzione-kolmogorov}.\\
Then, there exists ${\mathfrak{C}}_1=\mathfrak C_1({\mathfrak c})$ independent of  $T>0$, $n\in\N$, $k\in\{1,\ldots,n\}$, $i\in\{-,+\}$ and $g$ such that
\begin{align}
&\sup_{t\in [0,T]}\norm{\mathcal{V}_n^*\nabla_x u_{n,k}^{(i)}(t,\cdot)}_{C_b(H_n;U)}\leq \mathfrak{C}_1r_{k,i}^{\gamma_1\theta+\gamma_2(1-\theta)-1}\sup_{t\in [0,T]}\norm{g(t,\cdot)}_{C^\theta_b(H_n)},\label{Thm:stime-unif2}\\
&\lim_{{\mathfrak c}\rightarrow \infty}\mathfrak{C}_1=0,\label{Lim1}
\end{align}
where $r_{k,i}$ is defined in \eqref{eq:aut-r}, $\gamma_1$ and $\gamma_2$ are given by \eqref{G1} $($or \eqref{G1v}$)$ and \eqref{G2} $($or \eqref{G2v}$)$, respectively.\\ 
If, in addition, we assume that
\begin{equation}\label{eq:cond-2}
\theta>\max\left\{\frac23,\frac{2\delta+(2\alpha-1)\vee0}{2\delta+\alpha}\right\},
\end{equation}
then there exist $\mathfrak C_2=\mathfrak{C}_2({\mathfrak c})$ and $\mathfrak C_3=\mathfrak{C}_3({\mathfrak c})$  independent of  $T>0$, $n\in\N$, $k\in\{1,\ldots,n\}$, $i\in\{-,+\}$ and $g$ such that
\begin{align}
&\sup_{t\in [0,T]}\norm{\nabla_x u_{n,k}^{(i)}(t,\cdot)}_{C_b(H_n;H_n)}\leq \mathfrak{C}_2r_{k,i}^{\gamma_3(1-\theta)-1}\sup_{t\in [0,T]}\norm{g(t,\cdot)}_{C^\theta_b(H_n)}, \label{Thm:stime-unif1}\\
&\sup_{t\in [0,T]}\norm{D_x\left[\mathcal{G}_n^*\nabla_x u_{n,k}^{(i)}(t,\cdot)\right]}_{C_b(H_n;\mathcal{L}(H_n;U))}\leq \mathfrak{C}_3r_{k,i}^{\gamma_3(1-\theta)-1/2}\sup_{t\in [0,T]}\norm{g(t,\cdot)}_{C^\theta_b(H_n)}
\label{Thm:stime-unif3},\\
&\lim_{{\mathfrak c}\rightarrow \infty}\mathfrak{C}_2=\lim_{{\mathfrak c}\rightarrow\infty}\mathfrak{C}_3=0,\label{Lim23}
\end{align}
where $\gamma_3$ is given by \eqref{G4} $($or \eqref{G4v}$)$.
\end{thm}
\begin{rmk}
\label{rmk:condizioni_gamma}
Condition \eqref{eq:cond-2} ensures that $\gamma_1\theta+\gamma_2(1-\theta)< 1$ and $\gamma_3(1-\theta)+\frac12< 1$.    
\end{rmk}
\begin{proof}
Fix $n\in\N$, $k\in\{1,\ldots,n\}$, $i\in\{-,+\}$ and $g$ as in the statement. Let $\mathfrak c>{\mathfrak c}_0$ and let $u_{n,k}^{(i)}$ be the unique solution to \eqref{Def:Unki} given by Theorem \ref{Thm: soluzione-kolmogorov}. \\

By \eqref{Prop:Schauder1}, with $\vartheta=\theta$, $h=\mathcal{V}_nu$ and $u\in U$ we have
\[
\norm{\mathcal{V}_n^*\nabla R_n(t)\varphi}_{C_b(H_n;U)}\leq \norm{e^{tA_n}\mathcal{V}_n}^\theta_{\mathcal{L}(U,H_n)}\norm{\Gamma_{t,n}\mathcal{V}_n}_{\mathcal{L}(U,H_n)}^{1-\theta}\norm{\varphi}_{C^\theta_b(H_n)},\qquad t\in (0,T],\ \varphi\in C^\theta_b(H_n),
\]
and, from \eqref{SuperStima1} and \eqref{SuperStima2}, we deduce
\begin{equation}\label{C1V}
\norm{\mathcal{V}_n^*\nabla R_n(t)\varphi}_{C_b(H_n;U)}\leq \frac{M_0}{t^{\gamma_1\theta+\gamma_2(1-\theta)}}\norm{\varphi}_{C^\theta_b(H_n)},\qquad t\in (0,T],\ \varphi\in C^\theta_b(H_n),
\end{equation}
where $\gamma_1$ and $\gamma_2$ are given by \eqref{G1} (or \eqref{G1v}) and \eqref{G2} (or \eqref{G2v}), respectively.  By \eqref{StimaC0} and \eqref{C1V}, for every $t\in [0,T]$ we have 
\begin{align*}
&\norm{\mathcal{V}_n^*\nabla_x u_{n,k}^{(i)}(t,\cdot)}_{C_b(H_n;U)} \\
\leq & M_0\int_t^{T}e^{-{\mathfrak c}(s-t)r_{k,i}}\frac{1}{(s-t)^{\gamma_1\theta+\gamma_2(1-\theta)}}\norm{g(s,\cdot)}_{C^\theta_b(H_n)}\left[1+\mathfrak{C}_0\norm{C_n(s,\cdot)}_{C_b^\theta(H_n;U)}\right]ds\\
\leq & \frac{M_0}{r_{k,i}^{1-\gamma_1\theta-\gamma_2(1-\theta)}}\sup_{t\in [0,T]}\left[\norm{g(t,\cdot)}_{C^\theta_b(H_n)}\left(1+\mathfrak{C}_0\norm{C_n(t,\cdot)}_{C_b^\theta(H_n;U)}\right)\right]\int_0^{(T-t)r_{k,i}}e^{-{\mathfrak c}\xi}\xi^{-\gamma_1\theta-\gamma_2(1-\theta)}d\xi \\
\leq & \frac{M_0\left(1+\mathfrak{C}_0\sup_{t\in [0,T]}\norm{C_n(t,\cdot)}_{C_b^\theta(H_n;U)}\right)}{r_{k,i}^{1-\gamma_1\theta-\gamma_2(1-\theta)}}\sup_{t\in [0,T]}\norm{g(t,\cdot)}_{C^\theta_b(H_n)}\int_0^{\infty}e^{-{\mathfrak c}\xi}\xi^{-\gamma_1\theta-\gamma_2(1-\theta)}d\xi.
\end{align*}
Hence, noting that \eqref{eq:cond-1} implies $\gamma_1\theta+\gamma_2(1-\theta)<1$ and recalling \eqref{stima-normH}, we obtain \eqref{Thm:stime-unif2} with \begin{align*}
\mathfrak C_1=\frac{2M_0\left(1+\mathfrak{C}_0\sup_{t\geq 0}\norm{C(t,\cdot)}_{C_b^\theta(H;U)}\right)}{{\mathfrak{c}}^{1-\gamma_1\theta-\gamma_2(1-\theta)}(1-\gamma_1\theta-\gamma_2(1-\theta))}. 
\end{align*}

Now we prove \eqref{Thm:stime-unif1}. By \eqref{Prop:Schauder1} with $\vartheta=\theta$ we have
\[
\norm{\nabla R_n(t)\varphi}_{C_b(H_n;H_n)}\leq \norm{e^{tA_n}}^\theta_{\mathcal{L}(H_n)}\norm{\Gamma_{t,n}}_{\mathcal{L}(H_n)}^{1-\theta}\norm{\varphi}_{C^\theta_b(H_n)},\qquad t\in (0,T],\ \varphi\in C^\theta_b(H_n).
\]
Recalling that $\norm{e^{tA_n}}^\theta_{\mathcal{L}(H_n)}\leq \norm{e^{tA}}^\theta_{\mathcal{L}(H)}$, by \eqref{SuperStima4} we infer that
\begin{equation}\label{C2}
\norm{\nabla R_n(t)\varphi}_{C_b(H_n;H_n)}\leq \frac{M_0}{t^{(1-\theta)\gamma_3}}\norm{\varphi}_{C^\theta_b(H_n)},\qquad t\in (0,T],\ \varphi\in C^\theta_b(H_n),
\end{equation}
where $\gamma_3$ is given by \eqref{G4} (or \eqref{G4v}). Arguing as above, by \eqref{StimaC0} and \eqref{C2}, for every $t\in [0,T]$ we have 
\begin{align*}
& \norm{\nabla_x u_{n,k}^{(i)}(t,\cdot)}_{C_b(H_n;H_n)} \\
\leq & M_0\int_t^{T}e^{-{\mathfrak c}(s-t)r_{k,i}}\frac{1}{(s-t)^{\gamma_3(1-\theta)}}\norm{g(s,\cdot)}_{C^\theta_b(H_n)}\left[1+\mathfrak{C}_0\norm{C_n(s,\cdot)}_{C_b^\theta(H_n;U)}\right]ds\\
&\leq \frac{M_0\left(1+\mathfrak C_0\sup_{t\in[0,T]}\|C_n(t,\cdot)\|_{C_b^\theta(H_n;U)}\right)}{r_{k,i}^{1-\gamma_3(1-\theta)}}\sup_{t\in [0,T]}\norm{g(t,\cdot)}_{C^\theta_b(H_n)}\int_0^{\infty}e^{-{\mathfrak c}\xi}\xi^{-\gamma_3(1-\theta)}d\xi.
\end{align*}
Hence, noting that \eqref{eq:cond-2} implies $\gamma_3(1-\theta)<1$ and recalling \eqref{stima-normH}, we obtain \eqref{Thm:stime-unif1} with 
\begin{align*}
\mathfrak{C}_2=\frac{2M_0\left(1+\mathfrak C_0\sup_{t\geq 0}\|C(t,\cdot)\|_{C_b^\theta(H;U)}\right)}{{\mathfrak{c}}^{1-\gamma_3(1-\theta)}(1-\gamma_3(1-\theta))}.
\end{align*}

Finally, by \eqref{Prop:Schauder2}, with $\vartheta=\theta$, $h=\mathcal{G}_nu$ and $u\in U$ we have
\[
\norm{D_x\left[\mathcal{G}_n^*\nabla_x R_n(t)\varphi\right]}_{C_b(H_n;\mathcal{L}(H_n;U))}\leq \norm{e^{tA_n}}^\theta_{\mathcal{L}(H_n)}\norm{\Gamma_{t,n}}_{\mathcal{L}(H_n)}^{1-\theta}\norm{\Gamma_{t,n}\mathcal{G}_n}_{\mathcal{L}(U;H_n)}\norm{\varphi}_{C^\theta_b(H_n)}
\]
for every $t\in (0,T]$ and $\varphi\in C^\theta_b(H_n)$. 
By \eqref{SuperStima3} and \eqref{SuperStima4}, we obtain
\begin{equation}\label{C3V}
\norm{D_x\left[\mathcal{G}_n^*\nabla_x R_n(t)\varphi\right]}_{C_b(H_n;\mathcal{L}(H_n;U))}\leq \frac{M_0^2}{t^{\frac12+\gamma_3(1-\theta)}}\norm{\varphi}_{C^\theta_b(H_n)},\qquad t\in (0,T],\ \varphi\in C^\theta_b(H_n).
\end{equation}
Performing the same computations as above, we infer that
\begin{align*}
&\norm{D_x\left[\mathcal{G}_n^*\nabla_x u_{n,k}^{(i)}(t,\cdot)\right]}_{C_b(H_n;\mathcal{L}(H_n;U))} \\
\leq & \frac{M_0^2\left(1+\mathfrak C_0\sup_{t\in[0,T]}\|C_n(t,\cdot)\|_{C_b^\theta(H_n;U)}\right)}{r_{k,i}^{\frac12-\gamma_3(1-\theta)}}\sup_{t\in [0,T]}\norm{g(t,\cdot)}_{C^\theta_b(H_n)}\int_0^{\infty}e^{-{\mathfrak c}\xi}\xi^{-\frac12-\gamma_3(1-\theta)}d\xi
\end{align*}
for every $t\in [0,T]$. Noting that \eqref{eq:cond-2} implies $\frac12+\gamma_3(1-\theta)<1$ and recalling \eqref{stima-normH}, we obtain \eqref{Thm:stime-unif3} with 
\begin{align*}
\mathfrak C_3=\frac{2M_0^2\left(1+\mathfrak C_0\sup_{t\geq 0}\|C(t,\cdot)\|_{C_b^\theta(H;U)}\right)}{{\mathfrak{c}}^{\frac12-\gamma_3(1-\theta)}(\frac12-\gamma_3(1-\theta))}.
\end{align*}
\end{proof}

Let $n\in\N$, $k\in\{1,\ldots,n\}$, $i\in\{-,+\}$ and $\theta\in (0,1)$. In the next proposition, we will show that, for every $\mathfrak c>{\mathfrak c}_0$, the solution $u_{n,k}^{(i)}$ to \eqref{Def:Unki} given by Theorem \ref{Thm: soluzione-kolmogorov} is also a strong solution to the backward Kolmogorov equation
{\small
\begin{align*}
\left\{
\begin{array}{ll}
\displaystyle \frac{\partial}{\partial t}v(t,x)+(L_nv(t,\cdot))(x)=-\scal{\mathcal{V}_n^*\nabla_x v(t,x)}{C_{n}(t,x)}_{U}-g(t,x)-{\mathfrak c}\lambda_{k}^{(i)}v(t,x), & (t,x)\in [0,T)\times H_n, \\[2mm]
v(T,x)=0, & x\in H_n,
\end{array}
\right.
\end{align*}
}
where
\begin{align}
\label{Def:OU}
(L_n \varphi)(x)=\frac12\sum_{j,r=1}^{2n}[(\mathcal{G}_n\mathcal{G}_n^*)_{jr}D^2_{jr}\varphi(x)]+\langle A_n x,\nabla \varphi(x)\rangle_{H_n}, \qquad x\in H_n, \ \varphi\in C^2_b(H_n).
\end{align}

\begin{thm}\label{Thm:OU-finito}
Assume that hypotheses \ref{H0}, \ref{H1}, \ref{H2b}, \ref{H3} and \eqref{eq:cond-1} hold. For every $n\in\N$, $k\in\{1,\ldots,n\}$, $i\in\{-,+\}$ and $g:[0,T]\times H_n\rightarrow \R$ continuous such that 
\[
\sup_{t\in [0,T]}\norm{g(t,\cdot)}_{C^\theta_b(H_n)}<\infty,
\]
let $\mathfrak c>{\mathfrak c}_0$ and let $u_{n,k}^{(i)}$ be the unique solution to \eqref{Def:Unki} given by Theorem \ref{Thm: soluzione-kolmogorov}.\\
Then, there exists a sequence $\{f_{n,k,h}\}_{h\in\N}\subseteq C^{0,2}_b([0,T]\times H_n)$ such that:
\begin{enumerate}
\item 
$\displaystyle 
\sup_{h\in\N}\sup_{t\in[0,T]}\|f_{n,k,h}(t,\cdot)\|_{C_b^\theta(H_n)}<\infty$ and
\begin{align}
\lim_{h\rightarrow\infty}\norm{f_{n,k,h}(t,\cdot)-\scal{\mathcal{V}_n^*\nabla_x u^{(i)}_{n,k}(t,\cdot)}{C_n(t,\cdot)}_{U}-g(t,\cdot)}_{C_b(H_n)}=0, \qquad t\in[0,T];\label{Thm:OU-finito2}
\end{align}
\item  for every $h\in\N$, the parabolic equation 
\begin{equation}
\label{parabolica_generale}
\begin{cases}
\frac{\partial}{\partial t}v(t,x)+L_nv(t,x)=-f_{n,k,h}(t,x)-{\mathfrak c}\lambda_{k}^{(i)}v(t,x), & (t,x)\in[0,T)\times H_n, \\[1mm]
 v(T,x)=0, & x\in H_n,
\end{cases}
\end{equation}
has a unique strict solution $u^{(i)}_{n,k,h}\in C^{1,2}_b([0,T]\times H_n)$ given by
\begin{align}\label{Kolmo-nkh}
u^{(i)}_{n,k,h}(t,x)=\int_{t}^Te^{{\mathfrak c}(s-t)\lambda_k^{(i)}}(R_n(s-t)f_{n,k,h}(s,\cdot))(x)ds, \qquad (t,x)\in[0,T] \times H_n;
\end{align}
\item for every $t\in [0,T]$ we have
\begin{align}
&\lim_{h\rightarrow \infty}\norm{u^{(i)}_{n,k,h}(t,\cdot)-u^{(i)}_{n,k}(t,\cdot)}_{C^1_{b,\mathcal{V}_n}(H_n)}=0,\label{Thm:OU-finito3}\\
&\lim_{h\rightarrow \infty}\norm{\mathcal{G}_n^*\nabla u^{(i)}_{n,k,h}(t,\cdot)-\mathcal{G}_n^*\nabla u^{(i)}_{n,k}(t,\cdot)}_{C_{b}(H_n;U)}=0.\label{Thm:OU-finito4}
\end{align}

\end{enumerate}
\end{thm}
\begin{proof}
Since $\dim(H_n)<\infty$, all these statements are standard; we refer to \cite{Lun1997} for a detailed discussion. However, for the reader’s convenience, we present a sketch of the proof.

Fix $n\in\N$, $k\in\{1,\ldots,n\}$, $i\in\{-,+\}$ and $g$ as in the statement and let $\mathfrak c>{\mathfrak c}_0$. By assumption \ref{H2b} and Theorem \ref{Thm: soluzione-kolmogorov}, the function $x\rightarrow \scal{\mathcal{V}_n^*\nabla_x u^{(i)}_{n,k}(t,x)}{C_n(t,x)}_{U}+g(t,x)$ belongs to $C^\theta_b(H_n)$ for every $t\in [0,T]$. Therefore, by standard approximation arguments there exists a sequence $\{f_{n,k,h}\}_{h\in\N}\subseteq C^{0,2}_b([0,T]\times H_n)$ ($f_{n,k,h}$ should be defined by means of convolution with mollifiers for every $h\in\N$) such that $\displaystyle 
\sup_{h\in\N}\sup_{t\in[0,T]}\|f_{n,k,h}(t,\cdot)\|_{C_b^\theta(H_n)}<\infty$ and \eqref{Thm:OU-finito2} hold true.

From the definition of $\{R_n(t)\}_{t\geq 0}$ and the Cameron-Martin formula, we infer that, for every $t\geq 0$ and $\varphi\in C^2_b(H_n)$,
\begin{align*}
&R_n(t)\varphi\in C^2_b(H_n), \qquad \frac{\partial}{\partial t}(R_n(t)\varphi)=(L_nR_n(t)\varphi)=(R_n(t)L_n\varphi),
\end{align*}
where $L_n$ is the Ornstein-Uhlenbeck operator given by \eqref{Def:OU}. Since $f_{n,k,h}\in C^{0,2}_b([0,T]\times H_n)$ for every $h\in \N$, the function $u_{n,k,h}^{(i)}$ given by \eqref{Kolmo-nkh} solves \eqref{parabolica_generale} for every $h\in\N$. 

We prove \eqref{Thm:OU-finito3}. We set
\[
f_{n,k}(t,x):=\scal{\mathcal{V}_n^*\nabla_x u^{(i)}_{n,k}(t,x)}{C_n(t,x)}_{U}+g(t,x),\qquad t\in [0,T],\ x\in H_n.
\]
By \eqref{Prop:Vschauder2}, with $\vartheta=0$, for every $h\in \N$ and $t\in[0,T]$ we have
\begin{align}\label{pre-convergenza}
& \norm{u^{(i)}_{n,k,h}(t,\cdot)-u^{(i)}_{n,k}(t,\cdot)}_{C^1_{b,\mathcal{V}_n}(H_n)} \notag \\
\leq & \int_t^T e^{-{\mathfrak c}(s-t)r_{k,i}}\left(1+2\norm{\Gamma_{s-t,n}\mathcal{V}_n}_{\mathcal{L}(U;H_n)}\right)\norm{f_{n,k,h}(s,\cdot)-f_{n,k}(s,\cdot)}_{C_b(H_n)}ds.
\end{align}
Then, \eqref{SuperStima2} (taking into account that \eqref{eq:cond-1} implies $\gamma_2<1$) and \eqref{pre-convergenza} yield \eqref{Thm:OU-finito3}.

Finally, to prove \eqref{Thm:OU-finito4} it suffices to notice that, since $\delta,\beta\geq 0$, we have $\mathcal G_n^*=\Lambda_n^{-(\delta+\beta)}
\mathcal V_n^*$ for every $n\in\N$ and $\sup_{n\in\N}\big\|\Lambda_n^{-(\delta+\beta)}
\big\|_{\mathcal L(U)}=\mathfrak C<\infty$. Hence, from \eqref{Thm:OU-finito3}, for every $t\in [0,T]$ we obtain
{\small
\begin{align*}
& \lim_{h\rightarrow \infty}\norm{\mathcal{G}_n^*\nabla u^{(i)}_{n,k,h}(t,\cdot)-\mathcal{G}_n^*\nabla u^{(i)}_{n,k}(t,\cdot)}_{C_{b}(H_n;U)} \leq  \mathfrak C\lim_{h\rightarrow \infty}\norm{\mathcal{V}_n^*\nabla u^{(i)}_{n,k,h}(t,\cdot)-\mathcal{V}_n^*\nabla u^{(i)}_{n,k}(t,\cdot)}_{C_{b}(H_n;U)}=0.
\end{align*}
}
\end{proof}

\section{Uniqueness}\label{sec:uniq}

\subsection{The approximating process}
Assume that \ref{H0}, \ref{H0.5}, \ref{H1}, \ref{H2b} and \ref{H3} hold. Given a weak mild solution $(\Omega,\mathcal F,(\mathcal F_t)_{t\geq0}, \mathbb P, W,X)$ to \eqref{abstract_formulation}, we consider the approximating sequence of processes $\{X_n\}_{n\in\N}$, where, for every $t\in[0,T]$, the random variable $X_n(t)$ satisfies
\begin{align}
\label{Def:mild_form_approx_process}
X_n(t)=e^{tA_n}P_nx+\int_0^t e^{(t-s)A_n}\mathcal V_nC_n(s,X(s))ds+\int_0^t e^{(t-s)A_n}\mathcal G_ndW(s), \qquad \mathbb P\textrm{-a.s.},    
\end{align}
where $A_n=P_nA=AP_n$, $\mathcal{V}_n=P_n\mathcal{V}=\mathcal VQ_n$, $C_n(t,y)=C(t,P_ny)$ and $\mathcal G_n=P_n\mathcal G=\mathcal GQ_n$ for every $n\in\N$, $t\in [0,T]$ and $y\in H$. \\
We stress that, for every $n\in\N$, the process $\{X_n\}_{n\in\N}$ takes values into the finite dimensional space $H_n$ and, for every $t\in[0,T]$, it fulfills, $\mathbb P$-a.s.,
\begin{align}
\label{finite_dim_SDE}
X_n(t)=P_nx+\int_0^t\left[A_nX_n(s)+\mathcal V_nC_n(s,X(s))\right]ds+\mathcal G_nW(t), \quad t\in[0,T], \qquad X_n(0)=P_nx.
\end{align}
For every $n\in\N$, $A_n\in\mathcal{L}(H_n)$ and $\{\mathcal G_nW(t)\}_{t\geq 0}$ is a $H_n$-valued stochastic process, so that the integral formulation \eqref{finite_dim_SDE} is well defined in $H_n$. Further, for every $n\in\N$ and $t\in[0,T]$ we set
\begin{align*}
W_{A,n}(t)=\int_0^te^{(t-s)A_n}\mathcal G_n dW(s), \qquad \mathbb P\textup{-a.s.}   
\end{align*}

\begin{prop}\label{prop:Xn}
Assume that \ref{H0}, \ref{H0.5}, \ref{H1}, \ref{H2b} and \ref{H3} hold and let $x\in H$. Then, for every $T>0$, 
\begin{equation}\label{supE}
\lim_{n \to \infty}\sup_{t\in [0,T]}\mathbb{E}
\|X_{n}(t)-X(t)\|_H^2=0\,.
\end{equation} 
\end{prop}
\begin{proof}
Fix $x\in H$. We begin by proving that 
\begin{align}\label{convsupWA}
\lim_{n\rightarrow \infty}\sup_{t\in [0,T]}\mathbb{E}\left[\|W_A(t)-W_{A,n}(t)\|_H^2\right]=0.
\end{align}
From the definition of $W_A$ and of $W_{A,n}$, for every $t\geq0$ we can write
\begin{align*}
W_A(t)-W_{A,n}(t)=\int_0^t e^{(t-s)A}\left(\mathcal{G}-P_n\mathcal{G}\right)dW(s), \qquad \mathbb P\textup{-a.s.}
\end{align*}

Let $\{\widetilde e_k\}_{k\in\N}$ be the orthonormal basis of $U$ introduced in 
\eqref{Def:base-ON}. By \eqref{rep_P_nh} we infer that
\[
(I-P_n)\mathcal G\widetilde e_k=
\begin{cases}
0, & k\leq n,\\
\mathcal G\widetilde e_k, & k>n .
\end{cases}
\]
Therefore, for every $t\in[0,T]$ and every $n\in\N$, we get
\begin{align*}
\mathbb{E}\left[\|W_A(t)-W_{A,n}(t)\|^2_H\right]
&=\int_0^t \left\|e^{(t-s)A}\left(I-P_n\right)\mathcal{G}\right\|_{\mathcal{L}_2(U;H)}^2ds
=\int_0^t\sum^{\infty}_{k=n+1} \|e^{sA}\mathcal{G}\widetilde e_k\|_H^2\,ds\\
&\leq \int_0^T\sum^{\infty}_{k=n} \|e^{sA}\mathcal{G}\widetilde e_k\|_H^2\,ds .
\end{align*}
By \ref{H3} we have
\[
\int_0^T\sum_{k=1}^{\infty} \|e^{sA}\mathcal{G}\widetilde e_k\|_H^2\,ds
=\int_0^T\|e^{sA}\mathcal G\|^2_{\mathcal L_2(U;H)}ds
=\int_0^T {\rm Trace}_H\left[e^{sA}\mathcal{G}\mathcal{G}^*e^{sA^*}\right]ds<\infty,
\]
and so that, by the dominated convergence theorem we obtain \eqref{convsupWA}.

We pass to the second term. Let $\gamma_1$ be the constant given by \eqref{G1} (or \ref{G1v}). By \ref{H3}, we have $\gamma_1<1$. Let us fix $q\in\left(1,\frac1{\gamma_1}\right)$ and $p:=\frac{q}{q-1}$, so that $\frac1p+\frac1q=1$. By estimates \eqref{stima_tilde S} and \ref{SuperStima1}, the H\"older inequality and the fact that, for every $s>0$ and $n\in\N$, $e^{sA_n}\mathcal V_n=e^{sA}\mathcal VQ_n$, where $Q_n$ is the projection on $U_n$, for every $n\in\N$ and $t>0$ we get
\begin{align*}
&\mathbb{E}\left[\norm{\int_0^te^{(t-s)A_n}\mathcal V_nC_n(s,X(s))ds
-\int_0^te^{(t-s)A}\mathcal VC(s,X(s))ds}^2_H\right]\\
\leq & 2\mathbb{E}\left[\norm{\int_0^t  e^{(t-s)A_n}\mathcal{V}_n
\left[C(s,P_nX(s))
- C(s,X(s))\right]ds}^2_H\right]\\
& +2\mathbb{E}\left[\norm{\int_0^t  e^{(t-s)A}\mathcal V[Q_nC(s,X(s))-C(s,X(s))]ds}^2_H\right] \\
\leq & 2\mathbb E\left[M_0\left(\int_0^t\frac{1}{(t-s)^{q\gamma_1}}ds\right)^{\frac1q}\left(\int_0^t\|C(s,P_nX(s))-C(s,X(s))\|_{U}^pds\right)^{\frac1p}\right]^2 \\
& +2\mathbb E\left[ M_0\left(\int_0^t\frac{1}{(t-s)^{q\gamma_1}}ds\right)^{\frac1q}\left(\int_0^t\|Q_nC(s,X(s))-C(s,X(s))\|_{U}^pds\right)^{\frac1p}\right]^2 \\
\leq & 2\left(\frac{(M_0)^qT^{1-q\gamma_1}}{1-q\gamma_1}\right)^{2/q} \mathbb E\left[\left(\int_0^T\|C(s,P_nX(s))-C(s,X(s))\|_{U}^pds\right)^{\frac1p}\right]^2 \\
& + 2\left(\frac{(M_0)^qT^{1-q\gamma_1}}{1-q\gamma_1}\right)^{2/q} \mathbb E\left[\left(\int_0^T\|Q_nC(s,X(s))-C(s,X(s))\|_{U}^pds\right)^{\frac1p}\right]^2 \\
= & I_1(n)+I_2(n).
\end{align*}
The convergence of $(I_1(n))_{n\in\N}$ to $0$ follows from the dominated convergence theorem, since the sequence $\|C(s,P_nX(s))-C(s,X(s))\|_{U}$ tends to $0$ as $n$ tends to infinity, for every $s\in[0,T]$, and $\|C(s,P_nX(s))-C(s,X(s))\|_{U}\leq 2\|C\|_{C_b([0,T]\times H;U)}$ for every $s\in[0,T]$ and every $n\in\N$. 

Since also $\|Q_nC(s,X(s))-C(s,X(s))\|_{U}$ tends to $0$ as $n$ tends to infinity for every $s\in[0,T]$ and $\|Q_nC(s,X(s))-C(s,X(s))\|_{U}\leq 2\|C\|_{C_b([0,T]\times H;U)}$ for every $s\in[0,T]$ and every $n\in\N$, it follows that also $(I_2(n))_{n\in\N}$ vanishes as $n$ tend to infinity.

Finally, since $A_n=AP_n$ we have 
\begin{align*}
\sup_{t\geq0}
\norm{e^{tA_n} P_nx-e^{tA}x}^2_H
&=\sup_{t\geq0}\norm{e^{tA}(P_nx-x)}^2_H
\leq 
\sup_{t\geq0}\norm{e^{tA}}^2_{\mathscr{L}(H)}
\norm{P_nx-x}^2_H\,,
\end{align*}
which implies, recalling that $A$ is of negative type thanks to \ref{H1}, that
\[
\lim_{n\rightarrow\infty}
\sup_{t\geq0}\norm{e^{tA_n} P_nx-e^{tA}x}^2_H=0\,.
\]
\end{proof}

\subsection{Weak Uniqueness}

\begin{thm}\label{weak-B}
Assume that conditions \ref{H0}, \ref{H0.5}, \ref{H1}, \ref{H2b}, \ref{H3} and 
    \begin{equation*}
\beta+\delta<\frac12\min\{\alpha,1-\alpha\}
\end{equation*}
hold true. Then, weak uniqueness in the sense of Definition \ref{uniqueness} holds for \eqref{abstract_formulation}.
\end{thm}
\begin{proof}
We prove weak uniqueness in $C([0,\infty);H)$. Let $(X,W_1)$ and $(Y,W_2)$ be two weak mild solutions to \eqref{abstract_formulation} in the sense of Definition \ref{def:mild-sol} with same initial datum $x\in H$, and let $\{X_n\}_{n\in\N}$ and $\{Y_n\}_{n\in\N}$ be the respective approximated solutions defined by \eqref{Def:mild_form_approx_process}.
Let $g\in C^\theta_b(H)$ be arbitrary but fixed, let $n\in\N$ and let $u_n:=u^{+}_{n,1}$ be the solution to the Kolmogorov equation \eqref{Def:Unki} with $i=+$ and $k=1$. Let ${\mathfrak{c}}>\mathfrak c_0$ where $\mathfrak c_0$ is the constant given by Theorem \ref{Thm: soluzione-kolmogorov}. We consider the sequences $\{u_{n,h}\}_{h\in\N}$ and $\{f_{n,h}\}_{h\in\N}$ given by Theorem \ref{Thm:OU-finito} (with $i=+$ and $k=1$), which fulfill, for every $h\in\N$, 
\begin{align}
\left\{
\begin{array}{ll}
\displaystyle \frac{\partial}{\partial t}u_{n,h}(t,x)+L_nu_{n,h}(t,x)=-f_{n,h}(t,x)-{\mathfrak{c}}\lambda_{1}^{+}u_{n,h}(t,x), & x\in H_n, \  t\in [0,T),\\[2mm]
\displaystyle u_{n,h}(T,x)=0, & x\in H_n,
\end{array}
\right.
\label{Parabolica-nh}
\end{align}
where $L_n$ is the Ornstein-Uhlenbeck operator defined in \eqref{Def:OU}. Then, an application of the It\^o formula, together with the Kolmogorov equation \eqref{Parabolica-nh}, yields, for all $t\geq0$,
\begin{align*}
\nonumber &e^{{\mathfrak{c}}\lambda^+_1 t}\left(\mathbb{E}_1\left[u_{n,h}(t,X_n(t))\right]-\mathbb{E}_2\left[u_{n,h}(t,Y_n(t)))\right]\right)\\
&\phantom{aaaaaaa}+\mathbb{E}_1\left[\int_0^te^{{\mathfrak{c}}\lambda^+_1 s}f_{n,h}(s,X_n(s))ds\right]-\mathbb{E}_2\left[\int_0^te^{{\mathfrak{c}}\lambda^+_1 s}f_{n,h}(s,Y_n(s))ds\right]\\
&\phantom{aaaaaaa}=-\mathbb{E}_1\left[\int_0^te^{{\mathfrak{c}}\lambda^+_1 s}\scal{\mathcal{V}^*_n\nabla u_{n,h}(s,X_n(s))}{
C_n(s,X(s))}_{U}d s\right]\\
&\phantom{aaaaaaa}+\mathbb{E}_2\left[\int_0^te^{{\mathfrak{c}}\lambda^+_1 s}\scal{\mathcal{V}^*_n\nabla u_{n,h}(s,Y_n(s))}{
C_n(s,Y(s))}_{U}d s\right],
\end{align*}
letting $h$ tend to infinity, by Theorem \ref{Thm:OU-finito} we obtain
\begin{align*}
\nonumber &e^{{\mathfrak{c}}\lambda^+_1 t}\left(\mathbb{E}_1\left[u_{n}(t,X_n(t))\right]-\mathbb{E}_2\left[u_{n}(t,Y_n(t)))\right]\right)\\
&\phantom{aaaaaaa}+\mathbb{E}_1\left[\int_0^te^{{\mathfrak{c}}\lambda^+_1 s}g(X_n(s))ds\right]-\mathbb{E}_2\left[\int_0^te^{{\mathfrak{c}}\lambda^+_1 s}g(Y_n(s))ds\right]\\
&\phantom{aaaaaaa}=-\mathbb{E}_1\left[\int_0^te^{{\mathfrak{c}}\lambda^+_1 s}\scal{\mathcal{V}^*_n\nabla u_n(s,X_n(s))}{
C_n(s,X(s))-C_n(s,X_n(s))}_{U}d s\right]\\
&\phantom{aaaaaaa}+\mathbb{E}_2\left[\int_0^te^{{\mathfrak{c}}\lambda^+_1 s}\scal{\mathcal{V}^*_n\nabla u_n(s,Y_n(s))}{
C_n(s,Y(s))-C_n(s,Y_n(s))}_{U}d s\right].
\end{align*}
Hence,
\begin{align*}
&\abs{\int_0^te^{{\mathfrak c}\lambda^+_1 s}\mathbb{E}_1\left[g(X_n(s))\right]ds-\int_0^te^{{\mathfrak c}\lambda^+_1 s}\mathbb{E}_2\left[g(Y_n(s))\right]ds} \\
& \leq \abs{e^{{\mathfrak{c}}\lambda^+_1 t}\mathbb{E}_1\left[u_{n}(t,X_n(t))\right]}
+\abs{e^{{\mathfrak{c}}\lambda_1^+t}\mathbb{E}_2\left[u_{n}(t,Y_n(t))\right]}
\\
&+\mathbb{E}_1\left[\int_0^t\abs{e^{{\mathfrak{c}}\lambda^+_1 s}\scal{\mathcal{V}_n^*\nabla u_n(s,X_n(s))}{
C_n(s,X(s))-C_n(s,X_n(s))}_{U}}d s\right]\\
&+\mathbb{E}_2\left[\int_0^t\abs{e^{{\mathfrak{c}}\lambda^+_1 s}\scal{\mathcal{V}_n^*\nabla u_n(s,Y_n(s))}{
C_n(s,Y(s))-C_n(s,Y_n(s))}_{U}}d s\right]
\end{align*}
Recalling that ${\rm Re}(\lambda^+_1)<0$, by assumption \ref{H2b} and estimate \eqref{StimaC0} we have
\begin{align*}
&\abs{\int_0^te^{{\mathfrak{c}}\lambda^+_1 s}\mathbb{E}_1\left[g(X_n(s))\right]ds-\int_0^te^{{\mathfrak{c}}\lambda^+_1 s}\mathbb{E}_2\left[g(Y_n(s))\right]ds}\leq 2e^{{\mathfrak{c}} {\rm Re}(\lambda^+_1)t}\mathfrak{C}_0\norm{g}_{C^\theta_b(H)}
\\
&+\norm{g}_{C^\theta_b(H)}\mathfrak{C}_0\sup_{s\geq 0}\norm{C(s,\cdot)}_{C^\theta_b(H;U)}\int_0^t\left(\mathbb{E}_1\left[\norm{X_n(s)-X(s)}^\theta_{H}\right]+\mathbb{E}_2\left[\norm{Y_n(s)-Y(s)}^\theta_{H}\right]\right)ds
\end{align*}
Since $\theta\in (0,1)$, letting $n$ tend to infinity, by \eqref{supE} we obtain
\begin{align*}
&\abs{\int_0^te^{{\mathfrak{c}}\lambda^+_1 s}\mathbb{E}_1\left[g(X(s))\right]ds-\int_0^te^{{\mathfrak{c}}\lambda^+_1 s}\mathbb{E}_2\left[g(Y(s))\right]ds}\leq 2e^{{\mathfrak{c}} {\rm Re}(\lambda^+_1)t}\mathfrak{C}_0\norm{g}_{C^\theta_b(H)}.
\end{align*}
Letting $t$ tend to infinity, we conclude that
\[
\int_0^{\infty}e^{{\mathfrak c}\lambda^+_1 s}\mathbb{E}_1\left[g(X(s))\right]ds=\int_0^{\infty}e^{{\mathfrak c}\lambda^+_1 s}\mathbb{E}_2\left[g(Y(s))\right]ds.
\]
Since ${\mathfrak c}>{\mathfrak c}_0$ is arbitrary, by standard properties of the Laplace 
transform and approximation arguments we have
\begin{equation}\label{eq:marginali}
\mathbb E_1\big[g(X(s))\big]=\mathbb E_2\big[g(Y(s))\big],\qquad s\geq 0,\ g\in C_b(H).
\end{equation}
Moreover, by the same arguments used above, \eqref{eq:marginali} holds for every $t_0\geq0$, every $s\geq t_0$ and every pair of weak 
mild solutions on $[t_0,\infty)$ with the same $\mathcal F_{t_0}$-measurable initial datum.\\
From the uniqueness of the marginals, we obtain uniqueness in law in the sense of Definition \ref{uniqueness} by following the classical approach via the martingale problem; see the Lemma \ref{lem:mart} in the Appendix \ref{app:mart}. Lemma \ref{lem:mart} is a slight modification of the results in \cite[Appendix A.2]{Pri2015} to cover the case in which the nonlinear part of the drift depends on time.

\end{proof}

\begin{thm}\label{weak-UB}
Assume that conditions \ref{H0}, \ref{H0.5}, \ref{H1}, \ref{H2}, \ref{H3} and 
\begin{equation*}
\beta+\delta<\frac12\min\{\alpha,1-\alpha\}.
\end{equation*}
Then, weak uniqueness holds for \eqref{abstract_formulation} in the sense of Definition \ref{uniqueness}.
\end{thm}

\begin{proof}
For every $N\in\N$ let $\tau_N:C([0,+\infty);H)\rightarrow[0,T]$ be defined as
\[
\tau_N(\zeta):=T\wedge\inf\{t\in[0,T]:\ \|\zeta(t)\|_H>N\},\qquad \zeta\in C([0,T];H),
\]
and set $\tau_N^X:=\tau_N(X)$ on $\Omega_1$ and $\tau_N^Y:=\tau_N(Y)$ on $\Omega_2$. It is not difficult to show that $\tau^X_N\nearrow T$ $\mathbb{P}_1$-almost 
surely and $\tau^Y_N\nearrow T$ $\mathbb{P}_2$-almost surely as $N$ tends to infinity.
Let $\Pi_N$ be the projection on the closed ball of $H$ with center in $0$ and radius $N$ and let $C^T_N:[0, \infty)\times H\rightarrow U$ be defined by 
\[
 C^T_N(t,x):=C(t\wedge T,\Pi_Nx)\,,\quad t\geq 0,\, x\in H.
\]
Since $C$ verifies \ref{H2}, it follows that $C^T_N$ verifies \ref{H2b}, for every $N\in\N$. Let now $N\in\mathbb N$ be fixed. By the localization-extension arguments in \cite[Section 4.2]{BOS}, there exist two weak mild solutions
$(\widetilde{X}_N,W_1)$ and $(\widetilde{Y}_N,W_2)$ to \eqref{abstract_formulation} with $C$ replaced by $C^T_N$ on $[0, \infty)$, such that
\begin{align*}
&X=\widetilde{X}_N,\quad \text{in } [\![0,\tau^X_N]\!],\\
&Y=\widetilde{Y}_N,\qquad\text{in } [\![0,\tau^Y_N]\!].
\end{align*}
Hence, by Theorem \ref{weak-B} and recalling Definition \ref{uniqueness}, for every $N\in\N$ and for every measurable bounded $\psi:C([0,+\infty); H)\to \R$, we get
\[
\mathbb{E}_1\left[\psi(X)\right]=\mathbb{E}_1\left[\psi\left(\widetilde{X}_N\left(\cdot\wedge\tau_N(\widetilde{X}_N)\right)\right)\right]=\mathbb{E}_2\left[\psi\left(\widetilde{Y}_N\left(\cdot\wedge\tau_N(\widetilde{Y}_N)\right)\right)\right] =\mathbb{E}_2\left[\psi(Y)\right],
\]
and so that, letting $N$ tend to infinity and recalling that $\tau^X_N\nearrow T$ $\mathbb{P}_1$-almost 
surely and $\tau^Y_N\nearrow T$ $\mathbb{P}_2$-almost surely as $N$ goes to infinity, we conclude the proof.
\end{proof}

By Theorem \ref{weak-UB} (assuming $\mu_n=n^{2k/d}$ with $k\in\{1,2\}$ and $d\in\{1,2,3\}$) we can deduce explicit conditions on $\alpha,\beta$ and $\delta$ which guarantee weak uniqueness for \eqref{concrete_damped_equation}.

\begin{coro}\label{cor:weak}
Weak uniqueness holds for \eqref{concrete_damped_equation} in the following cases
\begin{equation}
        \begin{cases}
          &d<2k\\
          &\alpha\in(d/4k,1/2]\\
          &\theta\in (0,1)\\
          &\beta\in [0,\min\{\alpha/2; \alpha-d/4k\})\\
          &\delta\in (d/4k-\alpha/2,\alpha/2-\beta)\cap [0,+\infty)
        \end{cases},\qquad \begin{cases}
          &d<2k\\
          &\alpha\in [1/2,1)\\
          &\theta\in (0,1)\\
          &\beta\in [0,\min\{(1-\alpha)/2; 1/2-d/4k\})\\
          &\delta\in (d/4k-\alpha/2,(1-\alpha)/2-\beta)\cap [0,+\infty)
        \end{cases}.
\end{equation}
\end{coro}

\subsection{Pathwise uniqueness}

Now we provide the representation of $X_n$ in terms of $u_{n,k}^{(i)}$ which will be crucial to prove the pathwise uniqueness.

\begin{thm}\label{Thm:Ito-Tanaka}
Assume that assumptions \ref{H0}, \ref{H0.5}, \ref{H1}, \ref{H2b}, \ref{H3} and 
\begin{equation*}
\beta+\delta<\frac12\min\{\alpha,1-\alpha\}
\end{equation*}
hold. Let $(\Omega,\mathcal F,(\mathcal F_t)_{t\geq0}, \mathbb P, W,X)$ be a weak mild solution to \eqref{abstract_formulation}. For every $n\in\N$ and $x\in H$, the process $\{X_n(t,x)\}_{t\in [0,T]}$ defined by \eqref{Def:mild_form_approx_process} satisfies the following representation: 
\begin{align}\label{Def:formula-Ito-Tanaka}
X_n(t,x)&=e^{tA_n}\left[P_nx+u_n(0,P_nx)\right]-u_n(t,X_n(t,x)) \notag \\
&-({\mathfrak c}+1)A_n\int_0^t e^{(t-s)A_n}u_n(s,X_n(s,x))ds \notag \\
&+\int_0^t e^{(t-s)A_n}\mathcal{V}_n\left[C_n(s,X(s,x))-C_n(s,X_n(s,x))\right]ds \notag \\
&+\int_0^t e^{(t-s)A_n}D_xu_{n}(s,X_n(s,x))\mathcal{V}_n\left[C_n(s,X(s,x))-C_n(s,X_n(s,x))\right]ds \notag \\
&+\int_0^te^{(t-s)A_n}D_xu_{n}(s,X_n(s,x))\mathcal{G}_ndW(s)\notag\\
&+\int_0^t e^{(t-s)A_n}\mathcal{G}_ndW(s),
\end{align}
where, in analogy with \eqref{expr_h+_h-}, for every $s\in[0,T]$ and every $x\in H_n$, we set 
\begin{align}
& u_n(s,x):=\sum_{k=1}^n \left(u_{n,k}^+(s,x)\Phi_k^{+}+u_{n,k}^-(s,x)\Phi_k^{-}\right), \label{def_Un}\\
&D_xu_n(s,x)\mathcal{V}_nh:=\sum_{k=1}^n\left( \scal{\mathcal{V}_n^ *\nabla_xu_{n,k}^+(s,x)}{h}_{U}\Phi_k^{+}+\scal{\mathcal{V}_n^*\nabla_xu_{n,k}^-(s,x)}{h}_{U}\Phi_k^{-}\right),\quad h\in U,
\label{def_grad_UnV}\\
&D_xu_n(s,x)\mathcal{G}_nh:=\sum_{k=1}^n\left( \scal{\mathcal{G}_n^*\nabla_xu_{n,k}^+(s,x)}{h}_{U}\Phi_k^{+}+\scal{\mathcal{G}_n^*\nabla_xu_{n,k}^-(s,x)}{h}_{U}\Phi_k^{-}\right),\quad h\in U,
\label{def_grad_Un}
\end{align}
$u_{n,k}^{(i)}$ is the unique solution to \eqref{Def:Unki} with $g=\scal{[\mathcal{V}_nC_n]^{(i)}}{\Phi^{(i)}_k}_{H_n}$ and ${\mathfrak{c}}$ is a constant given by Theorem \ref{Thm: soluzione-kolmogorov}, for every $k\in\{1,\ldots,n\}$ and $i\in\{-,+\}$.
\end{thm}
\begin{proof}
Fix $x\in H$, $n\in\N$ and a weak mild solution $(\Omega,\mathcal F,(\mathcal F_t)_{t\geq0}, \mathbb P, W,X)$ to \eqref{abstract_formulation}. Let $k\in\{1,\ldots,n\}$, $i\in\{-,+\}$ and let $u_{n,k}^{(i)}$ be the unique solution to \eqref{Def:Unki} with $g=\scal{[\mathcal{V}_nC_n]^{(i)}}{\Phi^{(i)}_k}_{H}$ and ${\mathfrak{c}}>{\mathfrak c}_0$, where ${\mathfrak c}_0$ is given by Theorem \ref{Thm: soluzione-kolmogorov}. We consider the sequences $\{u^{(i)}_{n,k,h}\}_{h\in\N}$ and $\{f_{n,k,h}\}_{h\in\N}$ given by Theorem \ref{Thm:OU-finito}, which fulfill, for every $h\in\N$, 
\begin{align}
\left\{
\begin{array}{ll}
\displaystyle \frac{\partial}{\partial t}u^{(i)}_{n,k,h}(t,x)+L_nu^{(i)}_{n,k,h}(t,x)=-f_{n,k,h}(t,x)-{\mathfrak{c}}\lambda_{k}^{(i)}u^{(i)}_{n,k,h}(t,x), & x\in H_n, \  t\in [0,T),\\[2mm]
\displaystyle u^{(i)}_{n,k,h}(T,x)=0, & x\in H_n,
\end{array}
\right.
\label{Parabolica-nkh}
\end{align}
where $L_n$ is the Ornstein-Uhlenbeck operator defined in \eqref{Def:OU}. Let $\{X_n(t,x)\}_{t\in [0,T]}$ be the approximation of $X$ defined by \eqref{Def:mild_form_approx_process}. Applying It\^o formula to $u_{n,k,h}^{(i)}(\cdot,X_n)$, by \eqref{finite_dim_SDE} we obtain
\begin{align*}
u^{(i)}_{n,k,h}(t,X_n(t,x))&=u^{(i)}_{n,k,h}(0,P_nx)+\int_0^t\left[\dfrac{\partial}{\partial s}u^{(i)}_{n,k,h}(s,X_n(s,x))+L_nu^{(i)}_{n,k,h}(s,X_n(s,x))\right]ds\\
&+\int_0^t\scal{C_n(s,X(s,x))}{\mathcal{V}_n^*\nabla_xu^{(i)}_{n,k,h}(s,X_n(s,x))}_{U}ds\\
&+\int_0^t\scal{\mathcal{G}_n^*\nabla_xu^{(i)}_{n,k,h}(s,X_n(s,x))}{dW(s)}_{U}
\end{align*}
for every $(t,x)\in [0,T]\times H_n$. From \eqref{Parabolica-nkh}, it follows that, for every $(t,x)\in[0,T]\times H_n$,
\begin{align*}
u^{(i)}_{n,k,h}(t,X_n(t,x))&=u^{(i)}_{n,k,h}(0,P_nx)-\int_0^t \left[f_{n,k,h}(s,X_n(s,x))+{\mathfrak{c}}\lambda_{k}^{(i)}u_{n,k,h}^{(i)}(s,X_n(s,x))\right]ds\\
&+\int_0^t\scal{C_n(s,X(s,x))}{\mathcal{V}_n^*\nabla_xu^{(i)}_{n,k,h}(s,X_n(s,x))}_{U}ds\\
&+\int_0^t\scal{\mathcal{G}_n^*\nabla_xu^{(i)}_{n,k,h}(s,X_n(s,x))}{dW(s)}_{U}.
\end{align*}
Letting $h$ tend to infinity, from Theorem \ref{Thm:OU-finito}, we infer that, for every $(t,x)\in[0,T]\times H_n$,
\begin{align*}
u^{(i)}_{n,k}(t,X_n(t,x))&=u^{(i)}_{n,k}(0,P_nx)-\int_0^t {\mathfrak{c}}\lambda_{k}^{(i)}u_{n,k}^{(i)}(s,X_{n}(s,x))ds-\int_0^t \scal{[\mathcal{V}_nC_n]^{(i)}(s,X_n(s,x))}{\Phi_k^{(i)}}_{H_n}ds\\
&+\int_0^t\scal{C_n(s,X(s,x))-C_n(s,X_n(s,x))}{\mathcal{V}_n^*\nabla_x u^{(i)}_{n,k}(s,X_n(s,x))}_{U}ds\\
&+\int_0^t\scal{\mathcal{G}_n^*\nabla_xu^{(i)}_{n,k}(s,X_n(s,x))}{dW(s)}_{U}.
\end{align*}
Multiplying both sides of the above equality by $\Phi_k^{(i)}$ and summing up first over $i\in\{-,+\}$ and then $k$ from $1$ to $n$, from \eqref{rep_P_nh} and \eqref{rep_A_n} we infer that \begin{align*}
u_n(t,X_n(t,x))&=u_n(0,P_nx)-{\mathfrak{c}}\int_0^t A_nu_n(s,X_n(s,x))ds-\int_0^t \mathcal{V}_nC_n(s,X_n(s,x))ds\\
&+\int_0^tD_xu_{n}(s,X_n(s,x))\mathcal{V}_n\left[C_n(s,X(s,x))-C_n(s,X_n(s,x))\right]ds\\
&+\int_0^tD_xu_{n}(s,X_n(s,x))\mathcal{G}_ndW(s),
\end{align*}
where, for every $x\in H_n$ and $s\in [0,T]$, $u_n(s,x)$, $D_xu_n(s,x)\mathcal{V}_nh$ and $D_xu_n(s,x)\mathcal{G}_nh$ have been defined in \eqref{def_Un}, \eqref{def_grad_UnV} and \eqref{def_grad_Un}, respectively. Hence,
\begin{align}
\int_0^t \mathcal{V}_nC_n(s,X_n(s,x))ds&=-u_n(t,X_n(t,x))+u_n(0,P_nx)-{\mathfrak{c}}\int_0^t A_nu_n(s,X_n(s,x))ds\notag\\
&+\int_0^tD_xu_{n}(s,X_n(s,x))\mathcal{V}_n\left[C_n(s,X(s,x))-C_n(s,X_n(s,x))\right]ds\notag\\
&+\int_0^tD_xu_{n}(s,X_n(s,x))\mathcal{G}_ndW(s).\label{formula-bn}
\end{align}
Adding and subtracting $\int_0^t[\mathcal{V}_nC_n](s,X_n(s,x))ds$ in \eqref{finite_dim_SDE}, from \eqref{formula-bn} we get, for every $(t,x)\in[0,T]\times H_n$, 
\begin{align*}
X_n(t,x)&=P_nx+\int_0^t A_nX_n(s,x)ds+\mathcal{G}_nW(t)+\int_0^t \left[\mathcal{V}_nC_n(s,X(s,x))-\mathcal{V}_nC_n(s,X_n(s,x))\right]ds\\
&-u_n(t,X_n(t,x))+u_n(0,P_nx)-{\mathfrak{c}}\int_0^t A_nu_n(s,X_n(s,x))ds\\
&+\int_0^tD_xu_{n}(s,X_n(s,x))\mathcal{V}_n\left[C_n(s,X(s,x))-C_n(s,X_n(s,x))\right]ds\\
&+\int_0^tD_xu_{n}(s,X_n(s,x))\mathcal{G}_ndW(s).
\end{align*}
Finally, passing to the mild formulation, we obtain \eqref{Def:formula-Ito-Tanaka}.
\end{proof}

Let $n\in\N$. By \eqref{lambda_beta_n}, we infer that
\begin{equation}\label{Def:VnCn}
\mathcal{V}_nC_n(s,y)=\sum_{j=1}^{n}\mu_j^{\beta}\scal{C(s,P_ny)}{\widetilde e_j}_U\left(b_j^+\Phi^+_j+b_j^-\Phi^-_j\right),\qquad s\geq0,\ y\in H,
\end{equation}
where $\{\widetilde e_j\}_{j\in\N}$ is the orthonormal basis of $U$ introduced in \eqref{Def:base-ON}. Since $\{\Phi^+_j\}_{j\in\N}$ and $\{\Phi^-_j\}_{j\in\N}$ are orthonormal systems in $H$, it follows that, for every $k\in \{1,\ldots,n\}$ and $i\in\{-,+\}$,
\[
\scal{\left[\mathcal{V}_nC_n(s,y)\right]^{(i)}}{\Phi^{(i)}_k}_H=\mu_k^{\beta}\,b_k^{(i)}\,\scal{C(s,P_ny)}{\widetilde e_k}_U,\qquad s\geq0,\ y\in H .
\]
Moreover, under condition \ref{H2b}, we set
\begin{align*}
C_0=\max\{1;c\}\sup_{t\in [0,\infty)}\norm{C(t,\cdot)}_{C^\theta_b(H;U)},
\end{align*}
where $c=\sup_{n}\|P_n\|_{\mathscr L(H)}$, see Remark \ref{norm-Pn}. In particular, we have
\begin{equation}\label{holder-Cn}
\norm{C_n(s,x)-C_n(s,y)}_U\leq C_0\norm{x-y}_H^{\theta},\qquad 
s\geq0,\ x,y\in H,\ n\in\N.
\end{equation}
Let $T>0$ and let $u_{n,k}^{(i)}$ be the unique solution to \eqref{Def:Unki} with $g=\scal{\left[\mathcal{V}_nC_n\right]^{(i)}}{\Phi^{(i)}_k}_H$ and $\mathfrak c>\mathfrak c_0$, so that
\begin{equation}\label{stima-g-nki}
\sup_{t\in [0,T]}\norm{g(t,\cdot)}_{C^\theta_b(H_n)}\leq \mu_k^{\beta}\big|b_k^{(i)}\big|\sup_{t\in [0,T]}\norm{C(t,\cdot)}_{C^\theta_b(H;U)} .
\end{equation}
Using \eqref{stima-g-nki} in the estimates \eqref{Thm:stime-unif2}, \eqref{Thm:stime-unif1} and \eqref{Thm:stime-unif3} we deduce
\begin{align}
&\sup_{t\in [0,T]}\norm{\mathcal{V}_n^*\nabla_x u_{n,k}^{(i)}(t,\cdot)}_{C_b(H_n;U)}\leq C_0 \mathfrak{C}_1({\mathfrak{c}})\mu_k^{\beta}\big|b_k^{(i)}\big|\left(-{\rm Re}\big(\lambda_k^{(i)}\big)\right)^{\gamma_1\theta+\gamma_2(1-\theta)-1},\label{EF1}\\
&\sup_{t\in [0,T]}\norm{\nabla_x u_{n,k}^{(i)}(t,\cdot)}_{C_b(H_n;H_n)}\leq C_0\mathfrak{C}_2({\mathfrak{c}})\mu_k^{\beta}\big|b_k^{(i)}\big|\left(-{\rm Re}\big(\lambda_k^{(i)}\big)\right)^{\gamma_3(1-\theta)-1},\label{EF2}\\
&\sup_{t\in [0,T]}\norm{D_x\left[\mathcal{G}_n^*\nabla_x u_{n,k}^{(i)}(t,\cdot)\right]}_{C_b(H_n;\mathcal{L}(H_n;U))}\leq C_0
\mathfrak{C}_3({\mathfrak{c}})\mu_k^{\beta}\big|b_k^{(i)}\big|\left(-{\rm Re}\big(\lambda_k^{(i)}\big)\right)^{\gamma_3(1-\theta)-1/2}\label{EF3},
\end{align}
where $\{\mu_k\}_{k\in\N}$ are the eigenvalues of $\Lambda$ given by condition \ref{H1}, $\gamma_1$, $\gamma_2$, and $\gamma_3$ are defined by \eqref{gamma-unificato},  \eqref{G2} (or \eqref{G2v}) and  \eqref{G4} (or \eqref{G4v}), respectively, $\mathfrak{C}_1,\mathfrak{C}_2$ and $\mathfrak{C}_3:\R^+\rightarrow\R^+$ are independent of $n,k\in\N$, $t\in [0,T]$ and $T>0$, and satisfy
\begin{equation}\label{Lim-cost}
\lim_{{\mathfrak{c}}\rightarrow  \infty}\mathfrak{C}_1({\mathfrak{c}})=  \lim_{{\mathfrak{c}}\rightarrow  \infty}\mathfrak{C}_2({\mathfrak{c}})=  \lim_{{\mathfrak{c}}\to\infty}\mathfrak{C}_3({\mathfrak{c}})=0.
\end{equation}
\begin{rmk}
For the proof of pathwise uniqueness, we are interested in the asymptotic behaviour of the right-hand sides of \eqref{EF1}, \eqref{EF2} and \eqref{EF3}. By \eqref{damped_stime_coefficienti_avl} and \eqref{damped_stime_coefficienti_avl_2} there exists $c_1>0$ such that for every $q\leq-1/2$, $k\in\N$ and $i\in\{-,+\}$ we have
\begin{equation}\label{stima-b-r}
\mu_k^{\beta}\big|b_k^{(i)}\big|\left(-{\rm Re}\big(\lambda_k^{(i)}\big)\right)^{q}\leq c_1\mu_k^{\min\{\alpha; 1-\alpha\}q+\beta+\min\{0;1/2-\alpha\}}.
\end{equation}
Indeed, if $\alpha\in\left(0,\frac12\right]$ then $|b_k^\pm|\sim1$ and $-{\rm Re}(\lambda_k^\pm)\sim\mu_k^\alpha$, so that both sides have the same asymptotic behaviour. If $\alpha\in\left[\frac12,1\right)$ then $|b_k^+|\sim\mu_k^{\frac12-\alpha}$ and $-{\rm Re}(\lambda_k^+)\sim\mu_k^{1-\alpha}$, so that for $i=+$ the two sides again coincide, while for $i=-$ we have $|b_k^-|\sim1$ and $-{\rm Re}(\lambda_k^-)\sim\mu_k^{\alpha}$, and the claim follows from $2q(2\alpha-1)\leq 1-2\alpha$. Finally, we notice that it is enough to consider the case $q\leq -\frac12$, as the proof of the following theorem shows.
\end{rmk}

\begin{thm}\label{thm:U-path}
Assume that assumptions \ref{H0}, \ref{H0.5}, \ref{H1}, \ref{H2b}, \ref{H3} and conditions
\begin{align}
&\beta+\delta<\frac12\min\{\alpha,1-\alpha\},\label{cond:1}\\
&\theta>\max\left\{\frac23,\frac{2\delta+(2\alpha-1)\vee0}{2\delta+\alpha}
\right\},\label{cond:2}\\
&\phantom{a}\notag\\
&\begin{cases}\label{cond:3}
\vspace{5pt}
\displaystyle\sum_{k=1}^{ \infty}\mu_k^{2\beta+1-\alpha(1+3\theta)}<\infty, &{\rm if }\;\; \alpha\in \left(0,\frac12\right],\\[1mm]
\displaystyle\sum_{k=1}^{ \infty}\mu_k^{2\beta+1-2\alpha+(1-\alpha)(1-3\theta)}<\infty, &{\rm if }\;\; \alpha\in \left[\frac12, \frac34-\frac\delta2\right], \\[1mm]
\displaystyle\sum_{k=1}^\infty\mu_k^{2\beta-1+2\delta(1-\theta)+\alpha(1-\theta)}<\infty, &{\rm if }\;\; \alpha\in\left[\frac34-\frac\delta2,1\right),
\end{cases}
\end{align}
are satisfied. Then, there exists a constant $\mathfrak L>0$, depending only on the structural data $\alpha,\beta,\delta,\theta,T,C$, such that, for every weak mild solutions $(X,W)$ and $(Y,W)$ to \eqref{abstract_formulation} with initial data $x,y\in H$, respectively, in the sense of Definition \ref{def:mild-sol}, defined on the same probability space, it holds that 
\begin{equation}\label{eq:stima-lip}
\sup_{t\in [0,T]}\mathbb{E}\left[\norm{X(t)-Y(t)}^2\right]\leq \mathfrak{L}\norm{x-y}_H^2.
 \end{equation}
In particular, pathwise uniqueness holds for \eqref{abstract_formulation}.
\end{thm}
\begin{proof}
Let $x,y\in H$. Suppose that there exist two weak mild solutions $(X,W)$ and $(Y,W)$ to \eqref{abstract_formulation} defined on the same probability space $(\Omega,\mathcal{F},\{\mathcal{F}_t\}_{t\in [0,T]},\mathbb{P})$ with initial datum $x\in H$ and $y\in H$, respectively. By Theorem \ref{Thm:Ito-Tanaka}, we deduce
\begin{equation}\label{diff}
X_n(t) - Y_n(t) = \sum_{j=0}^7I_j(t),\qquad t\in [0,T],
\end{equation}
where
\begin{align*}
I_0&:=e^{tA_n}[P_n(x-y) + u_n(0,P_nx)-u_n(0,P_ny)]\,,\\
I_1(t)&:=u_{n}(t,Y_n(t))- u_{n}(t,X_n(t))\,,\\
I_2(t)&:= -({\mathfrak c}+1)A_n\int_0^t  e^{(t-s)A_n}
[u_{n}(s,X_n(s))-u_{n}(s,Y_n(s))]ds\,,\\
I_3(t)&:=\int_0^t e^{(t-s)A_n} Du_n(s,X_n(s))\mathcal{V}_n[C_n(s,X(s))-C_n(s,X_n(s))] ds\,,\\
I_4(t)&:=-\int_0^t e^{(t-s)A_n} Du_n(s,Y_n(s))\mathcal{V}_n[C_n(s,Y_n(s))-C_n(s,Y(s))] ds\,,\\
I_5(t)&:=\int_0^te^{(t-s)A_n}\mathcal{V}_n[C_n(s,X(s))-C_n(s,X_n(s))]ds\,,\\
I_6(t)&:=-\int_0^te^{(t-s)A_n}\mathcal{V}_n[C_n(s,Y_n(s))-C_n(s,Y(s))]ds \,,\\
I_7(t)&:=\int_0^te^{(t-s)A_n}
(Du_{n}(s,X_n(s))-Du_{n}(s,Y_n(s)))[\mathcal{G}_ndW(s)],
\end{align*}
where ${\mathfrak{c}}>{\mathfrak{c}}_0$ and ${\mathfrak{c}}_0$ is given by Theorem \ref{Thm: soluzione-kolmogorov}.
We estimate now the left-hand side of \eqref{diff}
in the space $C([0,T]; L^2(\Omega; H))$
by analyzing all the terms $I_0,\ldots,I_7$ separately. Let $0<T_0\leq T$. We will later specify the values of $T_0$ and ${\mathfrak{c}}$.\\
 We recall that
\[
\norm{X}^2_{C([0,T]; L^2(\Omega; H))}:=\sup_{t\in [0,T]}\mathbb{E}\left[\norm{X(t)}_H^2\right].
\]
For simplicity, we set
\begin{align*}
D_{0}:=\sup_{t\in [0,T]}\norm{e^{tA}}_{\mathcal{L}(H)}.
\end{align*}
and
\begin{align}\label{eq:S1}
& S_1:=\sum_{k=1}^{ \infty}\mu_k^{2\min\{\alpha;1-\alpha\}(\gamma_3(1-\theta)-1)+2\beta+\min\{0;1-2\alpha\}}, \\
& S_2:=\sum_{k=1}^{ \infty}\mu_k^{2\min\{\alpha;1-\alpha\}(\gamma_3(1-\theta)-1)+2\beta+\min\{0;1-2\alpha\}+\max\{0;1-2\alpha\}}, \label{eq:S2}\\
& S_3:=\sum_{k=1}^{ \infty}\mu_k^{2\min\{\alpha;1-\alpha\}(\gamma_1\theta+\gamma_2(1-\theta)-\frac32)+2\beta+\min\{0;1-2\alpha\}}, \label{eq:S3}
\end{align}
where $\gamma_1$, $\gamma_2$, and $\gamma_3$ are defined by \eqref{G1} (or \eqref{G1v}),  \eqref{G2} (or \eqref{G2v}) and  \eqref{G4} (or \eqref{G4v}), respectively. \\
We underline that, by \eqref{cond:1}, \eqref{cond:2} and \eqref{cond:3}, we obtain \begin{align}\label{SS2}
S_1,S_3\leq S_2<\infty.
\end{align}
Clearly, $S_1<S_2$ if $\alpha\in\left(0,\frac12\right)$ and $S_1=S_2$ if $\alpha\in\left[\frac12,1\right)$. If $\alpha\in\left(0,\frac12\right]$ the we have
\begin{align*}
\gamma_1=\frac\beta\alpha, \qquad \gamma_2=\frac12+\frac{\beta+\delta}\alpha,   \qquad \gamma_3=\frac12+\max\left\{\frac\delta\alpha;1\right\}.    
\end{align*}
Hence, for every $\theta\in(0,1)$,
\begin{align*}
\gamma_1\theta+\gamma_2(1-\theta)-\frac32-(\gamma_3(1-\theta)-1)
\leq  & \theta\frac\beta\alpha+(1-\theta)\left(\frac{\beta+\delta}\alpha-\frac\delta\alpha\right)-\frac12 
= \frac\beta\alpha-\frac12<0,
\end{align*}
since $\beta<\frac12\alpha$ by \eqref{cond:1}. If $\alpha\in\left[\frac12,1\right)$, then (see also Remark \ref{rmk:conseguenze_varie})
\begin{align*}
\gamma_1=\frac\beta\alpha, \qquad \gamma_2=\frac12+\frac{\beta+\delta}{1-\alpha}, \qquad \gamma_3=\frac 12+\max\left\{1;\frac{\alpha+\delta-\frac12}{1-\alpha}\right\}.
\end{align*}
This implies that, for every $\theta\in(0,1)$,
\begin{align*}
\gamma_1\theta+\gamma_2(1-\theta)-\frac32-(\gamma_3(1-\theta)-1)
\leq  & \theta\frac\beta\alpha+(1-\theta)\left(\frac{\beta+\delta}{1-\alpha}-\frac{\delta+\alpha-\frac12}{1-\alpha}\right)-\frac12  \\
= & \theta\frac\beta\alpha+(1-\theta)\frac{\beta+\frac12-\alpha}{1-\alpha}-\frac12
\leq  \frac\beta\alpha-\frac12<0,    
\end{align*}
recalling that $\frac\beta\alpha\geq \frac{\beta+\frac12-\alpha}{1-\alpha}$ for every $\alpha\in\left[\frac12,1\right)$ and $\beta<\frac12(1-\alpha)$.

For the term $I_0$, we have 
\begin{align*}
I_0= & e^{tA_n}\bigg(P_n(x-y)+
\int_0^1Du_n(0,P_ny+rP_n(x-y))[P_n(x-y)]dr\bigg)\\
= & e^{tA_n}\bigg[P_n(x-y)+\sum_{k=1}^n\int_0^1\scal{\nabla_xu^+_{n,k}(0,P_ny+rP_n(x-y))}{P_n(x-y)}_H\Phi^+_kdr\\
&+\sum_{k=1}^n\int_0^1\scal{\nabla_xu^-_{n,k}(0,P_ny+rP_n(x-y))}{P_n(x-y)}_H\Phi^-_k dr\bigg],
\end{align*}
so that, by Remark \ref{norm-Pn}, \eqref{EF2} and \eqref{stima-b-r}, with $q=\gamma_3(1-\theta)-1\leq -\frac12$ from Remark \ref{rmk:condizioni_gamma}, we infer that 
\begin{align}
\label{I0}
\|I_0\|^2_{C([0,T_0];L^2(\Omega;H))}\leq 2D_0^2c^2\left[1+C_0^2\mathfrak C_2^2({\mathfrak c})S_1\right]\|x-y\|_H^2,
\end{align}
where $c=\sup_{n}\|P_n\|_{\mathscr L(H)}$ and by \eqref{SS2}, $S_1<\infty$.
Analogously, for the term $I_1$ we have
\begin{align*}
I_1(t)&=\int_0^1Du_n(t,X_n(t)+r(Y_n(t)-X_n(t)))[Y_n(t)-X_n(t)]dr\\
&=\sum_{k=1}^n\bigg[\int_0^1\scal{\nabla_xu^+_{n,k}(t,X_n(t)+r(Y_n(t)-X_n(t)))}{ 
Y_n(t)-X_n(t)}_H\Phi^+_k dr\\
&\qquad\quad+\int_0^1\scal{\nabla_xu^-_{n,k}(t,X_n(t)+r(Y_n(t)-X_n(t)))}{ 
Y_n(t)-X_n(t)}_H\Phi^-_k dr\bigg]
\end{align*}
so that, by using \eqref{EF2} and \eqref{stima-b-r} (with $q=\gamma_3(1-\theta)-1$), we infer that 
\begin{align}\label{i1}
\|I_1\|^2_{C([0,T_0]; L^2(\Omega; H))}
\leq 2C_0^2\mathfrak C_2^2({\mathfrak{c}})S_1
\|X_n-Y_n\|_{C([0,T_0]; L^2(\Omega; H))}^2,
\end{align}
where $S_1$ is given by \eqref{eq:S1}.

 For the term $I_2$, let $t\in[0,T_0]$. By \eqref{FourierA}, \eqref{FourierETA}, \eqref{def_Un} and the H\"older inequality we obtain
\begin{align*}
&\|I_2(t)\|^2_{L^2(\Omega;H)}\\
&\leq2({\mathfrak c}+1)^2\mathbb{E}
\bigg[\sum_{k=1}^{n}\bigg(\abs{\lambda_k^+}^2\left(\int_0^te^{(t-s){\rm Re}(\lambda_k^+)}
\abs{u^+_{n,k}(s,X_n(s))-u^+_{n,k}(s,Y_n(s))}ds\right)^2\\
&\qquad\qquad+\abs{\lambda_k^-}^2\left(\int_0^te^{(t-s){\rm Re}(\lambda_k^-)}
\abs{u^-_{n,k}(s,X_n(s))-u^-_{n,k}(s,Y_n(s))}ds\right)^2\bigg)\bigg]\\
&\leq2({\mathfrak c}+1)^2\mathbb{E}
\bigg[\sum_{k=1}^{n}\bigg(\abs{\lambda_k^+}^2\int_0^te^{(t-s){\rm Re}(\lambda_k^+)}ds
\int_0^te^{(t-s){\rm Re}(\lambda_k^+)}
\abs{u^+_{n,k}(s,X_n(s))-u^+_{n,k}(s,Y_n(s))}^2ds\\
&\qquad\qquad+\abs{\lambda_k^-}^2\int_0^te^{(t-s){\rm Re}(\lambda_k^-)}ds
\int_0^te^{(t-s){\rm Re}(\lambda_k^-)}
\abs{u^-_{n,k}(s,X_n(s))-u^-_{n,k}(s,Y_n(s))}^2ds\bigg)\bigg].
\end{align*}

By the mean value theorem, \eqref{EF2} and \eqref{stima-b-r}, we have
\begin{align*}
&\|I_2(t)\|^2_{L^2(\Omega;H)}\leq
2({\mathfrak c}+1)^2\mathfrak C_2^2({\mathfrak c})C_0^2
\|X_n-Y_n\|^2_{C([0,T_0];L^2(\Omega;H))}\\
&\times\sum_{k=1}^{n}
\mu_k^{2\min\{\alpha;1-\alpha\}(\gamma_3(1-\theta)-1)+2\beta+\min\{0;1-2\alpha\}}
\left[\abs{\lambda_k^+}^2\left(\frac{1-e^{t{\rm Re}(\lambda_k^+)}}{-{\rm Re}(\lambda_k^+)}\right)^{2}
+\abs{\lambda_k^-}^2\left(\frac{1-e^{t{\rm Re}(\lambda_k^-)}}{-{\rm Re}(\lambda_k^-)}\right)^{2}\right].
\end{align*}
Since $1-e^{-x}\leq\min\{x,1\}$ for every $x\geq0$, for every $k\in\N$, $i\in\{-,+\}$ and 
$t\in[0,T_0]$ we have
\begin{equation}\label{stima-Theta}
0\leq \frac{1-e^{t{\rm Re}(\lambda_k^{(i)})}}{-{\rm Re}(\lambda_k^{(i)})}
\leq\min\left\{T_0,\frac{1}{-{\rm Re}(\lambda_k^{(i)})}\right\}.
\end{equation}
Let now $h\in\N$ be arbitrary. By \eqref{stima-Theta}, we infer that
\begin{align}\label{i2}
\|I_2\|^2_{C([0,T_0]; L^2(\Omega; H))}
&\leq 2({\mathfrak c}+1)^2\mathfrak C_2^2({\mathfrak c})C_0^2
\Big[T_0^2\Sigma_{1,h}+\Sigma_{h,\infty}\Big]
\|X_n-Y_n\|_{C([0,T_0]; L^2(\Omega; H))}^2,
\end{align}
where
\begin{align}
&\Sigma_{1,h}:=\sum_{k=1}^{h}
\mu_k^{2\min\{\alpha;1-\alpha\}(\gamma_3(1-\theta)-1)+2\beta+\min\{0;1-2\alpha\}}
\left(\abs{\lambda_k^+}^2+\abs{\lambda_k^-}^2\right),\label{eq:sigmaK}\\
&\Sigma_{h,\infty}:=\sum_{k=h+1}^{\infty}
\mu_k^{2\min\{\alpha;1-\alpha\}(\gamma_3(1-\theta)-1)+2\beta+\min\{0;1-2\alpha\}}
\left(\frac{\abs{\lambda_k^+}^2}{{\rm Re}(\lambda_k^+)^2}
+\frac{\abs{\lambda_k^-}^2}{{\rm Re}(\lambda_k^-)^2}\right).\label{eq:thetaK}
\end{align}
We stress that, by \eqref{damped_stime_coefficienti_avl}, \eqref{damped_stime_coefficienti_avl_2} and \eqref{SS2}, we have
\begin{equation}\label{eq:thetaK-zero}
\lim_{h\to\infty}\Sigma_{h,\infty}\leq\lim_{h\to\infty}\sum_{k=h+1}^{ \infty}\mu_k^{2\min\{\alpha, 1-\alpha\}(\gamma_3(1-\theta)-1)+2\beta+\min\{0;1-2\alpha\}+\max\{0,1-2\alpha\}}=0.
\end{equation}

For the term $I_3$, by assumption \ref{H2b} and \eqref{EF1}, for every $t\in[0,T_0]$ we get
\begin{align*}
&\|I_3(t)\|^2_{L^2(\Omega; H)}\\
&\leq 2t\sum_{k=1}^{n}\mathbb{E}
\bigg[\int_0^t e^{2(t-s){\rm Re}(\lambda_k^+)}\scal{\mathcal{V}_n^*
\nabla_xu^+_{n,k}(s,X_n(s))}{C_n(s,X(s))-C_n(s,X_n(s))}_U^2ds\\
&\qquad\qquad +\int_0^te^{2(t-s){\rm Re}(\lambda_k^-)}\scal{\mathcal{V}_n^*
\nabla_xu^-_{n,k}(s,X_n(s))}{C_n(s,X(s))-C_n(s,X_n(s))}_U^2ds\bigg]\\
&\leq T_0C_0^4\mathfrak{C}_1^2({\mathfrak{c}})\left[\sum_{k=1}^{n}\sum_{i\in\{-,+\}}\left(\mu_k^{\beta}\big|b_k^{(i)}\big|\left(-{\rm Re}\big(\lambda_k^{(i)}\big)\right)^{\gamma_1\theta+\gamma_2(1-\theta)-\frac32}\right)^2\right]\sup_{t\in [0,T_0]}\mathbb{E}\left[\norm{X_n(t)-X(t)}^{2\theta}_H\right]\\
&\leq T_0C_0^4\mathfrak{C}_1^2({\mathfrak{c}})\left[\sum_{k=1}^{n}\mu_k^{2\min\{\alpha; 1-\alpha\}(\gamma_1\theta+\gamma_2(1-\theta)-1)+2\beta+\min\{0;1-2\alpha\}-\min\{\alpha; 1-\alpha\}}\right]\sup_{t\in [0,T_0]}\mathbb{E}\left[\norm{X_n(t)-X(t)}^{2\theta}_H\right],
\end{align*}
where the last inequality follows by \eqref{stima-b-r} with $q=\gamma_1\theta+\gamma_2(1-\theta)-\frac32$, which satisfies $q\leq-\frac12$ by \eqref{cond:1}. Hence
\begin{align}\label{i3}
&\|I_3\|^2_{C([0,T_0]; L^2(\Omega; H))}\leq  T_0C_0^4\mathfrak{C}_1^2({\mathfrak{c}})S_3\sup_{t\in [0,T_0]}\mathbb{E}\left[\norm{X_n(t)-X(t)}^{2\theta}_H\right],
\end{align}
where $S_3$ has been defined in \eqref{eq:S3} and, by \eqref{SS2}, $S_3<\infty$.
In the same way, for $I_4$ we get
\begin{align}\label{i4}
&\|I_4\|^2_{C([0,T_0]; L^2(\Omega; H))}\leq T_0C_0^4\mathfrak{C}_1^2({\mathfrak{c}})S_3\sup_{t\in [0,T_0]}\mathbb{E}\left[\norm{Y_n(t)-Y(t)}^{2\theta}_H\right].
\end{align}

By \eqref{SuperStima1} (recalling that by assumption \ref{H3} $\gamma_1<1$) and the H\"older inequality we get
\begin{align*}
\mathbb{E}\norm{I_5(t)}_H^2
&\leq \mathbb{E}\left[\left(\int_0^t\norm{e^{(t-s)A_n}
\mathcal{V}_n[C_n(s,X(s))-C_n(s,X_n(s))]}_Hds\right)^2\right]\\
&\leq M_0^2C_0^2\mathbb{E}\left[\left(\int_0^t\frac{1}{(t-s)^{\frac{\gamma_1}{2}}}\frac{1}{(t-s)^{\frac{\gamma_1}{2}}}\norm{X(s)-X_n(s)}^\theta_Hds\right)^2\right]\\
&\leq M_0^2C_0^2\mathbb{E}\left[\left(\int_0^t\frac{1}{(t-s)^{\gamma_1}}ds\right)\left(\int_0^t\frac{1}{(t-s)^{\gamma_1}}\norm{X(s)-X_n(s)}^{2\theta}_H ds\right)\right]\\
&\leq \frac{M_0^2C_0^2 T_0^{1-\gamma_1}}{1-\gamma_1}
\int_0^t\frac{1}{(t-s)^{\gamma_1}}\mathbb{E}
\norm{X(s)-X_n(s)}^{2\theta}_Hds\\
&\leq \frac{M_0^2C_0^2 T_0^{2-2\gamma_1}}{(1-\gamma_1)^2}
\sup_{t\in [0,T_0]}\mathbb{E}\norm{X(t)-X_n(t)}^{2\theta}_H,
\end{align*}
so that
\begin{align}\label{i5}
\|I_5\|^2_{C([0,T_0]; L^2(\Omega; H))}
\leq \frac{M_0^2C_0^2T_0^{2-2\gamma_1}}{(1-\gamma_1)^2}
\sup_{t\in [0,T_0]}\mathbb{E}\norm{X(t)-X_n(t)}^{2\theta}_H.
\end{align}
Analogously, we obtain for $I_6$ that 
\begin{align}\label{i6}
\|I_6\|^2_{C([0,T_0]; L^2(\Omega; H))}
\leq \frac{M_0^2C_0^2T_0^{2-2\gamma_1}}{(1-\gamma_1)^2}
\sup_{t\in [0,T_0]}\mathbb{E}\norm{Y(t)-Y_n(t)}^{2\theta}_H.
\end{align}

Eventually, we estimate $I_7$. Let $\{\widetilde e_k\}_{k\in\N}$ be the orthonormal basis of $U$ introduced in \eqref{Def:base-ON}. By the It\^o isometry, for every $t\in[0,T_0]$ we get
\begin{align*}
\|I_7(t)\|^2_{L^2(\Omega; H)}&=\int_0^t\mathbb{E}\left[\sum_{\ell=1}^{ \infty}\norm{e^{(t-s)A_n}\left[D_xu_n(s,X_n(s))-D_xu_n(s,Y_n(s))\right]\mathcal{G}_n\widetilde e_\ell}^2_H\right]ds\\
&\leq 2\int_0^t\mathbb{E}\bigg[\sum_{\ell=1}^{ \infty}\sum_{k=1}^n\bigg(e^{2(t-s){\rm Re}(\lambda_k^+)}\scal{\mathcal{G}_n^*\nabla_xu^+_{n,k}(s,X_n(s))-\mathcal{G}_n^*\nabla_xu^+_{n,k}(s,Y_n(s))}{\widetilde e_\ell}^2_U\\
&\qquad\quad\qquad+e^{2(t-s){\rm Re}(\lambda_k^-)}\scal{\mathcal{G}_n^*\nabla_xu^-_{n,k}(s,X_n(s))-\mathcal{G}_n^*\nabla_xu^-_{n,k}(s,Y_n(s))}{\widetilde e_\ell}^2_U\bigg)\bigg]ds\\
&= 2\int_0^t\mathbb{E}\bigg[\sum_{k=1}^n\bigg(e^{2(t-s){\rm Re}(\lambda_k^+)}\norm{\mathcal{G}_n^*\nabla_xu^+_{n,k}(s,X_n(s))-\mathcal{G}_n^*\nabla_xu^+_{n,k}(s,Y_n(s))}^2_U\\
&\qquad\quad\qquad+e^{2(t-s){\rm Re}(\lambda_k^-)}\norm{\mathcal{G}_n^*\nabla_xu^-_{n,k}(s,X_n(s))-\mathcal{G}_n^*\nabla_xu^-_{n,k}(s,Y_n(s))}^2_U\bigg)\bigg]ds\\
&\leq C_0^2\mathfrak{C}_3^2({\mathfrak{c}})\left[\sum_{k=1}^{n}\sum_{i\in\{-,+\}}\left(\mu_k^{\beta}\big|b_k^{(i)}\big|\left(-{\rm Re}\big(\lambda_k^{(i)}\big)\right)^{\gamma_3(1-\theta)-1}\right)^2\right]\|X_n-Y_n\|_{C([0,T_0]; L^2(\Omega; H))}^2\\
&\leq C_0^2\mathfrak{C}_3^2({\mathfrak{c}})\left[\sum_{k=1}^{n}\mu_k^{2\min\{\alpha, 1-\alpha\}(\gamma_3(1-\theta)-1)+2\beta+\min\{0;1-2\alpha\}}\right]\|X_n-Y_n\|_{C([0,T_0]; L^2(\Omega; H))}^2,
\end{align*}
where the last inequality follows by \eqref{stima-b-r} with $q=\gamma_3(1-\theta)-1$, which satisfies $q\leq-\frac12$ by \eqref{cond:2}. Hence
\begin{align}\label{i7}
&\|I_7\|^2_{C([0,T_0]; L^2(\Omega; H))}\leq C_0^2\mathfrak{C}_3^2({\mathfrak{c}})S_1\|X_n-Y_n\|_{C([0,T_0]; L^2(\Omega; H))}^2,
\end{align}
where $S_1$ is given by \eqref{eq:S1}.
Combining \eqref{I0}, \eqref{i1}, \eqref{i2}, \eqref{i3}, \eqref{i4}, \eqref{i5}, 
\eqref{i6} and \eqref{i7}, and recalling that $X_n-Y_n=\sum_{j=0}^7I_j$, for every $n,h\in\N$ we 
obtain
\begin{align*}
&\|X_n-Y_n\|_{C([0,T_0]; L^2(\Omega; H))}^2\leq 8\sum_{j=0}^{7}
\|I_j\|^2_{C([0,T_0]; L^2(\Omega; H))}\\
&\leq 16D_0^2c^2
\big[1+C_0^2\mathfrak C_2^2({\mathfrak c})S_1\big]\|x-y\|_H^2
+16C_0^2\Big[\mathfrak C_2^2({\mathfrak c})
+\mathfrak C_3^2({\mathfrak c})\Big]S_1\,
\|X_n-Y_n\|_{C([0,T_0]; L^2(\Omega; H))}^2\\
&\quad+16({\mathfrak c}+1)^2C_0^2\mathfrak C_2^2({\mathfrak c})
\Big[T_0^2\Sigma_{1,h}+\Sigma_{h,\infty}\Big]
\|X_n-Y_n\|_{C([0,T_0]; L^2(\Omega; H))}^2\\
&\quad+8\left[T_0C_0^4\mathfrak C_1^2({\mathfrak c})S_3
+\frac{M_0^2C_0^2T_0^{2-2\gamma_1}}{(1-\gamma_1)^2}\right]
\sup_{t\in [0,T_0]}\left(\mathbb{E}\norm{X_n(t)-X(t)}^{2\theta}_H
+\mathbb{E}\norm{Y_n(t)-Y(t)}^{2\theta}_H\right).
\end{align*}
Since $\theta\in(0,1)$, by Proposition \ref{prop:Xn} and the Jensen inequality, letting $n$ tend to 
infinity we obtain
\begin{align}\label{eq:pre-con}
&\|X-Y\|_{C([0,T_0]; L^2(\Omega; H))}^2
\leq 16D_0^2c^2
\big[1+C_0^2\mathfrak C_2^2({\mathfrak c})S_1\big]\|x-y\|_H^2\notag\\
&\qquad+16C_0^2\Big[\big(\mathfrak C_2^2({\mathfrak c})
+\mathfrak C_3^2({\mathfrak c})\big)S_1
+({\mathfrak c}+1)^2\mathfrak C_2^2({\mathfrak c})
\big(T_0^2\Sigma_{1,h}+\Sigma_{h,\infty}\big)\Big]
\|X-Y\|_{C([0,T_0]; L^2(\Omega; H))}^2 .
\end{align}
By \eqref{Lim-cost} we choose 
${\mathfrak c}>{\mathfrak c}_0$ large enough such that
\begin{align}\label{scelta-c}
16C_0^2\Big[\mathfrak C_2^2({\mathfrak c})
+\mathfrak C_3^2({\mathfrak c})\Big]S_1&\leq\frac14.
\end{align}
Since
${\mathfrak c}$ is fixed, by \eqref{eq:thetaK-zero} we choose $h\in\N$ 
large enough such that
\begin{equation}\label{scelta-K}
16({\mathfrak c}+1)^2C_0^2\mathfrak C_2^2({\mathfrak c})\Sigma_{h,\infty}
\leq\frac18 .
\end{equation}
Finally, since ${\mathfrak c}$ and $h$ are fixed, we choose $T_0\in(0,T]$ small enough such that
\begin{equation}\label{scelta-T0}
16({\mathfrak c}+1)^2C_0^2\mathfrak C_2^2({\mathfrak c})\,
\Sigma_{1,h}T_0^2\leq\frac18.
\end{equation}
Hence, by \eqref{eq:pre-con}, \eqref{scelta-c}, \eqref{scelta-K} and \eqref{scelta-T0}, recalling 
that $C_0^2\mathfrak C_2^2({\mathfrak c})S_1\leq 1$ by \eqref{scelta-c}, we deduce
\[
\|X-Y\|_{C([0,T_0]; L^2(\Omega; H))}^2\leq 
32D_0^2c^2\|x-y\|_H^2+\frac12\|X-Y\|_{C([0,T_0]; L^2(\Omega; H))}^2 ,
\]
and so that
\begin{equation}\label{stima-T0}
\|X-Y\|_{C([0,T_0]; L^2(\Omega; H))}^2\leq 64\,D_0^2c^2\,\|x-y\|_H^2.
\end{equation}
Finally, by \eqref{stima-T0} and standard arguments we obtain the statement for an arbitrary $T>0$.
\end{proof}

In the next theorem, we will remove the boundedness assumption \ref{H2b} by exploiting the localization method presented in \cite[Section 4.2]{BOS}.

\begin{thm}
\label{thm:path_uniq}
Let $T>0$. Assume that assumptions \ref{H0}, \ref{H0.5}, \ref{H1}, \ref{H2}, \ref{H3} and conditions
\begin{align*}
&\beta+\delta<\frac12\min\{\alpha,1-\alpha\},\\
&\theta>\max\left\{\frac23,\frac{2\delta+(2\alpha-1)\vee0}{2\delta+\alpha}
\right\},\\
&\phantom{a}\notag\\
&\begin{cases}
\vspace{5pt}
\displaystyle\sum_{k=1}^{ \infty}\mu_k^{2\beta+1-\alpha(1+3\theta)}<\infty, &{\rm if }\;\; \alpha\in \left(0,\frac12\right],\\[1mm]
\displaystyle\sum_{k=1}^{ \infty}\mu_k^{2\beta+1-2\alpha+(1-\alpha)(1-3\theta)}<\infty, &{\rm if }\;\; \alpha\in \left[\frac12, \frac34-\frac\delta2\right], \\[1mm]
\displaystyle\sum_{k=1}^\infty\mu_k^{2\beta-1+2\delta(1-\theta)+\alpha(1-\theta)}<\infty, &{\rm if }\;\; \alpha\in\left[\frac34-\frac\delta2,1\right),
\end{cases}
\end{align*}
are satisfied. Then pathwise uniqueness for \eqref{abstract_formulation} holds.
\end{thm}
\begin{proof}
Let $T>0$ and $x\in H$. Suppose that there exist two weak mild solutions $(X,W)$ and $(Y,W)$ to \eqref{abstract_formulation}  defined on the same probability space $(\Omega,\mathcal{F},\{\mathcal{F}_t\}_{t\in [0,T]},\mathbb{P})$ with initial datum $x$.
 For every $N\in\N$ we define the stopping time
\[
\tau_N:=T\wedge \inf\{t\in [0,T]\, :\, \norm{X(t)}_H>N\}\wedge \inf\{t\in [0,T]\, :\, \norm{Y(t)}_H>N\}.
\]
It is not difficult to show that $\tau_N\nearrow T$ $\mathbb{P}$-almost surely as $N\rightarrow\infty$. Let $\Pi_N$ be the orthogonal projection on the closed ball of $H$ with center in $0$ and radius $N$ and let $C^T_N:[0, \infty)\times H\rightarrow U$ be defined by 
\[
 C^T_N(t,x):=C(t\wedge T,\Pi_Nx)\,,\quad t\geq 0,\, x\in H.
\]
Since $C$ verifies \ref{H2} then $C^T_N$ verifies \ref{H2b}, for every $N\in\N$. Let now $N\in\mathbb N$ be fixed. By the localisation-extension arguments in \cite[Section 4.2]{BOS} (with $\alpha=0$), there exist two weak mild solutions
$(\widetilde{X}_N,W)$ and $(\widetilde{Y}_N,W)$ to \eqref{abstract_formulation}  with $C$ replaced by $C^T_N$ on $[0, \infty)$, such that
\begin{equation*}
X=\widetilde{X}_N\,,\quad  
Y=\widetilde{Y}_N\,,\qquad\text{in } [\![0,\tau_N]\!].
\end{equation*}
Hence, by Theorem \ref{thm:U-path}  we infer that 
\begin{align*}
\mathbb P(\widetilde X_N(t)=\widetilde Y_N(t),
\quad\forall\,t\in[0,T])=1\,,
\end{align*}
so that 
\[
  X=Y \quad\text{in } [\![0,\tau_N]\!] \quad\forall\,N\in\mathbb  N\,.
\]
By setting $\Omega_N:=\{\omega\in \Omega: X(t)=Y(t),\;\,\forall\,t\in[0,\tau_N(\omega)]\}$, we have then that 
$\mathbb P(\Omega_N)=1$ for every $N\in\mathbb N$.
Since $\Omega_N\supseteq \Omega_{N+1}$ for every $N\in\N$,  we get 
\[
\Omega_N\downarrow \{\omega\in\Omega\, :\, X(\omega,t)=Y(\omega,t),\quad \forall\, t\in [0,T]\}
\]
and we conclude by the continuity of $\mathbb{P}$.
\end{proof}

By Theorem \ref{thm:path_uniq} (assuming $\mu_n=n^{2k/d}$ with $k\in\{1,2\}$ and $d\in\{1,2,3\}$) we can deduce explicit conditions on $\alpha,\beta,\delta$ and $\theta$ which guarantee pathwise uniqueness for \eqref{concrete_damped_equation}.

\begin{coro}\label{cor:path}
Pathwise uniqueness holds for \eqref{concrete_damped_equation} in the following cases, which for reader's convenience we split according to the different intervals the parameter $\alpha$ belongs to. 

\vspace{1mm}
\begin{itemize}
\item Case $\displaystyle \alpha\in\left(\frac{d}{8k}+\frac14,\frac12\right]: \ 
\begin{cases}
&d<2k\\
&\theta\in \left(\max\left\{\frac23; \frac{d}{6k\alpha}+\frac{1-\alpha}{3\alpha}\right\}, 1\right),\\
&\beta\in \left[0,\min\left\{\frac\alpha2; \alpha-\frac d{4k}; \frac{\alpha(1+3\theta)}{2}-\frac{d}{4k}-\frac12\right\}\right),\\
&\delta\in \left(\frac d{4k}-\frac\alpha2,\frac\alpha2-\beta\right)\cap [0,\infty),
\end{cases}$\\[1mm]
\item Case $\displaystyle\alpha\in \left(\frac12, \frac34\right): \ 
\begin{cases}
&d<2k\\
&\theta\in \left(\max\left\{\frac23; \frac{2-3\alpha+\frac d{2k}}{3(1-\alpha)}\right\},1\right), \\
&\beta\in \left[0,\min\left\{\frac{1-\alpha}2; \alpha-\frac12+\frac{(1-\alpha)(3\theta-1)}2-\frac d{4k}\right\}\right), \\
&\delta\in \left(\frac d{4k}-\frac\alpha 2,\min\left\{\frac{1-\alpha}2-\beta; \frac32-2\alpha; \frac{1-\alpha(2-\theta)}{2(1-\theta)}\right\}\right)\cap [0,\infty),
\end{cases}$ \\[1mm]
\item Case $\displaystyle \alpha\in \left[\frac34, 1\right):
\ \begin{cases}
&d<2k\\
&\delta\in \left(\frac d{4k}-\frac{\alpha}2,\frac{1-\alpha}2\right)\cap [0,\infty), \\
&\theta\in \left(\max\left\{\frac23; \frac{2\delta+2\alpha-1}{2\delta+\alpha}; \frac{\alpha+2\delta-1+\frac d{2k}}{\alpha+2\delta}\right\},1\right), \\
&\beta\in \left[0,\min\left\{\frac{1-\alpha}2-\delta; \frac{1-\alpha(1-\theta)}2-\delta(1-\theta)-\frac{d}{4k}\right\}\right).
\end{cases}$
\end{itemize}

\end{coro}

\appendix
\section{A criterion for weak uniqueness}
\label{app:mart}
In this appendix, we slightly modify the results in \cite[Appendix A.2]{Pri2015} to cover the case in which the nonlinear part of the drift depends on time. Before stating the lemma we recall that, given $t_0\geq0$ and $x\in H$, we say that 
$(\Omega,\mathcal F,(\mathcal F_t)_{t\geq t_0},\mathbb P,W,X)$ is a weak mild solution to 
\eqref{abstract_formulation} on $[t_0,\infty)$ with initial datum $x$ if 
$(\Omega,\mathcal F,(\mathcal F_t)_{t\geq t_0},\mathbb P)$ is a complete filtered probability 
space, $X=\{X(t)\}_{t\geq t_0}$ is a continuous $(\mathcal F_t)_{t\geq t_0}$-adapted $H$-valued 
process and $W=\{W(t)\}_{t\geq t_0}$ is a $U$-valued cylindrical Wiener process with 
$W(t_0)=0$, such that, for every $t\geq t_0$,
\begin{align}\label{mild_solution_t0}
X(t)=e^{(t-t_0)A}x+\int_{t_0}^te^{(t-s)A}\mathcal VC(s,X(s))ds
+\int_{t_0}^t e^{(t-s)A}\mathcal GdW(s), \qquad \mathbb P\textrm{-a.s.}
\end{align}
Definition \ref{def:mild-sol} corresponds to the case $t_0=0$.

\begin{lemma}\label{lem:mart}
Assume \ref{H0}, \ref{H0.5}, \ref{H1}, \ref{H2b} and \ref{H3}, and suppose that, for every 
$t_0\geq0$ and $x\in H$, any two weak mild solutions to \eqref{abstract_formulation} on 
$[t_0,\infty)$ with initial datum $x$ have the same one-dimensional marginal laws on $H$. 
Then weak uniqueness holds for \eqref{abstract_formulation}.
\end{lemma}

\begin{proof}
For every $k,n\in\N$ and $\mathcal{O}\subseteq\R^n$ we denote by $C^k_c(\mathcal{O})$ the space of $C^k$ functions from $\mathcal{O}$ to $\R$ with compact support. We set $E:=[0,\infty)\times H$.
For every $k\in\N$ we set
\[
g_{2k-1}:=\begin{pmatrix}\widetilde e_k\\ 0\end{pmatrix},\qquad
g_{2k}:=\begin{pmatrix}0\\ \widetilde e_k\end{pmatrix},
\]
where $\{\widetilde e_k\}_{k\in\N}$ is the orthonormal basis of $U$ introduced in 
\eqref{Def:base-ON}, so that $\{ g_j\}_{j\in\N}$ is an orthonormal basis of $H$. We 
introduce the following space of cylindrical functions
\[
\mathcal D:=\Big\{\varphi(y)=f\big(\scal{y}{g_{1}}_H,\ldots,
\scal{y}{g_{2m}}_H\big):\ y\in H, \ m\in\N,\ f\in C^2_c(\R^{2m})\Big\},
\]
and we set
\[
\widetilde{\mathcal D}:={\rm span}\Big\{\psi(t,y)=\eta(t)\varphi(y):\ t\in[0,\infty), \ y\in H, \ \ 
\eta\in C^1_c([0,\infty)),\ \varphi\in\mathcal D\Big\}.
\]
We stress that, for every $m\in\N$, 
$F_m:={\rm span}\{g_1,\ldots,g_{2m}\}$ is contained in $D(A^*)\cap D(\mathcal V^*)$ and is 
invariant for $A^*$. In particular $F:=\bigcup_{m\in\N}F_m$ is dense in $H$ and invariant for $\{e^{tA^*}\}_{t\geq0}$, 
hence it is a core for $A^*$.
Moreover, for every $t \geq 0$ and $\varphi \in \mathcal D$, we consider the following time-dependent second-order differential operator
\[
L_t\varphi(y):=\frac12{\rm Tr}\big[\mathcal G\mathcal G^*D^2\varphi(y)\big]
+\scal{y}{A^*\nabla\varphi(y)}_H+\scal{\mathcal V^*\nabla\varphi(y)}{C(t,y)}_U,
\]
and for every $\psi\in \widetilde{\mathcal D}$, we consider the following time-independent differential operator
\[
\widetilde L\psi:=\frac{\partial\psi}{\partial t}+L_t\big(\psi(t,\cdot)\big).
\]
First of all we study the relation between the solutions to the martingale problem associated to $\widetilde L$ and the weak mild solutions to \eqref{abstract_formulation}. We say that $\{(t,Z(t))\}_{t\geq 0}$ is a solution to the martingale problem for  $(\widetilde L,\delta_{(t_0,x)})$ if, for every $\psi\in \widetilde{\mathcal{D}}$, the process $\{M(t)\}_{t\geq t_0}$ given by 
\[
M(t):=\psi(t,Z(t))-\int_{t_0}^t\widetilde L\psi(s,Z(s))ds
\]
is a martingale with respect to the natural filtration of $\{(t,Z(t))\}_{t\geq 0}$. Let $t_0\geq0$, $x\in H$, and let $\{Z(t)\}_{t\geq t_0}$ be a weak mild solution on $[t_0,\infty)$ with $Z(t_0)=x$. 
For every $\varphi\in\mathcal D$ the It\^o formula gives 
$d\varphi(Z(t))=L_t\varphi(Z(t))dt+\scal{\mathcal G^*\nabla\varphi(Z(t))}{dW(t)}_U$, so that the process $\{M(t)\}_{t\geq t_0}$ defined by
\[
M(t):=\varphi(Z(t))-\int_{t_0}^{t}L_s\varphi(Z(s))ds
\]
is a martingale. For 
$\eta\in C^1_c([0,\infty))$, replacing $\varphi(Z(s))=M_s+\int_{t_0}^{s}L_r\varphi(Z(r))dr$ 
and using the Fubini theorem, we obtain
\[
(\eta\varphi)(t,Z(t))-\int_{t_0}^{t}\widetilde L(\eta\varphi)(s,Z(s))ds
=\eta(t)M(t)-\int_{t_0}^{t}\eta'(s)M_s\,ds, \qquad t\geq 0,
\]
whose right-hand side is a martingale, since $\mathbb E[M_s|\mathcal F_u]=M_u$ for $s\geq u$ 
and $\int_u^t\eta'=\eta(t)-\eta(u)$. Hence, for every initial datum $(t_0,x)$, the process $\{\widetilde Z(t)\}_{t\geq t_0}:=\{(t,Z(t))\}_{t\geq t_0}$ solves the 
martingale problem for $(\widetilde L,\delta_{(t_0,x)})$, namely, for every $\psi\in \widetilde{\mathcal D}$, the process $\{\widetilde M(t)\}_{t\geq t_0}$ defined by
\[
\widetilde M(t)=\psi(\widetilde Z(t))-\int_{t_0}^{t}\widetilde L\psi(\widetilde Z(s))ds
\]
is a martingale. Conversely, in the same way as in \cite[Theorem 3.6]{Kun2013}, up to a localization argument which allows us to replace the test functions of $\widetilde{\mathcal D}$ by the ones used in \cite{Kun2013},  it is possible to prove that every solution to the martingale problem for  $(\widetilde L,\delta_{(t_0,x)})$ is of the type $\{\widetilde Z(t)\}_{t\geq t_0}=\{(t,Z(t))\}_{t\geq t_0}$ where $\{Z(t)\}_{t\geq t_0}$ is a weak mild solution to \eqref{abstract_formulation} on $[t_0,\infty)$ with initial datum $x$.
\\
Due to the previous consideration, to conclude the proof it is sufficient to prove that, for every initial datum $(t_0, x)\in E$,  given any two solutions $\{\widetilde Z^1(t)\}_{t\geq t_0}$ and $\{\widetilde Z^2(t)\}_{t\geq t_0}$ to the martingale problem for $(\widetilde L,\delta_{(t_0,x)})$, if the laws of $\widetilde Z^1(t)$ and $\widetilde Z^2(t)$ coincide, for every $t\geq t_0$ then the processes $\{\widetilde Z^1(t)\}_{t\geq t_0}$ and $\{\widetilde Z^2(t)\}_{t\geq t_0}$ have the same law.\\
The strategy is to apply \cite[Theorem 18]{Pri2015} to the operator $\widetilde L:\widetilde{\mathcal D}\subset C_b(E)\to B_b(E)$. However, to do so, we need to check that $\widetilde L:\widetilde{\mathcal D}\subset C_b(E)\to B_b(E)$ satisfies \cite[Hypothesis 17]{Pri2015}.\\
For every $\varphi\in\mathcal D$ and $y\in H$, the vectors $\nabla\varphi(y)$ and 
$D^2\varphi(y)$ belong to $D(A^*)\cap D(\mathcal{V}^*)$ and they are finite linear combinations of the $g_k$'s and of their tensor 
products. Moreover, also $\scal{y}{A^*\nabla\varphi(y)}_H$ depends on $y$ only through the same finitely many coordinates as $\varphi$ and, by definition of $\mathcal{D}$, it is bounded. Taking into account also condition \ref{H2b}, we obtain $L_t\varphi\in B_b(H)$ for every $t\geq0$, and hence $\widetilde L\psi\in B_b(E)$ for every $\psi\in\widetilde{\mathcal D}$. As for \cite[Hypothesis 17]{Pri2015}, one argues as in \cite[Remark 8]{Pri2015}, choosing countable subsets of $C^1_c([0,\infty))$ and of $C^2_c(\R^{2m})$ which are dense with respect to the $C^1-$ and the $C^2-$ norm, respectively, and taking the union over $m\in\N$.\\
Finally we can apply \cite[Theorem 18]{Pri2015} to $\widetilde L$. Let $(t_0, x)\in E$ and let $\{\widetilde Z^1(t)\}_{t\geq t_0}$ and $\{\widetilde Z^2(t)\}_{t\geq t_0}$ be any two solutions to the martingale problem for $(\widetilde L,\delta_{(t_0,x)})$. By 
the results above, for every $i\in\{1,2\}$,  $\{\widetilde Z^i(t)\}_{t\geq t_0}=\{(t,Z^i(t))\}_{t\geq t_0}$ where $\{Z^i(t)\}_{t\geq t_0}$ is a weak mild solution to \eqref{abstract_formulation} on $[t_0,\infty)$ with initial datum $x$, so that the law of $\widetilde Z^i(t)$ on $E$ is 
$\delta_t\otimes\mu^i_t$, where $\mu^i_t$ denotes the law of $Z^i(t)$ on $H$. By hypothesis $\mu^1_t=\mu^2_t$, for every $t\geq t_0$, and so that the laws of $\widetilde Z^1(t)$ and $\widetilde Z^2(t)$ coincide, for every $t\geq t_0$. Hence by \cite[Theorem 18]{Pri2015}
 the processes $\{\widetilde Z^1(t)\}_{t\geq t_0}$ and $\{\widetilde Z^2(t)\}_{t\geq t_0}$ have the same law, and, by the equivalence proved above, weak uniqueness holds for \eqref{abstract_formulation}.

\end{proof}

\section*{Acknowledgments} 
The authors D. A. and D. A. B. are members of GNAMPA (Gruppo Nazionale per l'Analisi Matematica,
la Probabilit\`a e le loro Applicazioni) of the Italian Istituto Nazionale di Alta Matematica (INdAM).\\
The authors have no relevant financial or non-financial interests to disclose.

\section*{Declarations} 

\noindent\textbf{Declaration of generative AI use.}
The authors used a generative AI assistant to check the English
grammar, to detect misprints, and to verify the consistency of the
notation and the structure of the paper. The authors reviewed and
edited the output and take full responsibility for the content of
the article.

\bibliographystyle{siam}
\bibliography{biblio.bib}

\end{document}